\documentclass[a4paper,leqno,11pt]{amsart}
\usepackage{etex}
\usepackage[top=1.5in,bottom=1.3in,left=1.3in,right=1.3in,marginparwidth=1.5cm]{geometry}

\usepackage{times}
\usepackage{tensor}

\usepackage[new]{old-arrows}

\usepackage{comment}
\usepackage{amsmath, amssymb, amsxtra, 
amsfonts, latexsym, mdwlist, amsthm}
\usepackage{mathrsfs}
\usepackage{subfig}
\usepackage{graphicx}
\usepackage{longtable}
\usepackage{enumerate}

\usepackage{mathtools} 

\usepackage[shortlabels]{enumitem}
\usepackage{caption}

\usepackage[mathscr]{eucal}
\usepackage{mathrsfs}
\usepackage{upgreek}
\usepackage{wasysym}
\usepackage[all,cmtip,dvipdfmx]{xy}
\usepackage{xcolor}
\definecolor{rouge}{rgb}{0.85,0.1,.4}
\definecolor{bleu}{rgb}{0.1,0.2,0.9}
\definecolor{violet}{rgb}{0.7,0,0.8}

\usepackage{enumitem}   
\usepackage{adjustbox} 

\usepackage[english]{babel}

\usepackage[driverfallback=dvipdfmx,
colorlinks=true, linkcolor=bleu, urlcolor=violet, citecolor=rouge]{hyperref}

\usepackage[nameinlink]{cleveref}
\usepackage{thmtools}
\usepackage{lscape}
\usepackage{afterpage}
\usepackage{capt-of}

\newcommand{\nc}{\newcommand}
\nc{\We}[2]{\mathbb{V}^{#1}_{#2}}
\nc{\Si}[2]{\mathbb{L}^{#1}_{#2}}
\nc{\mi}{\varphi}
\nc{\ppart}{(\!(t)\!)}
\nc{\al}{\alpha}
\nc{\ka}{\kappa}
\nc{\LP}{{}^L\neg P}
\nc{\n}{{\mathfrak n}}
\nc{\ghat}{\wh{\g}}
\def\neg{\negthinspace}

\newcommand{\wh}{\widehat}

\newcommand{\mc}{\mathcal}
\newcommand{\mf}{\mathfrak}
\newcommand{\mb}{\mathbb}
\newcommand{\g}{\mf{g}}
\renewcommand{\aa}{\mf{a}}
\newcommand{\ga}{\mf{g}^\natural}
\newcommand{\h}{\mf{h}}

\newcommand{\affg}{\widehat{\mf{g}}}

\newcommand{\Z}{\mathbb{Z}}
\newcommand{\C}{\mathbb{C}}

\newcommand{\Q}{\mathbb{Q}}
\newcommand{\W}{\mathscr{W}}

\newcommand{\lam}{\lambda}

\def\leq{\leqslant}
\def\geq{\geqslant}

\DeclareMathOperator{\Hom}{Hom}
\DeclareMathOperator{\Com}{Com}

\def\commentPM#1{{\color{magenta}Pierluigi:#1}}

\theoremstyle{theorem}
\newtheorem{theorem}{Theorem}[section]
\newtheorem{Pro}[theorem]{Proposition}

\newtheorem{Lem}[theorem]{Lemma}
\newtheorem{lemma}[theorem]{Lemma}
\newtheorem{criterion}[theorem]{Criterion}
\newtheorem{Co}[theorem]{Corollary}

\theoremstyle{remark}

\newtheorem{Rem}[theorem]{Remark}
\newtheorem{Conj}{Conjecture}

\newtheorem{Ex}[theorem]{Example}

\newcommand{\bs}{\boldsymbol}

\title{$\W$-algebras as conformal extensions of 
affine VOAs}

\author[Adamovi\'c]{Dra{\v z}en~Adamovi\'c} 
\address[Dra{\v z}en~Adamovi\'c]{Department of Mathematics, Faculty of science, University of Zagreb, Bijeni\v cka 30, 10 000 Zagreb, Croatia}
\email{adamovic@math.hr}

\author[Arakawa]{Tomoyuki Arakawa}
\address[Tomoyuki Arakawa]{Okinawa Institute of
Science and Technology (OIST),
1919-1 Tancha, Onna-son, Kunigami-gun, Okinawa, 904-0495, JAPAN}
\email{tomoyuki.arakawa@oist.jp}

\author[Creutzig]{Thomas Creutzig}
\address[Thomas Creutzig]{Department Mathematik, Algebra und Geometrie,
Friedrich-Alexander Universit\"at Erlangen}
\email{thomas.creutzig@fau.de}

\author[Linshaw]{Andrew R. Linshaw}
\address[Andrew R. Linshaw]{Department of Mathematics, 
University of Denver, 
2390 S.~York St., Denver, CO 80208 USA}
\email{andrew.linshaw@du.edu}

\author[Moreau]{Anne Moreau}
\address[Anne Moreau]{Universit\'{e} Paris-Saclay, CNRS, Laboratoire de Math\'{e}matiques d'Orsay, 
Rue Michel Magat, B\^{a}t. 307, 
91405 Orsay, France}
\email{anne.moreau@universite-paris-saclay.fr} 

\author[M\"oseneder]{Pierluigi M\"oseneder Frajria}
\address[Pierluigi M\"oseneder Frajria]{Politecnico di Milano, Polo regionale di Como,  Via Anzani 42, 22100, Como, Italy}
\email{pierluigi.moseneder@polimi.it}

\author[Papi]{Paolo Papi}
\address[Paolo Papi]{Sapienza Universit\`a di Roma, 
Piazzale Aldo Moro 2, 
00185 Roma, Italy}
\email{papi@mat.uniroma1.it}

\begin{document}
\begin{abstract}
We provide a criterion for a vertex operator superalgebra homomorphism from an affine vertex algebra to another vertex superalgebra to be conformal, and an additional criterion that guarantees that this homomorphism is surjective. 

This situation is applied to $\W$-algebras and $\W$-superalgebras and we list all cases where our criterion applies. This gives many new examples of $\W$-algebras that collapse to affine vertex algebras or are conformal extensions. In particular, we provide many examples of simple $\W$-algebras at non-admissible levels that collapse to admissible level affine vertex algebras. 
\end{abstract}

\maketitle

 \section{Introduction}
 Let $\g=\g_{\bar 0}\oplus\g_{\bar 1}$ be a basic classical Lie superalgebra, i.e., a finite-dimensional simple Lie superalgebra with reductive even part, endowed with an even invariant non-degenerate symmetric bilinear form.
Let $\W^k(\g,f)$ be the $\W$-algebra 
associated with $\g$ 
and with a nilpotent element $f\in \g_{\bar 0}$ at level $k$ \cite{FF90,KacRoaWak03}. 
Its simple quotient is denoted by $\W_k(\g,f)$ and understanding 
these simple quotients is in general an interesting and difficult question. 
Sometimes these simple quotients can be described via conformal embeddings  
of vertex algebras that are better understood. 
Here
an embedding $\iota \colon V \hookrightarrow W$ 
of a vertex operator algebra $(V,\omega^V)$ 
into a vertex operator algebra $(W,\omega^W)$ 
is called {\em conformal} if the 
conformal vector is mapped to the 
conformal vector: $\iota(\omega^V) = \omega^W$. 
The notion is very natural, and was 
studied by part of the authors and their collaborators 
in several works \cite{AP13,AKMPP17,AKMPP18,AKMPP20,AMP23} in the case 
where $V,W$ are affine vertex (super)algebras, 
or when $W$ is an affine $\W$-algebra. We continue these studies.


The notion of collapsing level is specific to $\W$-algebras. 
The level $k$ is called {\em collapsing} 
if the simple quotient $\W_k(\g,f)$ of $\W^k(\g,f)$ 
is isomorphic to its affine vertex algebra $L_{k^\natural}
(\g^\natural)$, where $\g^\natural$ is the centralizer 
of an $\mf{sl}_2$-triple $(e,h,f)$ in~$\g$, 
$L_{k^\natural}
(\g^\natural)$ the simple quotient of the universal affine 
vertex algebra $V^{k^\natural}(\g^\natural)$, 
and $k^\natural$ a level entirely 
determined by the condition that 
$V^{k^\natural}(\g^\natural)$ is a vertex subalgebra of $\W^k(\g,f)$. 
See~\Cref{Sec:conditions} and~\Cref{Sec:W-algebras} 
for more details. 
The study of collapsing levels was initiated in \cite{AKMPP17,AKMPP18} 
where the authors were able to classify all 
such levels for minimal $\W$-algebras (that is, 
$f$ lies in the minimal nilpotent orbit), 
including the supercase. 
Recall that a minimal $\W$-algebra has a minimal strong generating set given by elements in conformal weight $1$ and $\frac{3}{2}$ together with the Virasoro element. 
In this case an explicit $\lambda$-bracket formula for the weight $\frac{3}{2}$ 
can be derived \cite[eq.~(1.1)]{AKMPP17}, using the work of \cite{KW04}. 
This facilitates the studies of conformal embeddings enormously. 
For general $\W$-algebras concrete $\lambda$-brackets or operator product formula are not computable and so one needs to establish other strategies for the studies of conformal embeddings. 

The results for minimal $\W$-algebras have interesting applications in  representation theory because in this special case the quantum Hamiltonian reduction functor is exact and non-zero in the Kazhdan--Lusztig category KL$_k$ for $k \notin {\Z}_{\geq 0}$. 
The classification of collapsing levels in \cite{AKMPP18} was used in \cite{AKMPP20} to prove that the category KL$_k$ is semi-simple when $k$ is collapsing. 
Next development is given in \cite{CY21}, where it was proved that 
KL$_k$ is a braided tensor category  for 
collapsing $k$. 
The rigidity can also be proved by using certain decompositions of conformal embeddings.  
In \cite{ACPV} it was proved that the category KL$_{k_n}(\mf{sl}_{2n})$ 
for $k_n=-n-1/2$ is a semisimple, rigid braided tensor category, 
although the level $k_n$ is not collapsing for minimal affine $\W$-algebras. 
But it was proved in \cite{AMP23} that the level $k_n$ is collapsing for hook type $\W$-algebras. 
These results indicate that it is important to construct and classify 
collapsing levels for non-minimal affine $\W$-algebras.  
So let us conjecture:

\begin{Conj}
Assume that $\g^\natural$ is a reductive Lie algebra and $k$ is a collapsing level for $\W_k(\g, f)$, then KL$_k(\g)$ is a semisimple, braided tensor category.
\end{Conj}

%

Assume that $\g$ is a simple Lie algebra. 
Recall that a level $k$ is said to be {\em admissible} 
\cite{KW88,KW89, KW08}
if 
$$k+h^\vee = \frac{p}{q} \in \Q_{>0} 
\quad \text{with} \quad (p,q)=1 
\quad \text{and} \quad
\begin{cases}p \geq h^\vee &\text{if }(q,r^\vee)=1,\\
p \geq h &\text{if }(q,r^\vee)\not=1,
\end{cases}$$ 
where $h^\vee$ and $h$ are the dual Coxeter number 
and the Coxeter numbers of $\g$, respectively, and 
$r^\vee$ is the lacety of $\g$: $r^\vee=1$ for $\g$ of type $A,D,E$, 
$r^\vee=2$ for $\g$ of type $B,C,F_4$ 
and $r^\vee=3$ for $\g$ of type $G_2$. 

For admissible levels,  a strategy to study 
 collapsing levels for 
 any affine $\W$-algebra  
 was found in
\cite{AEM}.
 A necessary condition for a $\W$-algebra to collapse to an affine vertex algebra is of course that their associated varieties \cite{Ar12} coincide. 
 This condition was used to detect possible collapsing levels 
 in the classical types. 
 A sufficient condition for collapsing is the agreement of 
 the asymptotic data of characters: using this condition  
a list of admissible collapsing levels has been determined 
in \cite{AEM}.  
This list is conjecturally exhaustive, and proven to be exhaustive in the 
 exceptional types.

\subsection{Results}
The aim of the present article is first to summarize and harmonize 
the previously obtained results in a concise way, 
and second to extract specific conditions which allow us 
to detect new examples (for both conformal 
embeddings and collapsing levels at both admissible and non-admissible levels).

Our main criterion is as follows, and we refer to \Cref{Sec:conditions} 
for the more precise assumptions for the validity of the statements.
\begin{criterion}[\Cref{criterion}]
Let $V$ be a conical vertex operator superalgebra. 
Let $\aa$ be a direct sum of a reductive Lie algebra 
and Lie superalgebras of type $\mathfrak{osp}_{1|2n}$.
Assume that 
there exists a vertex algebra homomorphism 
$\varphi \colon V^k(\aa) \to V$ 
satisfying the following conditions: 
\begin{enumerate}[{\rm (A)}]
\item 
$k+h^\vee \not\in \mathbb Q_{\leq 0}$, 
\item 
$\varphi(x)$ is primary for all $x\in\g$ and 
$\varphi(\aa)=V_1$,
\item 
the central charges of $V$ and $V^k(\aa)$ coincide, 
\item 
the weight-two subspace of $\text{\rm Com}(\varphi(V^k(\aa)), V)$ is one-dimensional, 
\item 
we have $h_i - h_{\lambda_i}\neq 0$ 
for all $i$ such that $h_i >1$. 
\end{enumerate}
Our conclusions are the following:  
\begin{enumerate} [{\rm (i)}]
\item 
if all the conditions \ref{crit-1}--\ref{crit-iv} hold, 
then 
$$L \cong L_k(\aa),$$ 
where $L$ is the simple quotient of $V$, 
\item 
if the conditions \ref{crit-1}--\ref{crit-iii} hold, 
then 
$L$ is a 
conformal extension of ${\tilde V}_k(\aa)$, the image 
of $V^k(\aa)$ in $L$ through the composition map 
$\pi \circ \varphi \colon V^k(\aa) \to L$, 
with $\pi \colon V \to L$ the natural projection.
\end{enumerate}
\end{criterion}
We apply this criterion to the case where 
$V = \W^k(\g, f)$, $\aa=\g^\natural$ and $\varphi \colon V^{k^\natural}(\g^\natural) \to \W^k(\g, f)$.  
We find all pairs  $(\W^k(\g, f), V^{k^\natural}(\g^\natural))$ that are covered by this criterion.
For this,  
 one firstly needs to find those $\W$-algebras for which condition \ref{crit-iii} applies,  
 that is the commutant 
$\text{\rm Com}(V^{k^\natural}(\g^\natural), \W^k(\g, f))$ has one-dimensional weight-two subspace: see \Cref{Pro:partitions}, \Cref{confBCD}. 
Once that is done, one has to enumerate the levels for which central charges 
of $\W^k(\g, f)$ and $V^{k^\natural}(\g^\natural)$ coincide (condition \ref{crit-ii}).
Our conclusions for the classical types, that is for $\g=\mf{sl}_{m|n}$ or $\g=\mf{osp}_{m|n}$, 
are summarized in \Cref{cesl} and \Cref{pro:osp1}. Then we proceed to verify whether condition \ref{crit-iv} is verified. The outcome of this analysis
appears in \Cref{Pro:concl_Type-A} and \Cref{Pro:concl_Type-BCD}, and summarized in \Cref{BCD}.  
There are also plenty of examples of conformal embeddings 
or collapsing levels which are not covered by our criterion. 
For the case where $f$ is minimal, we refer to \cite{AKMPP18},  
for the case where $k$ is admissible, we 
refer to \cite{AEM}.


Our conclusions for the exceptional types 
are summarized in \Cref{Tab:Data-E6}--\Cref{Tab:Data-E8-6}. 
The tables are established using either the criterion of the present paper, 
or the results of \cite{AEM} when the level is admissible,  
or other specific methods (for instance, explicit computations 
of OPEs). 
In particular, combining all our techniques, we obtain a complete list 
of collapsing levels for~$\g=G_2$.

\subsection{History, motivation and outlook}
As so often, the first studies of conformal embeddings have been done 
in the physics of conformal field theory. 
The motivation for these studies was the search for new exceptional conformal field theories. 
The best-known example is the Capelli--Itzykson--Zuber $ADE$ classification 
of rational conformal field theories associated to $SU(2)$ \cite{CIZ}. 
For each $A_n, D_n$, $n \in \mathbb Z_{>0}$ and $E_6, E_7, E_8$ they found exactly one rational conformal field theory. 
The case of $A_n$ is just the usual Wess--Zumino--Witten theory of $L_n(\mathfrak{sl}_2)$. 
The cases of $D_n$ are explained by simple currents, the cases of 
$E_6$ and $E_8$ correspond to the conformal embeddings 
$L_{10}(\mathfrak{sl}_2) \hookrightarrow L_1(\mathfrak{sp}_4)$ and $L_{28}(\mathfrak{sl}_2) \hookrightarrow L_1(G_2)$, 
while the case of $E_7$ is of a different nature. 

The modern interest of physics in vertex operator algebra is thanks to their correspondences 
with higher-dimensional quantum field theories \cite{B+}. 
One such family of quantum field theories are Argyres--Douglas theories \cite{AD}, 
that are labelled by pairs  of Dynkin diagrams. 
The associated vertex algebras seem to be possibly trivial conformal extension 
of two different affine vertex algebras or $\W$-algebras. 
In particular, many conformal embedding conjectures appeared in this context. 
Typical examples come from boundary admissible levels, 
\cite{SonXieYan17,XieYan,XieYanYau,C17}. 
Results of \cite{AEM, AMP23, ACGY} confirm several such conjectures. 
For low rank 4D $\mc{N}=2$ SFCTs, 
a list of expected isomorphisms 
between VOAs coming from a single theory  have been 
exhibited in \cite[Tables 11--14]{LiXieYan}. 
Our results confirm some of their predictions. 
For example, we obtain the following isomorphisms (at non-admissible 
levels for the $\W$-algebras), 
where we use the Bala--Carter classification for the nilpotent orbits:
$$\W_{-12}(E_7,A_2+3 A_1) 
\cong 
\W_{-12}(E_7,2A_2)
\cong 
\W_{-6}(F_4,\tilde A_2) 
\cong L_{-2}(G_2).
$$

Another famous family of collapsing $\W$-algebras is 
the one associated with the Cvitanovi\'{c}--Deligne 
exceptional series, that is $\g=A_1, A_2, G_2, D_4, F_4, E_6, E_7, E_8$,  
and a minimal nilpotent element $f_\theta$. 
For such a $\g$, one has 
that $\W_{-\frac{h^\vee}{6}-1}(\g, f_\theta) = \mathbb C$, i.e., 
the $\W$-algebra becomes trivial.  
By  \cite{AM}, the above cases and  $k=-\frac{1}{2}$ when $\g=\mathfrak{sp}_{2n}$ are the only  ones for which  $\W_{k}(\g, f_\theta)=\mathbb C$  if $\g$ is a Lie algebra. These results have been extended to the super case in \cite{AKMPP18}. 
Note that these levels are non-admissible for types $D_4, E_6, E_7, E_8$ so 
one gets new strongly rational (lisse and rational) $\W$-algebras at non-admissible levels. 
This becomes somehow even more interesting by noting that if we shift the level by $+1$, 
then $\W_{-\frac{h^\vee}{6}}(\g, f_\theta)$ is a conformal extension of 
$L_{k^\natural}(\g^\natural) \otimes \W_{-\frac{1}{3}}(\mathfrak{sl}_2)$ \cite{Kaw}. 
In particular, for  types $D_4, E_6, E_7, E_8$ one again gets new strongly rational 
$\W$-algebras at non-admissible levels. 
If we shift the level again by plus one, then similar phenomena seem to appear, 
see \cite{ACK24}. 
%

An example occurs for   $\W_{-k}(\mathfrak{so}_{2n}, f_\theta)$ 
for $k=-1, -2$ and $n \in \mathbb Z_{\geq 4}$. 
Indeed, 
it is proved in \cite{AKMPP18} that
$\W_{-2}(\mathfrak{so}_{2n}, f_\theta) \cong L_{n-4}(\mathfrak{sl}_2)$ 
while one of the main results of \cite{CFKLN} is that $\W_{-1}(\mathfrak{so}_{2n}, f_\theta) \cong \left(L_{n-3}(\mathfrak{osp}_{1|2}) \otimes \mathcal F^{2n-4}\right)^{\mathbb Z/2\mathbb Z}$, 
with $\mathcal F^{2n-4}$ the vertex superalgebra of free fermions of rank $2n-4$. 
These two results prove a conjecture of two of us on strong rationality \cite{AM}. 
Also, in  \cite{AKMPP20}, the maximal ideal of 
$V^{-2} (\mathfrak{so}_{2n})$ was described using the  isomorphism $\W_{-2}(\mathfrak{so}_{2n}, f_\theta) \cong L_{n-4}(\mathfrak{sl}_2)$.
Here we reprove this result: see  the case of $q=2, \ell =2$, $ k = k'_4$ of \Cref{pro:osp1} 
and~\Cref{Pro:concl_Type-BCD}. 

The general guiding principle is that if a $\W$-algebra $\W_k(\g, f)$ 
at a non-admissible level $k$ collapses to an affine vertex algebra at an admissible level, 
then, if we shift the level by plus one and the shifted level is still non-admissible, 
 it is worthwhile to try to understand if $\W_{k+1}(\g, f)$ is a conformal extension of 
$L_{(k+1)^\natural}(\g^\natural) \otimes W$ with $W$ a principal $\W$-algebra 
at a non-degenerate admissible level. 
We expect to find plenty of new and interesting results in this way. 

In general, conformal embeddings give us a hint on representation theory. 
Not much is known on the representation theory of simple $\W$-algebras 
(except when they are strongly rational \cite{Ar15,AvE23,Mc21}). 
If the $\W_k(\g, f)$-algebra collapses to an affine vertex algebra 
and if the representation theory of the latter is known, 
then we  understand the representation theory of $\W_k(\g, f)$ as well. 
Besides category $\mc{O}$ at admissible levels, there are presently only good results for the much larger category of weight modules of $L_k(\mathfrak{sl}_2)$ 
at any admissible level.
This category is neither finite nor semi-simple, but it has enough injectives and projectives \cite{ACK25}, 
it is a vertex tensor category \cite{C1}, it is ribbon \cite{CMY2, NOCW}, 
and its fusion rules are computed \cite{C2, NOCW}. 
A major goal is to make substantial progress on these representation theoretical properties 
for the categories of weight modules of simple $\W$-algebras and so far conformal extensions 
have been a major aid on this, 
e.g., all these properties have been established for the category of weight modules of 
$\W_{-n + \frac{n}{n+1}}(\mathfrak{sl}_n, f_{\text{sub}})$ \cite{CMY3}, 
where $f_{\text{sub}}$ denotes a subregular nilpotent element.

Finally, collapsing levels can be useful in determining, or at least predicting, associated varieties of $\W$-algebras at non-admissible levels. 
First, we find  many new quasi-lisse $\W$-algebras. 
Recall here that a vertex algebra is called quasi-lisse 
\cite{AK}
if its associated variety has finitely many symplectic 
leaves (it is lisse if it is just a point). 
An affine vertex algebra  
and its Drinfeld--Sokolov reduction 
are quasi-lisse if and only if 
the associated variety is contained in the nilpotent cone. 
Hence, if a $\W$-algebra collapses to an admissible 
affine vertex algebra, which is quasi-lisse by 
\cite{Ara09b}, then it is quasi-lisse as well. 
For instance, we obtain that 
$$\W_{-12+9/4}(E_6,D_4) \cong L_{-3+3/2}(A_2),$$
so $\W_{-12+9/4}(E_6,D_4)$ yields a new 
non admissible quasi-lisse $\W$-algebra. 

Second, from such an isomorphism, that is, 
$$\W_{k}(\g,f) \cong L_{k^\natural}(\g^\natural),$$
one can sometimes 
predict the variety of $L_{k}(\g)$ 
using the fact that the 
associated variety of the 
Drinfeld--Sokolov reduction corresponding to $f$ of $L_k(\g)$ 
is the intersection of the associated variety of $L_k(\g)$ 
with the nilpotent Slodowy slice of $f$ (\cite{Ara09b}). 
Here, for instance, we expect 
that the associated variety of $L_{-12+9/4}(E_6)$ 
is the closure of the  nilpotent orbit labelled $D_5(a_1)$ 
(see \Cref{Conj:E6:D5a1}).  
We refer to \Cref{Sec:examples_exceptional} for other examples 
and details about the strategy. 
The recent works of Shan--Yan--Zhao \cite{ShanYanZhao26}
provides a conjectural uniform description for the 
associated varieties of the simple affine vertex algebra 
$L_{k}(\g)$ attached to a 
simple Lie algebra $\g$ 
of simply-laced type and any rational level $k$ greater than the critical level. 
The conjectures presented in \Cref{Sec:examples_exceptional} 
are consistent with their predictions.





The key point for the detection of conformal embeddings is the structure 
of the vertex operator algebra 
$V=\mbox{Com}(V^{k^{\natural}} (\g^{\natural}), \W^k(\g, f))$. 
Our main criterion works when the weight-two space $V_2$ of $V$ 
is one-dimensional  (cf.~\Cref{Pro:paolo}).
In \Cref{Griess} we notice that $V_2$ 
can be equipped with the structure of a Griess algebra.  
We obtain this result in full generality,  for any $\W$-algebra. 
It is known that, when the Griess algebra is semisimple and associative, it generates a vertex subalgebra which is isomorphic to the tensor product of certain copies of the Virasoro vertex algebras. 
This implies the following: 
\begin{itemize}
\item assume that the Griess algebra $V_2$ is semisimple and associative,  then $\W^k(\g, f)$ has a vertex subalgebra isomorphic to
$$ V^{k^{\natural}} (\g^{\natural}) \otimes V^{Vir} _{c_1} \otimes \cdots \otimes V^{Vir}_{c_t},$$
where $V^{Vir} _c$ denotes the universal Virasoro vertex algebra of central charge $c$, 
\item the embedding 
$L_{k^{\natural}} (\g^{\natural}) \hookrightarrow \W_k(\g, f)$ 
is conformal if and only if $c_i =0$ for $i= 1, \dots, t$.
\end{itemize}
Using this observation, in \Cref{Griess} we find some new cases 
of conformal embedding in the cases where \Cref{Pro:paolo} 
cannot be applied.

\subsection*{Acknowledgements}
We would like to thank Antun Milas,  Ozren Per\v se,  Shigenori Nakatsuka, Wenbin Yan, 
for long time discussion related to conformal embeddings and collapsing levels, and to Ching Hung Lam for discussion related to  the Griess algebras. 
We are 
grateful to the organisers of the conferences  
Representation Theory XVIII and XIX in Dubrovnik where this project 
was initiated. 
We also thank Peng Shan, Wenbin Yan and Qixian Zhao for relevant discussions. 
We finally thank the referee for his/her careful reading of the paper.
%

D.A. is partially supported by the Croatian Science Foundation under
project IP-2022-10-9006 and by the project Implementation of
cutting-edge research and its application as part of the Scientific Center
of Excellence for Quantum and Complex Systems, and Representations of Lie
Algebras, grant no.~PK.1.1.10.0004, co-financed by the European Union
through the European Regional Development Fund, Competitiveness and
Cohesion Programme 2021--2027.

T.A is partially supported by JSPS Kakenhi Grant numbers 26H01997, 25K21659,
21H04993 and 19KK0065.

T.C. is partially supported by DFG Ptojektnummer 551865932.

A.L.  is partially supported by National Science Foundation Grant DMS-2401382 and Simon Foundation Travel Support for Mathematicians MPS-TSM-00007694.

A.M. is partially supported by the European Research 
Council (ERC) under the European Union's Horizon 2020 research innovation programme
(ERC-2020-SyG-854361-HyperK) and the ANR project ANR-24-CE40-3389 (GRAW). 

P.MF and P.P.  are partially supported by the PRIN project 2022S8SSW2 - Algebraic and geometric aspects of Lie theory - CUP B53D2300942 0006, a project cofinanced
by European Union - Next Generation EU fund.

\section{Sufficient conditions for conformal embeddings}
\label{Sec:conditions}

\subsection{Assumptions} 
\label{Sub:assumptions}
Let $V$ be a conical vertex operator superalgebra, 
that is, $V$ is a conformally $\frac{1}{2}\mathbb Z_{\geq 0}$-graded 
vertex algebra,  
$$V = \bigoplus_{n \in \frac{1}{2}\mathbb Z_{\geq 0}} V_n,$$ 
with 
$V_0 = \mathbb C \mathbf 1$ and $\dim V_n < \infty$ 
for all $n  \geq 0$.  Denote the state-field correspondence by  $x\mapsto Y(x,z)=\sum_{n\in\mathbb Z} x_{(n)}z^{-n-1}$.
 We will also use the notation $Y(x,z)=\sum_{n\in\mathbb Z} x_{n}z^{-n-\Delta_x}$ when $x$ has  the conformal weight $\Delta_x$. 
Assume that $V$ is finitely strongly 
generated by fields 
$X^1=Y(x_1, z), \ldots, X^s=Y(x_s, z)$ 
such that $X^1=L(z)$ is the Virasoro field of~$V$, 
where we write 
$$L(z) = \sum\limits_{n\in\Z} L_n z^{-n-2}.$$

Assume that $L_1V_1=\{0\}$. By \cite[Theorem 3.1]{Li} there is an invariant symmetric bilinear  form $(-|-)$ 
on $V$ such that $(\mathbf 1,\mathbf 1)=1$
(more precisely, \cite{Li} deals with vertex operator algebras with integral conformal weight; one can generalize to our setting  as in \cite{KMP22}).  One shows that   $(- |-)$ is 
uniquely determined 
by $(y_{(m)} \mathbf 1 | z_{(n)} \mathbf 1) 
= (\mathbf 1 | \theta(y_{(m)})z_{(n)} \mathbf 1)$ 
together with $(\mathbf 1 | \mathbf 1)=1$, 
where $\theta$ is the anti-homomorphism 
on the algebra of 
modes defined by 
$$\theta(v_n)= e^{-\pi\sqrt{-1}(\frac{1}{2} p(v) + \Delta_v)}(e^{L_1}v)_{-n},$$
where $p(v)\in\{0,1\}$ and the parity of $v$ is $\overline{p(v)}$.
	

\begin{Pro}
\label{Co:trivial_quotient}
If $V$ has a one-dimensional weight-two subspace and has central charge zero, 
then the simple quotient $L$ of $V$ is trivial, 
i.e., $L \cong \mathbb C$.\end{Pro}

\begin{proof} We have $(x_1|V_j)=0$ for $j\ne 2$. Since $V_2=\mathbb Cx_1$, and 
\begin{align*} 
(L_{-2}{\bf 1}| L_{-2}{\bf 1}) & = ({\bf 1} | \theta(L_{-2})L_{-2}{\bf 1}) &\\
& =({\bf 1} | L_2 L + \text{ terms of conformal weights }>0) 
= \frac{c}{2}({\bf 1} | {\bf 1}) = \frac{c}{2}=0,&
\end{align*}
we deduce that  $x_1$ is in both the right and left kernel of the bilinear form 
$(-|-)$, in particular it is in the maximal proper ideal $I$ of~$V$. 
Now $L_0$ descends to an operator acting as zero  on $V/I$, hence $\bigoplus\limits_{n>0} V_n\subseteq I $. Since $V_0=\mathbb C {\bf 1}$, 
it follows that $V/I \cong\mathbb C$.
\end{proof}

When $U$ is a vertex subalgebra of $V$, 
we denote by $\Com(U,V)$ the {\em commutant} 
of $U$ in $V$ defined by 
$$\Com(U,V)=\{v \in V \colon u_{(n)}v =0 \text{ for all } u \in U \text{ and } 
n \geq 0 \}.$$


Let $\aa$ be a Lie superalgebra 
which is either a simple Lie 
algebra or $\aa= \mf{osp}_{1|2n}$ 
equipped with a non-degenerate, super-symmetric, 
invariant, even bilinear form 
normalised such that long roots have
norm two. 
Let $V^k(\aa)$ be the corresponding affine Lie algebra. 
 
Next lemma is proven in \cite[Theorem 4.1]{ACK24}. 
We briefly outline the proof for the convenience of the reader. 
\begin{Lem}
\label{Lem:commutant}
Assume that $V$ is equipped with an action of $V^k(\aa)$, 
that is, there is a vertex superalgebra homomorphism 
$\varphi \colon V^k(\aa) \to V$. 
Assume furthermore that the following conditions hold: 
\begin{enumerate}[{\rm (i)}]
\item 
\label{Lem:commutant_simple-i} 
the form $(-|-)$ is non-degenerate, 
\item 
\label{Lem:commutant_simple-ii}
$\varphi(x)$
is primary of conformal weight one for all 
$x \in \aa$, 
\item 
\label{Lem:commutant_simple-iii}
$k+h^\vee \not\in \mathbb Q_{\leq 0}$. 
\end{enumerate} 
Then $\Com(\varphi(V^k(\aa)), V)=V^{\aa[t]}$ is simple. 
\end{Lem}

In the lemma, assumption \ref{Lem:commutant_simple-ii}, and in all what follows, 
we identify $\aa$ with a subspace of $V^k(\aa)$ 
by $x \mapsto x_{(-1)}{\bf 1}$.


\begin{proof}
By Condition~\ref{Lem:commutant_simple-iii}, note that 
$V^k(\aa)$ is conformal by the Sugawara construction.  
Let $\omega^\aa$ be its conformal vector, 
and define the corresponding Virasoro 
field $L^\aa(z)$ by $L^\aa_n = \omega^\aa_{(n+1)}$. 
By Condition~\ref{Lem:commutant_simple-ii}, 
$\varphi(\omega^\aa) \in V_2$ and 
$L_1  \varphi(\omega^\aa) =0$. 
Hence by \cite[Theorem 3.11.12 and Remark 3.11.13]{LL04}, 
$L_n$ acts as $L_n^\aa$ on $\varphi(V^k(\aa))$ 
and as $\omega'_{(n+1)}$ on $\Com(\varphi(V^k(\aa)), V)$, 
where 
\begin{equation*}\label{omega'}\omega' := \omega - \varphi(\omega^\aa),\end{equation*}
with $\omega$ the conformal vector of $V$.  
Moreover, $\omega'$ belongs 
to $\Com(\varphi(V^k(\aa)), V)$ and is a conformal vector. 
It follows that the form $(-|-)$ restricts to 
an invariant symmetric bilinear form of $\Com(\varphi(V^k(\aa)), V)$. 

Next, because $V$ is conical, $V$ belongs to the Kazhdan--Lusztig category 
KL$_k$ of $\affg$. Moreover, by Condition~\ref{Lem:commutant_simple-iii} 
it follows from \cite[Lemma 2.1]{ACL} that 
$V^k(\aa)$ is projective in KL$_k$ 
and that 
$$\Com(\varphi(V^k(\aa)), V) = V^{\aa[t]} \cong \Hom_{\widehat \g}(V^k(\aa), V)=V_{[0]},$$ 
where 
$V_{[\lambda]}$ denotes the $L_0^\aa$-eigenspace 
relatively to the eigenvalue $\lambda$. 
So the restriction of the form $(-|-)$ to $V_{[0]}=\Com(\varphi(V^k(\aa)), V)$ 
is non-degenerate. 
Indeed, if  $v\in V_{[0]}$
 is orthogonal to $V_{[0]}$, it is orthogonal to the whole $V$ 
using $(L_0^\aa v | w)=  (v | L_0^\aa w)$ for all $w \in V$, 
but $(-|-)$ is non-degenerate by Condition~\ref{Lem:commutant_simple-i}. 
This implies the simplicity of $\Com(\varphi(V^k(\aa)), V)$. 
\end{proof}

\begin{Co}
\label{Co:commutant}
Keep the hypothesis of \Cref{Lem:commutant}. 
Let $L$ be the simple quotient of $V$ 
and denote by $\widetilde V^k(\aa)$ the image of  $V^k(\aa)$ in $L$. 
Assume that the central charges of $V$ and $V^k(\aa)$ coincide 
and that the weight-two subspace of $\Com(\varphi(V^k(\aa)), V)$ 
is one-dimensional. 
Then $\Com(\widetilde V^k(\aa), L)\cong \mathbb C$.
\end{Co}

\begin{proof}
Applying \Cref{Lem:commutant} to $L$, 
we obtain that $\Com(\widetilde V^k(\aa), L)$ is simple.  
Note that Condition \ref{Lem:commutant_simple-i} is satisfied 
because $L$ 
is simple. 
Since $V$ belongs to KL$_k$, the Cartan subalgebra $\h$ of $\aa$ acts semisimply 
on $V$ and so together with condition \ref{Lem:commutant_simple-iii} 
the assumptions of \cite[Theorem 8.1]{CL2} are satisfied, that is 
 $\Com(\widetilde V^k(\aa), L)$ is 
a quotient of $\Com(\varphi(V^k(\aa)), V)$ and, hence, is 
the simple quotient of $\Com(\varphi(V^k(\aa)), V)$.
Hence by \Cref{Co:trivial_quotient}, 
it is trivial. 
\end{proof}

In the setting of \Cref{Co:commutant}, 
we deduce that $L$ 
is a conformal extension 
of a quotient of $V^k(\aa)$. 
Indeed, let us denote by $\psi$ 
the composition map $\pi \circ \varphi$ 
where $\pi \colon V \to L$ is the natural projection. 
Since the weight-two space of $\Com(\widetilde V^k(\aa), L)=\C$ 
is zero, we have $\omega^L- \psi(\omega^\aa) =0$, 
where $\omega^L$ and $\omega^\aa$ 
are the Virasoro vectors of $L$ and $V^k(\aa)$, 
respectively. 
This means that $\psi\colon V^k(\aa) \to L$ is conformal, 
and so $L$ is a conformal extension of the quotient 
of $V^k(\aa)$ by the kernel of $\psi$.

\subsection{Main criterion}
\label{Sub:criterion}
Keep all assumptions on the vertex superalgebra 
$V=\bigoplus_{n \in \frac{1}{2}\Z}V_n$ as in \Cref{Sub:assumptions}. 
We also assume 
that $V$ comes equipped with an action of 
the affine vertex algebra $V^k(\aa)$, 
where $\aa$ is either simple or $\aa=\mathfrak{osp}_{1|2n}$, 
that is, 
there exists a vertex algebra homomorphism 
\begin{equation*} 
\varphi \colon V^k(\aa) \to V,
\end{equation*}
and, moreover that 
\begin{align*}
&k+h^\vee \not\in \mathbb Q_{\leq 0}, \\
& 
\varphi(x) 
\text{ is primary for all } x \in \aa. 
\end{align*}
Let $L$ be the simple quotient 
of $V$, $I$ the kernel of the natural projection $\pi \colon V \to L$, 
and $L_k(\aa)$ the simple quotient of $V^k(\aa)$. 

We aim to state a 
sufficient condition for establishing that $L \cong L_k(\aa)$.

Let $x_1, \dots, x_s$ be the strong generating 
vectors of $V$ 
as in \Cref{Sub:assumptions}. 
Assume that the linear span of these vectors 
is closed under the action of the zero-modes 
 of the weight-one fields in $V^k(\aa)$. 
 In other words, we assume that $\aa \cong V^k(\aa)_1$ 
 acts on this vector space. 
 Since $\aa$ is either a simple Lie algebra or $\aa=\mathfrak{osp}_{1|2n}$, 
 the action of $\aa$ is completely reducible 
 and we choose new generators $y_1, \dots, y_s$, 
such that each $y_i$ is in an irreducible representation of~$\aa$. 
Write $h_i$ the conformal weight of  $y_i$ 
so that 
$$L_0 y_i = h_i y_i.$$  

For each $i$, let $\lambda_i$ be the highest weight 
of the irreducible representation corresponding to $y_i$. 
Denote by $\mathbb{V}^k(\lambda_i)$ the Weyl module of $V^k(\aa)$ 
 whose top level is the highest weight representation $V(\lambda_i)$ 
 of $\aa$ corresponding to $\lambda_i$. 
The conformal weight of the top level $V(\lambda_i)$  is 
\begin{align}
\label{eq:h_lami}
h_{\lambda_i} =\frac{(\lambda_i, \lambda_i + 2\rho)}{2(k+h^\vee)},
\end{align}
where $\rho$ is the half-sum of positive roots of $\aa$. 
Here $(-,-)$ is the invariant non-degenerate 
bilinear form on $\aa\cong \aa^*$ 
such that $(\theta,\theta)=2$ with $\theta$ the highest 
positive root of $\aa$. 


Our main criterion is the following where, as before, 
$\aa$ is viewed as a subspace of $V^k(\aa)$ 
using $x \mapsto x_{(-1)}{\bf 1}$.

\begin{criterion}
\label{criterion} 
Assume that 
there exists a vertex algebra homomorphism 
$\varphi \colon V^k(\aa) \to V$ 
satisfying the following conditions: 
\begin{enumerate}[{\rm (A)}]
\item 
\label{crit-1} 
$k+h^\vee \not\in \mathbb Q_{\leq 0}$, 
\item 
\label{crit-i} 
$\varphi(x)$ is primary for all $x\in\aa$ and 
$\varphi(\aa)=V_1$, 
\item 
\label{crit-ii}
the central charges of $V$ and $V^k(\aa)$ coincide, 
\item 
\label{crit-iii}
the weight-two subspace of $\text{\rm Com}(\varphi(V^k(\aa)), V)$ 
is one-dimensional, 
\item 
\label{crit-iv}
we have $h_i - h_{\lambda_i}\neq 0$ 
for all $i$ such that $h_i >1$. 
\end{enumerate}
Our conclusions are the following:  
\begin{enumerate} [{\rm (i)}]
\item 
\label{criterion-i}
if all the conditions \ref{crit-1}--\ref{crit-iv} hold, 
then 
$$L \cong L_k(\aa),$$ 
where $L$ is the simple quotient of $V$, 
\item 
\label{criterion-ii}
if the conditions \ref{crit-1}--\ref{crit-iii} hold, 
then 
$L$ is a 
conformal extension of $\widetilde V^k(\aa)$, the image 
of $V^k(\aa)$ in $L$ through the composition map 
$\pi \circ \varphi \colon V^k(\aa) \to L$.  
\end{enumerate}
\end{criterion}

\begin{proof}
First, if the conditions \ref{crit-1}--\ref{crit-iii} hold, then 
by \Cref{Co:commutant} 
we get that $\Com(\widetilde V^k(\aa),L)$ $\cong \C$ 
and that the embedding $\widetilde V^k(\aa)\hookrightarrow 
L$ is conformal. 

Next, because $\Com(\widetilde V^k(\aa),L)\cong \C$, 
the conformal vector $\omega'$ of $\Com(\varphi(V^k(\aa)), V)$ is contained 
in a proper ideal of $V$. 
Indeed, by \Cref{Lem:commutant},  $\Com(\widetilde V^k(\aa),L)$ 
is the simple quotient of $\Com(\varphi(V^k(\aa)), V)$,  
but $\pi(\omega')=0$ in $L=V/I$, 
since $\Com(\widetilde V^k(\aa),L)\cong \C$ while $\omega' \in V_2$. 
So if in addition Condition \ref{crit-iv} holds, 
then all $y_i$ such that $h_i>1$  belong
to $I$ using $L'_0 y_i = \omega'_{(1)} y_i =(h_i - h_{\lambda_i})y_i$. 
Using the assumption $\varphi(\aa)=V_1$ of Condition~\ref{crit-i},  
we deduce that 
the composition map 
$\psi= \pi \circ \varphi$,  is surjective. 
So it factorizes through 
$L_k(\aa)$ since $L=V/I$ is simple. 
\end{proof}


\begin{Rem} 
\label{Rem:reductive}
All statements of this section 
hold for a reductive Lie algebra $\aa$ 
with summands either being reductive Lie algebras 
or of type $\mathfrak{osp}_{1|2n}$. 
More precisely, 
let us decompose $\aa$ as 
$\aa  =  \aa_0 \oplus  \aa_1 \oplus \cdots \oplus \aa_s$,  
where $\aa_0$ is the center of $\aa$ and all $\aa_i$ are simple. 
If $k$ is an invariant form on $\aa$, we set 
\begin{equation}
\label{Vkg}
V^{k}(\aa)=V^{k_0}(\aa_0) \otimes V^{k_1}(\aa_1) \otimes \dots 
\otimes V^{k_s}(\aa_s),
\end{equation} 
where $k_0:=k_{|\aa_0\times \aa_0}$ and for $i\geq 1$, 
$k_{|\aa_i\times \aa_i}=k_i (-,-)_i$ 
and  $(-,-)_i$ is the normalised form on $\aa_i$.
Then the condition $k+h^\vee \not\in\Q_{\leq 0}$ means that 
$k_j + h_j^\vee \not\in\Q_{\leq 0}$
for all $j=1,\ldots,s,$ 
where $h_j^\vee$ is the dual Coxeter number of $\aa_j$. 
\newline\indent 
From now on we  assume that the bilinear form $k_0$ is non-degenerate, unless otherwise specified.
In particular, $V^k(\aa)$ is conformal and 
the non-degeneracy of $k_0$ ensures that the  
Heisenberg vertex algebra $V^{k_0}(\aa_0)$ is simple. 
Conditions \ref{crit-i}--\ref{crit-iii} remain unchanged, 
and Condition \ref{crit-iv} means that it holds 
for all simple summands $\aa_j$. 
\end{Rem}

\begin{Pro}
\label{prop:ext}  
Assume that the conditions \ref{crit-1}--\ref{crit-iii} hold. 
Then
\begin{enumerate}[{\rm (i)}]
\item  
\label{prop:ext-i}
If $\aa = \mathbb C$,
then $L$ is either a rank one Heisenberg vertex algebra $M(1)$, or a lattice vertex algebra $V_{\sqrt{p}\mathbb Z}$ for some $p \in 2 \mathbb Z$ 
if $L$ is a vertex algebra, 
and   
$p \in \mathbb Z$ if $L$ is a vertex superalgebra.
\item 
\label{prop:ext-ii} If $\aa = \aa^\perp \oplus \mathbb C$ with $\aa^\perp$ semisimple 
and $k$ is admissible for $\aa^\perp$, 
then one of the two situations happens: 
\begin{itemize}
\item there is a (possibly trivial) conformal extension $W$ of $L_k(\aa^\perp)$, 
such that $L = W \otimes M(1)$, 
\item $L$ is a (possibly trivial) conformal extension of $L_k(\aa^\perp) \otimes V_{\sqrt{p}\mathbb Z}$ 
with $p \in 2\mathbb Z$ if $L$ is a vertex algebra, and $p \in \mathbb Z$ if $L$ is a vertex superalgebra.

In particular if we already know that $L$ is quasi-lisse, 
then $L$ must be a (possibly trivial) conformal extension of 
$L_k(\aa^\perp) \otimes V_{\sqrt{p}\mathbb Z}$. 
\label{criterion-B} 
\end{itemize}

\item 
\label{prop:ext-iii}
If $\aa$ is semisimple and $k$ admissible for $\aa$, 
then $L$ is a (possibly trivial) conformal extension of $L_k(\aa)$. 
\end{enumerate}
\end{Pro}

\begin{proof}
\ref{prop:ext-i}
The first statement follows from Proposition 2.15 of \cite{CGR}:
that Proposition classifies all commutative algebras satisfying certain conditions 
in the category of $\mathbb C$-graded vector spaces with a non-trivial quadratic form. 
Namely, any such commutative algebra is either trivial or corresponds to a lattice. 
The category of $\mathbb C$-graded vector spaces is braided tensor equivalent 
to the category of Fock modules of $M(1)$. 
Commutative algebras in the direct sum completion of a vertex tensor category are equivalent to vertex algebra extensions $V$ of $M(1)$ \cite{HKL, CMY1}. 
The non-degeneracy condition of \cite{CGR} is precisely the condition that the vertex algebra 
extension $V$ is simple as a module over itself and the grading condition is that 
$\text{Hom}(V, M_\lambda)$ is either one or zero-dimensional for every Fock module $M_\lambda$. 
Hence, \cite[Proposition 2.15]{CGR} translates to the following vertex algebra statement. 
Let $V$ be a simple vertex algebra that is a conformal extension of $M(1)$ in the category of Fock modules, such that $V$ has a strong generating set $S= \{ X_1, \dots, X_n\}$ 
of Heisenberg weights 
$\lambda_1, \dots, \lambda_n$ and conformal weights $h_1, \dots, h_n$ 
that are pairwise distinct: $(\lambda_i, h_i) \neq (\lambda_j, h_j)$ for $i \neq j$.  
Under this condition either $V \cong M(1)$, or $V \cong V_{\sqrt{2p}\mathbb Z}$ 
for a positive integer $p$. 

\ref{prop:ext-ii} More generally, assume that $\aa = \aa^\perp \oplus \mathbb C$, $\aa^\perp$ is semisimple and $k$ is admissible for $\aa^\perp$. 

The category of ordinary modules of $L_{k}(\aa^\perp)$ is a vertex tensor category \cite{CHY} 
and so the framework of the theory of vertex algebra extensions applies
 \cite{CKM}. 
 In particular $\text{Com}(L_k(\aa^\perp), L)$ is either $M(1)$ or a lattice vertex algebra 
 \cite[Theorem 3.5]{CKLR}.  
 If it were just $M(1)$, then 
 $L = W \otimes M(1)$ for some conformal extension $W$ of $L_k(\aa^\perp)$. 
 Otherwise  $L$ must be a (possibly trivial) conformal extension 
 of $L_k(\aa^\perp) \otimes V_{\sqrt{p}\mathbb Z}$ with $p \in 2\mathbb Z$ 
 if $L$ is a vertex algebra, and $p \in \mathbb Z$ if $L$ is a vertex superalgebra.
 
In particular, if we already know that $L$ is quasi-lisse, then $L$ must be a (possibly trivial) conformal extension of $L_k(\aa^\perp) \otimes V_{\sqrt{p}\mathbb Z}$ as the 
category of ordinary modules of a quasi-lisse vertex algebra is finite \cite{AK}.

\ref{prop:ext-iii} 
According to \Cref{Co:commutant} 
and the arguments that follow it, 
we deduce that the morphism $\psi = \pi \circ \varphi 
\colon V^k(\aa) \to L$ is conformal. 
Since $L$ is simple, 
the bilinear form $(-|-)$ 
induces a non-degenerate bilinear invariant form on $L$,  
equivalently, $L$ is self-dual. 
Therefore, it follows from the proof of  \cite[Theorem 3.5]{AEM} 
that $\psi$ factorizes through $L_k(\aa)$. 
\end{proof}


We remark that Theorem~3.5 of \cite{AEM} 
is stated only for vertex algebras, 
but the exact same argument also applies for vertex superalgebras.

\section{Conformal embeddings in $\W$-algebras}
\label{Sec:W-algebras}




As in the Introduction, let $\g$ be a basic classical Lie superalgebra.
Let $(e,h,f)$ be 
an $\mathfrak{sl}_2$-triple in $\g_{\bar 0}$, and 
$\W^k(\g,f)$ the corresponding $\W$-algebra 
\cite{KacRoaWak03}. 
We denote by $c(k,f)$ the central charge of $\W^k(\g,f)$.  
Let $\ga$ be the centralizer in $\g$ of $(e,h,f)$. 
We assume that  $\g^\natural$ is as in \Cref{Rem:reductive}, 
i.e., $\g^\natural$ is   a direct sum of a reductive Lie algebra
and copies of $\mf{osp}(1|2n)$.

The Lie superalgebra $\g$ decomposes into irreducible 
$\mathfrak{sl}_2$-modules under the action of the triple $(e,h,f)$. 
Note that the isotypic component of the trivial representation is $\ga$. 
Let 
\[
\g = \bigoplus_{i \in I} \rho_i
\]
be the decomposition into irreducible modules for $\ga \oplus \mathfrak{sl}_2$, 
indexed by the set $I$.
We denote by $\text{mult}(\rho)$ the multiplicity of a representation 
$\rho$ in $\g$. 
Let $\mathbb C$ and adj be the trivial representation 
of $\g^\natural$ and the adjoint representation of $\mathfrak{sl}_2$, 
respectively. 

\begin{Rem} In the case 
where $\g$ is a reductive Lie algebra, 
$\g^\natural$ is also a reductive Lie algebra. \end{Rem}


Let $k^\natural$ 
be the level of $\ga$ 
uniquely determined 
by the condition that the embedding 
$$V^{k^\natural}(\ga) \longrightarrow \W^k(\g, f)$$
is a vertex algebra homomorphism \cite{KW04}. 
We say that the level $k^\natural$ is {\em non-critical} 
if we can decompose 
$\ga = \ga_1 \oplus \dots \oplus \ga_s$, such that each $\ga_i$ 
is either simple or 
abelian so that 
$V^{k^\natural}(\ga) \cong V^{k^\natural_1}(\ga_1) \otimes \dots \otimes V^{k^\natural_s}(\ga_s)$, 
and each $k^\natural_i$ is non-critical for $\ga_i$, 
that is, $k^\natural_i \not= - h_i^\vee$ if $\ga_i$ is simple 
where $h_i^\vee$ is the dual Coxeter number of $\ga_i$, 
and $k^\natural_i \not=0$ if $\ga_i$ is abelian. 

\begin{lemma}
\label{Lem:condition_C}
Assume that $k^\natural$ 
is not critical for $V^{k^\natural}(\ga)$. 
Then
the weight-two subspace of $\Com(V^{k^\natural}(\ga), \W^k(\g,f))$ 
has dimension ${\rm mult}(\mathbb C \otimes {\rm adj} )$ in $\g$.
\end{lemma}

\begin{proof}
We follow the ideas of \cite[Corollary 3.6]{CL1}. 
Since $V^{k^\natural}(\ga)$ is non-critical, 
its vertex center is trivial, that is, 
$\Com(V^{k^\natural}(\ga), V^{k^\natural}(\ga))=
\mathbb C$. 
Next,  
$$\Com(V^{k^\natural}(\ga), \W^k(\g,f)) = \W^k(\g,f)^{\ga[t]}\subset \W^k(\g,f)^{\ga}.$$
Let $I \subset \{1, \dots, s\}$ be the set of those labels such that $\ga_i$ is simple if and only if $i \in I$. Let $J$ be the complement of $I$ in   $ \{1, \dots, s\}$.
The weight-two subspace of $\W^k(\g,f)$ is spanned by the strong generators of $\W^k(\g,f)$ at weight-two, together with the bilinears $X_{-1}Y_{-1} \bf 1$ with $X, Y$ fields of weight one. The linear span of these bilinears carries $\rm{Sym}(\rm{adj}(\ga))$ and the trivial representation of $\ga$ appears with multiplicity $|I| + |J|$. The corresponding fields in  $\W^k(\g,f)^{\ga}$ are the Sugawara vectors $\omega^i$ of the $V^{k^\natural_i}(\ga_i)$ for $i \in I$ together with the bilinears $J^j_{-1}J^\ell_{-1} \bf 1$ with $J^j, J^\ell$ in $\ga_j, \ga_\ell$ and $j, \ell \in J$.  
The weight-two subspace of $\W^k(\g,f)^{\ga}$ is thus spanned by those fields together with the 
 strong generators 
of weight-two that carry the trivial representation of~$\ga$. 

Let $X = Y(X_{-2} {\bf 1}, z) = \sum_{n\in \mathbb Z} X_n z^{-n-2}$ be a strong generator of 
$\W^k(\g,f)$ of weight-two that carries the trivial representation of $\ga$. We want to show that we can correct $X$ by fields in $V^{k^\natural}(\ga)$, so that the corrected field is in 
$\Com(V^{k^\natural}(\ga), \W^k(\g,f))$.
 
 
 Identify $\W^k(\g,f)_1$ with $\g^\natural$. 
 Consider the map 
 $\g^\natural\to\g^\natural, J\mapsto J_1X$. 
 It is $\g^\natural$-equivariant, since $X$ carries the trivial representation 
 of $\g^\natural$, hence it is $a^i\text{Id}_{\g_i}, i\in  I$, 
 for some constant $a^i$. 
 Thus $J_1 X_{-2} {\bf 1} = a^i J_{-1} {\bf 1}$ for some constant $a^i$ that does not depend on $J$.  
 
Since $J_1 L^i_{-2} {\bf 1} = J_{-1}{\bf 1}$ it follows that $X - a^i L^i$ commutes with $V^{k^\natural_i}(\ga_i)$.

Since $k^\natural$ is non-critical we can choose an orthonormal basis $J^1, \dots, J^{|J|}$ of $\bigotimes_{j\in J}  
 V^{k^\natural_j}(\ga_j)$, that is $J^j(z) J^\ell(w) = \delta_j, \ell(z-w)^{-2}$. 
 The action of $J^j_1X_{-2} {\bf 1} \in V_1 \cong \ga$ for $j =1, \dots, |J|$  is in the trivial representation of $\ga$ and hence is of the form 
 $$J^j_1X_{-2} {\bf 1} = \sum_{\ell =1}^{|J|} M^{j\ell} J^\ell_{-1} {\bf 1},$$ 
 i.e., it determines  a  $|J| \times |J|$-matrix $M = (M^{j \ell})$.
 The matrix coefficients are computed via 
 $$ J^k_1J^j_1X_{-2} {\bf 1} = \sum_{\ell =1}^{|J|} M^{j\ell} J^k_1 J^\ell_{-1} {\bf 1} = M^{jk} {\bf 1},$$
 in particular, since $[J^k_1,J^j_1] = 0$, the matrix $M$ is symmetric. 
 Let $N = \frac{1}{2} \sum_{k, \ell=1}^{|J|} M^{k \ell} J_{-1}^\ell J_{-1}^k {\bf 1}$, then $J^j_1X_{-2} {\bf 1} = J^j_1 N$ and so $X-N$ commutes with $V^{k_j^\natural}(\ga_j)$ for $j \in J$. 
 
 In conclusion $\widetilde X = X - \sum_{i \in I} a^i L^i - N$ commutes with $V^{k^\natural}(\ga)$, i.e., every strong generator of conformal weight-two can be corrected so that the corrected field belongs to 
$\Com(V^{k^\natural}(\ga), \W^k(\g,f))$.  
 \end{proof}

We notice that the setting of \Cref{Sub:criterion} applies  
to the vertex algebra $\W^k(\g,f)$ equipped with the action 
of $V^{k^\natural}(\g^\natural)$. 

Furthermore, we have the following lemma, 
proved in \cite[Theorem 7.4 b]{KMP22}
which guarantees that Condition~\ref{crit-i} 
is satisfied.

\begin{lemma}
\label{Lem:CondA}
Condition~\ref{crit-i} of \Cref{criterion} 
is always satisfied for $\W^k(\g,f)$ 
equipped with the action of $V^{k^\natural}(\g^\natural)$. 
\end{lemma}

To be precise, in \cite[Theorem 7.4 b]{KMP22}, 
the statement is proved if the datum $(\g,x,f)$ is Dynkin. 
In {\it loc. cit.} $\g$ is assumed to be a simple Lie superalgebra 
with reductive even part, 
but the proof works also if $\g$ is a reductive Lie algebra.

%
%
\vskip20pt

 In the special case of $\W$-algebras, 
 one can avoid also to 
 verify 
Condition~\ref{crit-1} 
of \Cref{criterion} 
 thanks to the following proposition. 

\begin{Pro} 
\label{Pro:paolo}
Assume that 
the central charges of $\W^k(\g,f)$ and $V^{k^\natural}(\g^{\natural})$ coincide, 
and that  the weight-two subspace of $\text{\rm Com}(V^{k^\natural}(\g^{\natural}), \W^k(\g,f))$ is one-dimensional. Then conclusion \ref{criterion-ii} 
in \Cref{criterion} holds. 
If moreover Condition \ref{crit-iv} holds, then conclusion 
\ref{criterion-i} in \Cref{criterion} holds.
\end{Pro}
\begin{proof} 
By \cite[Theorem 1.1]{AMP23}, conclusion \ref{criterion-ii} holds provided 
\begin{equation}\label{cprimo}\tag{C$^\prime$} 
(\omega'|y_j)=0 \quad \text{for all}\quad y_j\quad\text{with}\quad h_j=2.
\end{equation}
By \Cref{Lem:condition_C} and the assumption, 
${\rm mult}(\mathbb C \otimes {\rm adj} )=1$, 
hence $\omega'$ is the only strong generator at conformal weight $2$ 
which carries the trivial representation. 
This implies $(\omega'|y_j)=0$ for $j>1$. 
Since the central charges of $\W^k(\g,f)$ and $V^{k^\natural}(\g^{\natural})$ coincide, 
we have that  $(\omega'|\omega')=0$, 
so that \eqref{cprimo} holds. 
If in addition condition \ref{crit-iv} holds, 
exactly the same argument in the proof of \Cref{criterion} 
yields conclusion \ref{criterion-i}.
\end{proof}

\begin{Rem}
If $\g^\natural$ is trivial, then our criterion reads off as follows. 
If the central charge of $\W^k(\g,f)$ is zero, 
and if the multiplicity of adj in one, then $\W^k(\g,f) \cong \C$. 
\end{Rem}

For admissible level $k$, 
we can also apply the criterions of \cite{AEM} 
based on asymptotic data: 
\cite[Theorem 1.1]{AEM} 
for collapsing levels 
and \cite[Theorem 3.10]{AEM}  
for finite extensions. 
Finally, for $\W$-algebras of rank two, one can 
use the explicits OPEs described in \cite{Fasquel-OPE}. 

To finish the section, 
here are some comments about finite extensions.
Recall that $\W_k(\g, f)$ is the simple quotient of $\W^k(\g, f)$.
\begin{Pro}
\label{Pro:FE_admissible}
Assume that the conditions \ref{crit-1}--\ref{crit-iii} hold. 
If furthermore $k^\natural$ is admissible and that 
$\ga$ is semisimple, then 
$\W_k(\g, f)$ is a finite extension of $L_{k^\natural}(\ga)$. 
\end{Pro}

\begin{proof} 
By \Cref{prop:ext} \ref{prop:ext-iii}  
$\W_k(\g, f)$ is a conformal extension of $L_{k^\natural}(\ga)$. 
Since $\W_k(\g,f)$ belongs to 
the category $\mc{O}_{k^\natural}$ with 
admissible $k^\natural$, 
it decomposes into a direct sum of 
simple $L_{k^\natural}(\ga)$-modules.  
Because $\W_k(\g,f)$ is an ordinary module, 
the sum must be finite: any infinite extension would have 
necessarily infinite dimensional conformal weight spaces. 
This concludes the proof. 
\end{proof}


In the case where both $k^\natural$ and $k$ are admissible, 
\Cref{Pro:FE_admissible} is also a consequence of \cite[Theorem 3.10]{AEM}, 
provided that $\W_k(\g,f)$ is self-dual. This holds for instance 
if the conformal grading comes from the Dynkin grading 
(see \cite[Proposition 4.2]{AEM}). 

In this connection, it is worth pointing out
the following conjecture (see \cite[Conjecture 1.3]{AEM}):

\begin{Conj}
\label{Conj:finite_extension}
If $W$ is a finite extension of a vertex algebra $L$, 
then the corresponding morphism of
Poisson algebraic varieties $X_W \to X_L$ is a dominant morphism.
\end{Conj}
In the case where $L_{k^\natural}(\ga)$ is admissible 
and $\W_k(\g,f)$ is a finite extension of $L_{k^\natural}(\ga)$, the conjecture is equivalent 
to that $\dim X_{\W_k(\g,f)} = \dim X_{L_{k^\natural}(\ga)}$ 
and $X_{\W_k(\g,f)}$ is irreducible. 
Several of our conjectures in \Cref{Sec:examples_exceptional} go in this direction. 



\section{Examples in the classical types}

\subsection{Type $A$} 
Consider $\g=\mathfrak{sl}_{m|n},\,m\ne n$. 
Let $(e,x,f)\subset \g_{\bar 0}$ be an $\mathfrak{sl}_2$-triple with 
$[x,e]=e,\,[x,f]=-f,\,[e,f]=x$.
Let $\mathfrak s=\C e\oplus \C x\oplus \C f$. First observe 
that 
the centralizer of $\mathfrak s$ in $\mathfrak{gl}_{m|n}$ equals $\g^\natural\oplus\C \text{\text{Id}}.$ 
Moreover $\Hom(\C^{m|n},\C^{m|n})=\mathfrak{gl}_{m|n}$ decomposes as a 
$\mathfrak{gl}_{m|n}$-module as
$\mathfrak{sl}_{m|n}\oplus\C \text{Id},$ and the identity acts trivially. 
It follows that the multiplicity of a non-trivial representation $V(\nu)\otimes V$ 
of $\g^\natural\oplus \mathfrak s$ occurs in $\mathfrak{sl}_{m|n}$ 
with the same multiplicity of $V(\nu)\otimes V\otimes \C$ 
of $\g^\natural\oplus \mathfrak s\oplus \C \text{Id}$ in $\mathfrak{gl}_{m|n}$.
So it suffices to decompose $\Hom(\C^{m|n},\C^{m|n})$ 
as $\g^\natural\oplus\mathfrak s\oplus \C \text{Id}$-module.

Suppose that $f$ is encoded by  a pair of partitions $(\alpha,\beta)$
\begin{equation}
\label{part}
\alpha=(m_1^{\mu_1},m_2^{\mu_2},\ldots,m_t^{\mu_t})\vdash m,\quad
\beta=(m_1^{\nu_1},m_2^{\nu_2},\ldots,m_t^{\nu_t})\vdash n,
\end{equation} 
with non-negative integer multiplicities and $\mu_i+\nu_i>0$.
This means that  $\C^{m|n}=\C^{m|n}_{\bar 0}\oplus\C^{m|n}_{\bar 1}= \C^m\oplus \C^n$ 
decomposes under $\mathfrak s$ as
\begin{equation}
\label{deco}
\C^{m}\oplus \C^n=\left(\bigoplus_{i=1}^t M(m_i-1)\right)\oplus\left(\bigoplus_{i=1}^t \bar 
M(m_i-1)\right),
\end{equation}
where $M(m_i-1)\cong V(m_i-1)^{\oplus \mu_i}$ 
(resp. $\bar M(m_i-1)\cong V(m_i-1)^{\oplus \nu_i}$) 
are  the isotypic components of the decomposition 
of $\C^m$ (resp. $\C^n$) as $\mathfrak s$-module, and $V(p)$
is the  $(p+1)$-dimensional irreducible representation of $\mathfrak s$ of highest weight $p$. 

Let  $L(p)=M(p)^e, \bar L(p)=\bar M(p)^e$ be the highest weight spaces of 
$M(p), \bar M(p)$, 
respectively. 
The map $\varphi\in\g^\natural\oplus\C \text{Id}\mapsto \sum_i\varphi_{|L(m_i-1)\oplus \bar L(m_i-1) }$ defines an isomorphism 
\begin{equation}
\label{gnat}
\g^\natural\oplus\C \text{Id}\cong\prod_{s=1}^{t} \mathfrak{gl}_{\mu_s|\nu_s}.
\end{equation}
Note that, as a  $(\g^\natural\oplus \C \text{Id})\times \mathfrak{s}$-module
$$M(m_i-1)\oplus \overline M(m_i-1)= \C^{\mu_i|\nu_i}\otimes V(m_i-1),$$
where $\C^{\mu_i|\nu_i}$ is the natural representation of $\mathfrak{gl}_{\mu_i|\nu_i}$. 
We have, by Clebsch--Gordan formula, 
\begin{align}\notag\Hom(\C^{m|n},\C^{m|n})
&=\bigoplus_{p,q} \Hom (\C^{\mu_p|\nu_p},\C^{\mu_q|\nu_q})
\otimes \Hom(V(m_p-1),V(m_q-1))\\
&=\bigoplus_{p,q}   \Hom (\C^{\mu_p|\nu_p},\C^{\mu_q|\nu_q})
\bigotimes\sum_{s=0}^{\min(m_p-1,m_q-1)} V(m_p+m_q-2s-2).
\label{riga}
\end{align}
Note that $\Hom (\C^{\mu_p|\nu_p},\C^{\mu_q|\nu_q})$ is irreducible 
and non-trivial
as a $(\mathfrak{gl}_{\mu_p|\nu_p}\times \mathfrak{gl}_{\mu_q|\nu_q})$-module 
if $p\ne q$. If $p=q$ and $\mu_p\ne \nu_p,$ 
${\rm End}(\C^{\mu_p|\nu_p})=\mathfrak{sl}_{\mu_p|\nu_p}\oplus \C$ as a 
$\mathfrak{gl}_{\mu_p|\nu_p}$--module. 
In the remaining cases we do not have 
the complete reducibility, but we can count the 
Jordan--H\"older multiplicity. 
If 
 $p=q$ and $\mu_p=\nu_p\geq 2$ then the trivial representation has multiplicity $2$ 
 and we have a Jordan--H\"older factor $\mathfrak{psl}_{\mu_p}$.
Finally, if $\mu_p=\nu_p=1$, then only  the trivial representation occurs 
in the Jordan--H\"older series with multiplicity $4$.
 \par
Recall that we are interested only in the case 
when $\g^\natural$ is a reductive Lie algebra or 
is $\mathfrak{osp}_{1|2N}$.
Relation \eqref{gnat} shows that the latter case does not occur 
and that $\mu_k\nu_k=0$ for $1\leq k\leq t$. 
Recall that
${\rm adj}=V(2)$. Then $\mathbb C \otimes {\rm adj}$ occurs in \eqref{riga} only when $p=q$ and $m_p=m_q>1$
and $s=m_p-2$, hence the  multiplicity is $|\{m_i\colon m_i>1\}|$. 
Let us now handle the case $\g=\mathfrak{psl}_{m|m}$. Consider the decomposition
\begin{equation}\label{decglnn}\mathfrak{gl}_{m|m}=\mathfrak{sl}_{m|m}\oplus\C\begin{pmatrix} I_m & 0\\ 0 & 0\end{pmatrix}.\end{equation}
Note that the complement in the r.h.s. is not $\mathfrak{gl}_{m|m}$-stable, but  the multiplicity of ${\rm adj}$ for the action of  $\mathfrak s$ on $\mathfrak{psl}_{m|nm}$ equals the multiplicity of  ${\rm adj}$ for the action of a lifting of 
 $\mathfrak s$ in $\mathfrak{sl}_{m|m}$ on $\mathfrak{gl}_{m|m}$: indeed this lifting acts trivially both on $\begin{pmatrix} I_m & 0\\ 0 & 0\end{pmatrix}$ and on 
 $\begin{pmatrix} I_m & 0\\ 0 & I_n\end{pmatrix}$.

Set $C:=Centr_{\mathfrak{gl}_{m|m}}(\mathfrak s)$. First note that 
\begin{equation*}
\label{gnatnn}
\g^\natural= \pi(C\cap \mathfrak{sl}_{m|m}), 
\end{equation*}
where $\pi: \mathfrak{sl}_{m|m}\to \mathfrak{psl}_{m|m}$ is the natural projection. We have
\begin{equation}\label{isomcc}C\cong \prod_{p=1}^t\mathfrak{gl}_{\mu_p|\nu_p}.\end{equation}
The above isomorphism is explicitly given by  $C\ni (A_1,\ldots A_t)\mapsto diag(\underbrace{A_1\ldots}_{m_1\text{times}},\ldots,\underbrace{A_t\ldots}_{m_t\text{times}}).$   Since
$$C\cap \mathfrak{sl}_{m|m}=\left\{(A_1,A_1,\ldots,A_t,A_t)\mid \sum_{i=1}^t m_i\,str(A_i)=0\right\}, $$
under the isomorphism \eqref{isomcc} we have
$$C\cap \mathfrak{sl}_{m|m}\cong \left\{(A_1,\ldots A_t)\in \prod \mathfrak{gl}_{\mu_p|\nu_p}\mid \sum_{i=1}^t m_i\,str(A_i)=0\right\},$$
and in turn
\begin{equation}\label{gnatslnn}\pi(C\cap \mathfrak{sl}_{m|m})\cong \left\{(A_1,\ldots A_t)\in \prod \mathfrak{gl}_{\mu_p|\nu_p}\mid \sum_{i=1}^t m_i\,str(A_i)=0\right\}\big\slash \C(I_{\mu_1+\nu_1},\ldots,I_{\mu_t+\nu_t}).\end{equation}
Since we are interested in the cases when $\g^\natural$ is a  reductive Lie algebra, we deduce from \eqref{gnatslnn} that $\mu_i\nu_i=0$ for all $i=1,\ldots,t$. In particular $C$ is a Lie algebra. Moreover, using the identification 
\eqref{isomcc}, the decomposition \eqref{riga} reads more explicitly in this case as
 we have
 \begin{align}&\mathfrak{gl}_{m|m}=\Hom(\C^{m|m},\C^{m|m})=\notag\\&\bigoplus_{p} (\Hom (\C^{\mu_p|0},\C^{\mu_p|0})\oplus \Hom (\C^{0|\nu_p},\C^{0|\nu_p}))
 \otimes \Hom(V(m_p-1),V(m_p-1))\label{primariga}\\
 &\bigoplus_{p,q} (\Hom (\C^{\mu_p|0},\C^{0|\nu_q})\oplus  \Hom (\C^{0|\nu_p},\C^{\mu_q|0}))
 \otimes \Hom(V(m_p-1),V(m_q-1))\notag\\
 &\bigoplus_{p\ne q} (\Hom (\C^{\mu_p|0},\C^{\mu_q|0}))\oplus\Hom (\C^{0|\nu_p},\C^{0|\nu_q}))
 \otimes \Hom(V(m_p-1),V(m_q-1)).\notag
 \end{align}
 By using as above Clebsch-Gordan formula, 
we see that ${\rm adj}$ occurs in the above decomposition only 
when $p=q$ and $m_p=m_q>1$
 or $|m_p-m_q|=2$. We want to prove that if there are two distinct parts strictly greater than $1$, then $mult(\C\otimes {\rm adj})>1$, where $\C$ is the trivial representation of $\g^\natural$. Assume that  $m_p>m_q\ge 2$. Then in \eqref{primariga} appear at least two copies of the trivial representation for $C$, hence for $C\cap \mathfrak{sl}_{m|m}$. By \eqref{decglnn}, since $C\times \mathfrak s$ acts trivially on $\C\begin{pmatrix} I_m & 0\\ 0 & 0\end{pmatrix}$, the isotypic component of ${\rm triv}\otimes {\rm adj}$,  where ${\rm triv}$ denotes the trivial representation of $C$, lies in $\mathfrak{sl}_{m|m}$, in particular $mult({\rm triv}\otimes {\rm adj})>1$ w.r.t.
 $(C\cap \mathfrak{sl}_{m|m})\times \mathfrak s$. The action of $(C\cap \mathfrak{sl}_{m|m})$ on  $\mathfrak{sl}_{m|m}$ pushes down to define an action of $\g^\natural$ on $\mathfrak{sl}_{m|m}$. 
 $\C\times Id$ is stable for this action, and the action induced on the quotient is precisely the action of $\g^\natural$ on  $\mathfrak{psl}_{m|m}$. In this process,   $\g^\natural$ acts trivially on  the isotypic component $J$ of ${\rm triv}\otimes {\rm adj}$. But $J$ maps injectively into the isotypic component of $\C\otimes {\rm adj}$ in $\mathfrak{psl}_{m|m}$, since $\g^\natural\times \mathfrak s$ acts trivially on $\C Id$, hence 
 $J\cap \C Id=\{0\}$.
 The upshot is that the multiplicity of $\C\otimes {\rm adj}$ is at least 2. Let us now prove that when $f$ is encoded by $(q^l|1^m),\,m=ql$, then 
 $mult(\C\otimes {\rm adj})=1$. In this case
\begin{align}\label{r2}&\mathfrak{gl}_{m|m}=\Hom (\C^{l|0},\C^{l|0})\otimes \Hom(V(q-1),V(q-1))\\&\oplus \Hom (\C^{0|m},\C^{0|m})\otimes
 \C\oplus
 (\Hom (\C^{l|0},\C^{0|m})\oplus  \Hom (\C^{0|m},C^{l|0}))\otimes V(q-1),\notag\end{align}
and arguing as above one proves that $\C\otimes {\rm adj}$ occurs, with multiplicity 1, in $\mathfrak{psl}_{m|m}$, coming from the copy of 
${\rm triv}\otimes {\rm adj}$ which occurs in r.h.s of \eqref{r2}.\par
\vskip10pt
Summing up, we have proven
\begin{Pro}\label{Pro:partitions}
If $\g^{\natural}$ is a Lie algebra, then ${\rm mult}(\mathbb C \otimes {\rm adj} )=|\{m_i\colon m_i>1\}|$. 
In particular, ${\rm mult}(\mathbb C \otimes {\rm adj} )=1$ 
if and only if $f$ is encoded by pairs 
\begin{enumerate} [{\rm (1)}]
\item 
\label{part-1}
$(q^l,1^r|0)$, $m=ql+r$, $n=0$, 
$(0|q^l,1^r)$,  $n=ql+r$, $m=0$, 
\item  
\label{part-2}
$(q^l|1^n)$, $m=ql$, $(1^m|q^l)$, $n=ql$.
\end{enumerate}
\end{Pro}

Now we look for conformal levels.
Clearly, it suffices to consider the leftmost case in (1), (2). As for (1), 
consider $\mathfrak{sl}_{m}\equiv \mathfrak{sl}_{m|0}$ 
and the $\mathfrak{sl}_2$-triple corresponding to the partition $(q^l,1^r)$.
Then one has the central charge (cf. \cite[Appendix C]{XieYan})

\begin{equation}
\label{eq:cemnralA}
c(k,f)
=lq(k-m+(l+3m)q-(k+l+m)q^2)
-\frac{k+k(l-m)m+lm^2}{k+m}.
\end{equation}
By \cite[Table 2]{AEM}, we have $\g^\natural\cong  \C\times \mathfrak{sl}_{l}\times \mathfrak{sl}_{r}$ with shifted levels
$$k^\natural_0= k^\natural_1=qk+m(q-1),\quad k^\natural_2=k+l(q-1).$$
The equality of central charges reads, for $l>1, r>1$
$$c(k,f)=1+\frac{k_1^\natural(l^2-1)}{k_1^\natural+n}+\frac{k_2^\natural((m-lq)^2-1)}{k_2^\natural+m-lq}$$
which gives the following conformal levels
\begin{align*}&k_{1} =-\frac{m q}{q+1}=-h^\vee+\frac{h^\vee}{q+1},&&k_{2} =
   \frac{m-1-mq}{q}=-h^\vee+\frac{h^\vee-1}{q},\\&k_{3} =\frac{m+1-mq}{q}=-h^\vee+\frac{h^\vee+1}{q},&&k_{4} =\frac{2
   m-l-mq}{q-1}=-h^\vee+\frac{h^\vee-l}{q-1}
\end{align*}
 If $l=1$ or $m=lq+1$ the solution $k_{2}$ has to be disregarded; the rectangular case, 
 i.e., $r=0$, as well as the hook case $l=1$ have been already treated in  \cite{AMP23}. 
 \par
In case (2),  consider $\mathfrak{sl}_{ql|n}$ and the $\mathfrak{sl}_2$-triple corresponding to the partition $(q^l| 1^n)$.
Then one has the central charge 
$$
c(k,f)=\frac{k ((q l - n)^2 - 1)}{k + q l - n} - k q l (q^2 - 1) - 
 l^2 q (1 + (q - 2) q^2) + l n (q - 1) (q^2 - 2 q - 1)
$$
and $\g^\natural=\C\times\mathfrak{sl}_l\times\mathfrak{sl}_n$ the shifted levels of the semisimple summands are 
$$
k^\natural_1 = qk+ lq(q-1) -(q-1)n,\quad
k^\natural_2 = -k-(q-1)l
$$
We have equality of central charges at the following four points
\begin{align*}&k_1=\frac{n q-l q^2}{q+1} &&k_2=\frac{-l
   q^2+l q+n q-n-1}{q}\\
   &k_3=\frac{-l q^2+l q+n
   q-n+1}{q}&&k_4=\frac{-l q^2+2 l q-l+n q-2
   n}{q-1}
 \end{align*}
Recalling that the dual Coxeter number of $\mathfrak{sl}_{m|n}$ is $h^\vee=m-n$, the above analysis can be summed up in the following 

\begin{Pro}\label{cesl}
The conformal levels for $\W^k(\mathfrak{sl},f)$, for nilpotent elements $f$ 
attached to partitions (1)--(2), with respect to half the supertrace form on $\mathfrak{sl}_{m|n}$, $m\ne n$
are given by 
\begin{align*}
&k_{1} =-h^\vee+\frac{h^\vee}{q+1},
&&k_{2} =
-h^\vee+\frac{h^\vee-1}{q},\\&k_{3} =-h^\vee+\frac{h^\vee+1}{q},
&&k_{4}= -h^\vee+\frac{h^\vee-g}{q-1}.
\end{align*}
where $g=\mu_1-\nu_1$. For  $\mathfrak{psl}_{m|m}$ the conformal levels are $k_2, k_3, k_4$.
\end{Pro}

 \begin{Rem} We may consider also the case when the bilinear form $k_0$ (cf. \eqref{Vkg}), is degenerate. In this case
 the conformal vector $L_{\g^\natural}$ for $V^k(\g^\natural)$ is defined as in \cite[Remark 2.6]{AMP23}. The degeneracy  condition of $k_0$ is equivalent to \begin{equation}\label{k5}
  k=-h^\vee+\frac{h^\vee}{q}.
  \end{equation}
  One easily verifies that for this value $c(k,f)$ equals the central charge of $L_{\g^\natural}$.
 \end{Rem}
To check condition \ref{crit-iv} in Criterion \ref{criterion} we have to compute the $h_i$ and the $h_{\lambda_i}$ (cf.~\eqref{eq:h_lami}). Assume first $m\ne n$.
Consider case (1), so that $m_1=q, m_2=1, \mu_1=l, \mu_2=r$. Equation \eqref{riga} reads
\begin{align*}
\Hom(\C^m,\C^m)&= \Hom (\C^l,\C^l)\otimes \sum_{k=0}^{q-1}
  V(2q-2k-2)\\&\oplus( \Hom(\C^l,\C^r)\oplus \Hom(\C^r,\C^l))\otimes V(q-1)\oplus \Hom(\C^r,\C^r)\otimes V(0).\end{align*}
  Hence the conformal weights are $1,\ldots,q$ if $q$ is odd and $1,\ldots,q, \frac{q+1}{2}$ if $q$ is even. Next we compute the $h_{\lambda_i}$.
  If $h_i\ne \frac{q+1}{2}$ and $h_i\geq 2$, then we study the action of $\mathfrak g^\natural$ on $\Hom(\C^l,\C^l)$. 
  We have $\g^\natural\cong\C\oplus\mathfrak{sl}_l\oplus
  \mathfrak{sl}_r$: $\mathfrak{sl}_r$ acts trivially, $\mathfrak{sl}_l$ acts by the restriction of adjoint representation of $\mathfrak{gl}_l$ and $\C$ acts trivially. 
  Hence, in the normalised form on $\mathfrak{sl}_l$, we have
  $\lambda_i=0+\theta_{\mathfrak{sl}_l}+0$, so that $$h_{\lambda_i}=\frac{(\theta_{\mathfrak{sl}_l},\theta_{\mathfrak{sl}_l}+2\rho_{\mathfrak{sl}_l})}{2(q k+(q-1)m+l)}
  =\frac{l}{(q k+(q-1)m+l)}$$
  Substituting $k=k_i,\,1\le i\le 4$, we never found an integer between $2$ and $q$ 
  (except for $k_2, l=2$, to be discussed separately).\par
   If $h_i=\frac{q+1}{2}$,  we have to decompose 
   $\epsilon_q\cdot \Hom (\C^l,\C^l)\oplus \Hom (\C^l,\C^r)\oplus Hom (\C^r,\C^l)$ 
   as a $\g^\natural$-module (here $\epsilon_q$ is $1$ if $q$ is odd and $0$ otherwise). 
   As above, for the first summand (if present) we have   $\lambda_i=0+\theta_{\mathfrak{sl}_l}+0$ or $\lambda_i=0$. Arguing as above, we see that $h_{\lambda_i}$ cannot be an integer between $2$ and $q$. We now discuss the contribution of the remaining summands. Identify the center $\C$ with $\C\varpi$, 
 $\varpi = {\rm diag} (-rI_{ql},ql I_{r})$. 
 We normalise the form on $\C\varpi$ setting
   $(\varpi,\varpi)=rlm$ (cf. \cite[Lemma 8.4]{AEM}). 
   Let $\lambda^0\in(\C\varpi)^*$ be the functional having value $1$ on $\varpi$.
   Hence $|\lambda^0|^2=1/|\varpi|^2=\tfrac{1}{rlm}$. 
   $\Hom (\C^l,\C^r)$ (resp.~$\Hom (\C^r,\C^l)$) is irreducible with highest weight $m\lambda^0+\omega_{l-1}^{\mathfrak{sl}_l}+\omega_{1}^{\mathfrak{sl}_r}$ (resp $-m\lambda^0+\omega_{1}^{\mathfrak{sl}_l}+\omega_{r-1}^{\mathfrak{sl}_r}$). The corresponding $h_\lambda$ are the same and they equal
   $$h(k)=\frac{m}{2rl(q k+(q-1)m)}+\frac{l^2-1}{2l(q k+(q-1)m+l)}+\frac{r^2-1}{2r(k+(q-1)l+r)}.$$
   We have $h(k_i)=\frac{q+1}{2}$ for $i=1,2,3$, 
   hence \Cref{criterion} does not apply. 
   Instead, 
  $$
  h(k_4)=\frac{(q-1) \left(l \left(q^2+q
   \left(r^2-2\right)+1\right)+r
   \left(2 q+r^2-2\right)\right)}{2
   r^2 (l (q-1)+r)},
   $$
   which equals $\frac{q+1}{2}$ only if $r=q-1$. 
   Hence we conclude that $k_4$ is a collapsing level for $(q^l,1^r), r\ne q-1$.\par
   Consider case (2), so that $m_1=q, m_2=1, \mu_1=l, \nu_1=0,\mu_2=0, \nu_2=r$. Equation \eqref{riga} reads
\begin{align*}& \Hom(\C^{m|r}\C^{m|r})= \left(\Hom (\C^{l|0},\C^{l|0})\otimes \sum_{s=0}^{q-1}
  V(2q-2s-2)\right)\\&\oplus\left((\Hom(\C^{l|0}\otimes \C^{0|r})\oplus \Hom(\C^{r|0},\C^{0|l}))\otimes V(q-1)\right)\oplus\left( \Hom(\C^{0|r},\C^{0|r})\otimes V(0)\right).\end{align*}
and the computation of the $h_{\lambda_i}$ follows the lines of case (1). 
This time $\varpi = {\rm diag} (rI_{ql},ql I_{r})$; the relevant functions $h_\lambda$ are 
{\small
\begin{align*}&\frac{l}{k q+l (q-1)
   q+l-n q+n},\\&\tfrac{1}{2}
   \left(\frac{l^2-1}{l (k q+l
   (q-1) q+l-n
   q+n)}+\frac{n^2-1}{n (-k-l
   q+l+n)}+\frac{n-l q}{l n (q
   (k+l (q-1))+n
   (-q)+n)}\right).
\end{align*}}

\smallskip 

\noindent
   It turns out that \Cref{criterion} does not apply for $k_i, 1\leq i \leq 3$ and that $k_4$ is collapsing provided
  \begin{equation}\label{cond}l\notin \left\{
  \frac{n+1}{q},\frac{2
   n+q+1}{2 q}, \frac{n+2
   q-1}{q}\right\},\quad \frac{l-l q}{-2 l q+l+n+q}\notin\{2,\ldots,q\}.
\end{equation}
Finally consider the case of $\mathfrak{psl}_{m|m}$; then it suffices to consider partitions $(q^l|1^m)$, so that $\g^\natural= \mathfrak{sl}_l\times \mathfrak{sl}_m$. Since
\begin{align}&\mathfrak{gl}_{m|m}=\left(\mathfrak{sl}_{l}\otimes \sum_{s=0}^{q-1} V(2q-2s-2)\right)\oplus\left(\C \otimes\sum_{s=1}^{q-1} V(2q-2s-2)\right)\oplus \left(\C\otimes V(0)\right)\\&\oplus\left(\mathfrak{sl}_{m}\otimes
 \C\right)\oplus\left(\C\otimes V(0)\right)\oplus
 (\Hom (\C^{l|0},\C^{0|m})\oplus  \Hom (\C^{0|m},C^{l|0}))\otimes V(q-1),\notag\end{align}
 we have, as $(\g^\natural\times \mathfrak s)$-module,
   \begin{align*}\mathfrak{psl}_{m|m}&=\left(\mathfrak{sl}_{l}\otimes \sum_{s=0}^{q-1} V(2q-2s-2)\right)\oplus\left(\C \otimes\sum_{s=1}^{q-1} V(2q-2s-2)\right)\\ &\oplus\left(\mathfrak{sl}_{m}\otimes
 \C\right)\oplus
 (\Hom (\C^{l|0},\C^{0|m})\oplus  \Hom (\C^{0|m},C^{l|0}))\otimes V(q-1).\end{align*}
Repeating the above discussion we see that we have to calculate $h_\lambda$ when $\lambda\in\{\theta^{\mathfrak{sl}_l},  \omega_{l-1}^{\mathfrak{sl}_l}+\omega_{1}^{\mathfrak{sl}_m}, \omega_{1}^{\mathfrak{sl}_l}+\omega_{m-1}^{\mathfrak{sl}_m}\}$. These furnctions, in the first and in the second  case (which equals the third), are 
$$\frac{l}{k
   q+l},-\frac{(q+1) \left(k l
   (q-1) q+l^2 q-1\right)}{2 q
   (k-l) (k q+l)}.$$
One readily see that \Cref{criterion} applies only to $k_4$. 
Summing up

\begin{Pro} 
\label{Pro:concl_Type-A}
\Cref{criterion} never applies to the conformal levels $k_1,k_2, k_3$. 
The conformal level  $k_4$ is collapsing in case (1) provided $l\ne q-1$, 
and in case (2) provided \eqref{cond} holds.
 \end{Pro}

\begin{Rem} 
One proves that \Cref{criterion} applies to the value of $k$ 
given by \eqref{k5}, which is therefore collapsing.
  \end{Rem}    

By using arguments as in \cite[Section 9]{AEM} or using quasi-lisse properties of admissible affine vertex  algebras and $\W$-algebras we conclude:

\begin{Pro}
If $k_1, k_2$ are admissible,  
then they are not collapsing. 
\end{Pro}

\subsection{Types $B, C, D$} 
Let $V=\C^{m|n}, n=2q$ and $(\cdot |\cdot)$ 
be the supersymmetric bilinear form 
whose matrix in the standard basis is 
$\left[ 
\begin{array}{c|c c} 
  J_m & 0  & 0  \\   \hline  0 & 0 & J_q \\0 & -J_q & 0
  \end{array} 
\right]$, 
where $J_p$ is a $p\times p$ matrix with $1$ on the anti-diagonal and $0$ elsewhere.
Let $\g=\mathfrak{osp}(V)$. 
We assume that $\g$ is simple, hence we exclude the case $\g=\mathfrak{so}_4$. 
The case $\g=\mathfrak{so}_3$ has already been treated in the previous  section. 
Hence we also exclude this case.
\par
Consider  $\mathfrak{s}=\C e \oplus \C x\oplus \C f\subset\g_{\bar 0}$. 
Suppose that $f$ is encoded as in \eqref{part} and recall the decomposition 
\eqref{deco}. 
Set $\g^\natural=\g^{\mathfrak s}, 
\widetilde M(r)=M(r)+\bar M(r), \widetilde L(r)=\widetilde M(r)^e$ and 
\begin{equation}\label{gnatr}
\g^\natural(r)=\{X\in \g^\natural\colon X(\widetilde M(r))\subset \widetilde M(r), X(M(s))=X(\bar M(s))=0\text{ if $s\ne r$}\}.
\end{equation}
It is clear that $\g^\natural(r)$ is a sub(super)algebra of $\g^\natural$. 
If $X\in \g^\natural(r)$, 
then $X(\widetilde L(r))\subset \widetilde L(r)$
and $X$ is uniquely determined by $X_{|\tilde L(r)}$. 
This means that the map $\phi_r:\g^\natural(r)\to \Hom_\C(\widetilde L(r),\widetilde L(r))$ 
defined by $X\mapsto X_{|\widetilde L(r)}$ is an injective linear map.

Recall that $\widetilde L(r)$ is equipped   with the   non-degenerate bilinear
form $\langle\cdot,\cdot\rangle_r$ defined by
$$
\langle v,w\rangle_{r}=(v|f^r w),
$$
which is supersymmetric or skewsupersymmetric according to whether $r$ is even or odd.

\begin{lemma}
\label{trivial}
As Lie superalgebras
\begin{equation}\label{gp}
\g^\natural=\bigoplus_r\g^\natural(r).
\end{equation}
Moreover  
$\phi_r(\g^\natural(r))\cong \mathfrak{osp}(\widetilde L(r),\langle\cdot,\cdot\rangle_{r})$ 
if $r$ is even and 
$\phi_r(\g^\natural(r))\cong \mathfrak{spo}(\widetilde L(r),\langle\cdot,\cdot\rangle_{r})$ 
if $r$ is odd.
\end{lemma}

\begin{proof}
Let $\pi_{r}$ be the projection to $\widetilde M(r)$ 
with respect to the decomposition 
\begin{equation}
\label{od}
\C^{m|n}=\bigoplus_r \widetilde M(r).
\end{equation}
Then $\pi_{r}$ is  an $\mathfrak s$-map. 
If $X\in \g^\natural$ 
then $X(\widetilde M(r))\subset \widetilde M(r)$ so $\pi_sX\pi_r=0$ unless $r=s$. 
It follows that 
$$
X=\sum_{r,s}\pi_{s} \circ X\circ \pi_r=\sum_r\pi_{r} \circ X\circ \pi_r.
$$
Since the decomposition \eqref{od} is orthogonal with respect to $(\cdot|\cdot)$, 
we have
$$
(\pi_rv|w)=(\pi_rv|\pi_rw)=(v|\pi_rw).
$$
It follows that
$$(\pi_rX\pi_rv|w)=(X\pi_rv|\pi_rw)
=-(-1)^{p(X)p(v)}(\pi_rv|X\pi_rw)
=-(-1)^{p(X)p(v)}(v|\pi_rX\pi_rw),
$$
 and this means that $\pi_rX\pi_r\in\g$. 
 Since both $\pi_r$ and $X$ are $\mathfrak s$-maps, 
 it follows that $\pi_rX\pi_r\in\g^\natural$ and, by construction, 
 $\pi_rX\pi_r\in\g^\natural(r)$. 
 We have therefore proven that, as vector spaces,
$
\g^\natural=\sum_r\g^\natural(r)
$
and the sum is obviously direct. 
It is clear that $[\g^\natural(r),\g^\natural(s)]=\{0\}$ if $r\ne s$, hence 
$
\g^\natural=\bigoplus_r\g^\natural(r)
$
as Lie superalgebras.
We  have indeed shown that $\phi_r(\g^\natural(r))$ lies in 
$\mathfrak{osp}(\widetilde L(r))$ or $\mathfrak{spo}(\widetilde L(r))$ 
according to whether $r$ is even or odd.
It remains only to prove that $\phi_r$ is  an isomorphism. 
We already observed that $\phi_r$ is injective. 
We now show that  it is surjective. 
Recall that 
$$
\widetilde M(r)=\bigoplus_{s=0}^r f^s \widetilde L(r).
$$
Assume $r$ even and  let $a\in \mathfrak{osp}(L(r),\langle\cdot,\cdot\rangle_r)$; 
define $X_a\in Hom_\C(V,V)\cong gl(m|n)$ by
$$
X_a(v)=\begin{cases}0&\text{if $v\in \widetilde M(q)$ with $q\ne r$},\\
f^saw&\text{if $v=f^s w$ with $w\in \widetilde L(r)$ and $s\le r$}.
\end{cases}
$$
It is clear that, 
 if $v\in \widetilde M(q)$ and $w\in \widetilde M(s)$ 
 with $s\ne q$ or $s=q\ne r$, then 
$$
(X_av|w)=0=-(-1)^{p(X_a)p(v)}(v|X_aw).
$$
If $v_1=f^sw_1\in f^s\widetilde L(r)$ 
and $v_2=f^tw_2\in f^t\widetilde L(r)$ with $s+t\ne r$, then
$$
(X_av_1|v_2)=0=-(-1)^{p(X_a)p(v_1)}(v_1|X_av_2)
$$
as well.
Finally, if $s+t=r$, then
\begin{align*}
(X_af^sw_1|f^tw_2)&=(f^saw_1|f^tw_2)
=(-1)^s\langle aw_1,w_2\rangle_{r}
=-(-1)^s(-1)^{p(a)p(w_1)}\langle w_1,aw_2\rangle_{r}\\
&=-(-1)^{p(a)p(w_1)}(f^sw_1|X_af^tw_2).
\end{align*}
 It remains to check that $X_a\in\g^\natural(r)$ 
 and that $\phi_r$ is a Lie algebra map. 
It is clear that $[X_a,f]=[X_a,h]=0$. 
Since $X_a(\widetilde L(r))\subset  \widetilde L(r)$ an obvious induction 
proves that $[X_a,e](f^s\widetilde L(r))=\{0\}$, so $[X_a,e]=0$ as well. 
It is obvious that the map $X\mapsto X_{|\widetilde L(r)}$ is a Lie superalgebra map.
The statement when $r$ is odd is  proven in the very same way.
\end{proof}

\begin{Rem} 
Under the isomorphism $\phi=\bigoplus_r\phi_r$ of the previous Lemma, 
as $(\g^\natural \times \mathfrak s)$--module 
$$ \widetilde  M(r)= \widetilde  L(r)\otimes V(r)
$$
with  $\g^\natural(r)$ acting on $ \widetilde  L(r)$ by its defining representation and $\g^\natural(s)$ acting  by zero if $s\ne r$.
\end{Rem}

Recall that 
$\g=\mathfrak{osp}(\C^{m|n})=\Lambda ^2(\C^{m|n})$. 
We assume that $\g^\natural$ is even or  has components of type $\mathfrak{osp}_{1|r}$. 
This means that in the encoding \eqref{part}, we have
$\mu_i+\nu_i>0,\quad 0\le \mu_i\le 1 \text{ or }\  0\le \nu_i\le 1.$
We now partion the set $\{m_1,\ldots,m_t\}$
\begin{align*}&\{m_1,\ldots,m_t\}=\\&\{q_{1}\geq \ldots\geq q_{r}\}\sqcup \{q_{r+1}\geq \ldots\geq q_{s}\}
\sqcup \{q_{s+1}\geq \ldots\geq q_{u}\}\sqcup \{q_{u+1}\geq \ldots\geq q_{t}\}\end{align*}
as follows:
\begin{align*}%
&1\leq i\leq r: \nu_i=0,\\
 &r+1\leq i\leq s: \mu_i=0,\\
 &s+1\leq i\leq u: \mu_i=1,\nu_i>0,\\
 &u+1\leq i\leq t: \nu_i=1,\mu_i>0,\end{align*}
\begin{align*}&\widetilde M(q_i-1)= M(q_i-1),\,1\leq i\leq r,\\ &\widetilde M(q_i-1)= \bar M(q_i-1),\,r+1\le i\le s,\\ 
&M(q_i-1)= V(q_i-1),\,s+1\leq i\leq u,\\&\bar M(q_i-1)= V(q_i-1),\,u+1\le i\le t.\end{align*} 

\begin{Pro}
\label{confBCD} 
The representation 
$\mathbb C \otimes {\rm adj}$ appears in $\Lambda ^2(\C^{m|n})$ 
with multiplicity 
\begin{equation}\label{ff}
\begin{aligned}
&{\rm mult}(\mathbb C \otimes {\rm adj} )=|\{q_i>1\}| 
\\&+|\{(q_p,q_j)\colon 1\le p< j\leq r ,\, q_p=q_j+2,\,\mu_p=\mu_j=1,\text{ $q_p,q_j$ odd}\}|
\\&+|\{(q_p,q_j)\colon r+1\le p < j\leq s,\, q_p=q_j+2,\,\nu_p=\nu_j=1,\text{ $q_p,q_j$ even}\}|.
\end{aligned}\end{equation}
In particular ${\rm mult}(\mathbb C \otimes {\rm adj} )=1$ 
if and only if $f$ is encoded by pairs
\begin{enumerate}
\item $(q^l,1^r|0), m=ql+r, n=0$, $l$ even if $q$ is even;
\item $(0|q^l,1^r),  n=ql+r, m=0$, $l$ even if $q$ is odd, ($r$ even);
\item  $(q^l|1^n), m=ql$, $l$ even if $q$ is even;
\item  $(q^l,1|1^n), m=ql+1$, $l$ even if $q$ is even;
\item $(1^m|q^l), n=ql$, $l$ even if $q$ is odd;
\item $(1|q^l,1^r)$, $l$ even if $q$ is odd, ($r$ even);
\item $(q,1^r|q^l)$, $q+r=m, ql=n,$  $l$ even,  $q$ is odd;
\item $(q^l|q,1^r)$, $q, l$ even,  ($r$ even);
\item $(q^l,1^r|q)$, $ql+r=m, q=n,$  $l$ even,  $q$ is even;
\item $(q,1|q^l,1^r)$, $q+1=m, ql+r=n,$  $l$ even,  $q$ is even, ($r$ even);
\item $(q^l,1|q,1^r)$, $m=ql+1, n=q+r$, $q,l$ even, ($r$ even);
\item $(q|q^l,1^r)$, $q=m, ql+r=n,$  $l$ even,  $q$ is odd, ($r$ even).
\end{enumerate}
In (1), according to our conventions, we assume $m>4$.
\end{Pro}

\begin{proof}
We have
{\tiny
\begin{align}
\notag
&\Lambda ^2(\C^{m|n})\\
&=\Lambda ^2\left(\sum_{p=1}^r\C^{\mu_p|0}\otimes V_{q_p-1}\oplus\sum_{p=r+1}^s \C^{0|\nu_p}
\otimes V_{q_p-1}
\oplus\sum_{p=s+1}^u\C^{1|\nu_p}\otimes V_{q_p-1}\oplus\sum_{p=u+1}^t \C^{\mu_p | 1}
\otimes V_{q_p-1}
\right)\notag\\
&=\bigoplus_{p=1}^r\Lambda ^2\left(\C^{\mu_p|0}\otimes V_{q_p-1}\right)
\bigoplus_{p=r+1}^s\Lambda ^2\left(\C^{0|\nu_p}\otimes V_{q_p-1}\right)
\bigoplus_{p=s+1}^u\Lambda ^2\left(\C^{\mu_p|1}\otimes V_{q_p-1}\right)
\bigoplus_{p=u+1}^t\Lambda ^2\left(\C^{1|\nu_p}\otimes V_{q_p-1}\right)\label{19}\\
&\bigoplus_{1\le p<j\le r} \left(\C^{\mu_p|0}\otimes\C^{\mu_j|0}\right)\otimes \left(V_{q_p-1}\otimes V_{q_j-1}\right)\oplus
\bigoplus_{r+1\le p<j\le s} \left(\C^{0|\nu_p}\otimes\C^{0|\nu_j}\right)\otimes \left(V_{q_p-1}\otimes V_{q_j-1}\right)
\label{20}\\&\bigoplus_{s+1\le p<j\le u} \left(\C^{1|\nu_p}\otimes\C^{1|\nu_j}\right)\otimes \left(V_{q_p-1}\otimes V_{q_j-1}\right)
\oplus
\bigoplus_{u+1\le p<j\le t} \left(\C^{\mu_p|1}\otimes\C^{\mu_j|1}\right)\otimes \left(V_{q_p-1}\otimes V_{q_j-1}\right)\label{21}\\\
&\bigoplus_{1\le p\le r, r+1\le j\le s} \left(\C^{\mu_p|0}\otimes\C^{0|\nu_j}\right)\otimes \left(V_{q_p-1}\otimes V_{q_j-1}\right)
\bigoplus_{1\le p\le s, s+1\le j\le u} \left(\C^{\mu_p|0}\otimes\C^{1|\nu_j}\right)\otimes \left(V_{q_p-1}\otimes V_{q_j-1}\right)\label{22}\\\
&\bigoplus_{1\le p\le s, u+1\le j\le t} \left(\C^{\mu_p|0}\otimes\C^{\mu_j|1}\right)\otimes \left(V_{q_p-1}\otimes V_{q_j-1}\right)
\bigoplus_{r+1\le p\le s, s+1\le j\le u} \left(\C^{0|\nu_p}\otimes\C^{1|\nu_j}\right)\otimes \left(V_{q_p-1}\otimes V_{q_j-1}\right)\label{23}\\\
&\bigoplus_{r+1\le p\le s, u+1\le j\le t} \left(\C^{0|\nu_p}\otimes\C^{\mu_j|1}\right)\otimes \left(V_{q_p-1}\otimes V_{q_j-1}\right)
\bigoplus_{s+1\le p\le u, u+1\le j\le t} \left(\C^{1|\nu_p}\otimes\C^{\mu_j|1}\right)\otimes \left(V_{q_p-1}\otimes V_{q_j-1}\right).\label{24}
\end{align}}

Recall that as $\mathfrak{gl}(U)\times \mathfrak{gl}(W)$-modules, 
if both $U$ and $W$ have dimension $\geq 2$,
\begin{equation}\label{fff}\Lambda^2(U\otimes W) 
= \Lambda^2(U)\otimes S^2(W)\oplus S^2(U)\otimes \Lambda^2(W).\end{equation}
Using \eqref{fff}, we decompose the summands in  
\eqref{19}.
\begin{equation}
\label{e1}
\Lambda ^2\left(\C^{\mu_p|0}\otimes V_{q_p-1}\right)=\begin{cases}
\Lambda^2(V_{q_p-1})&\text{$\mu_p=1$,}\\ \mathfrak{so}_{\mu_p}&\text{$q_p=1$,}\\
\Lambda ^2(\C^{\mu_p|0})\otimes  S^2 V_{q_p-1}\oplus S^2(\C^{\mu_p|0})\otimes  
\Lambda^2 V_{q_p-1}&\text{$q_p>1,\mu_p>1$.}
\end{cases}
\end{equation}

\begin{equation}
\label{e2}
\Lambda ^2\left(\C^{0|\nu_p}\otimes V_{q_p-1}\right)=\begin{cases}
S^2(V_{q_p-1})&\text{$\nu_p=1$,}\\ \mathfrak{sp}_{\nu_p}&\text{$q_p=1$,}\\
\Lambda ^2(\C^{0|\nu_p})\otimes  S^2 V_{q_p-1}\oplus S^2(\C^{0|\nu_p})\otimes  \Lambda^2 V_{q_p-1}&\text{$q_p>1,\nu_p>1$.}
\end{cases}\end{equation}
\begin{equation}\label{e3}
\Lambda ^2\left(\C^{1|\nu_p}\otimes V_{q_p-1}\right)=
\begin{cases}\mathfrak{osp}_{1|\nu_p}&\text{$q_p=1, \nu_p>1$,}\\
\Lambda ^2(\C^{1|\nu_p})\otimes  S^2 V_{q_p-1}\oplus
S^2(\C^{1|\nu_p})\otimes  \Lambda^2 V_{q_p-1}&\text{$q_p>1, \nu_p>1$.}\end{cases}
\end{equation}
\begin{align}\notag
\Lambda ^2\left(\C^{\mu_p|1}\otimes V_{q_p-1}\right)&=
\Lambda ^2(\C^{\mu_p|1})\otimes  S^2 V_{q_p-1}\oplus
S^2(\C^{\mu_p|1})\otimes  \Lambda^2 V_{q_p-1},\quad q_p>1, \mu_p>1,\\
&=
S^2(\C^{1|\mu_p})\otimes  S^2 V_{q_p-1}\oplus
S^2(\C^{\mu_p|1})\otimes  \Lambda^2 V_{q_p-1},\quad q_p>1, \mu_p>1.\label{e4}\end{align}
Recall that 
\begin{equation}
\label{CG}\Lambda^2(V_p)
=\bigoplus_{i=0}^{\lfloor\frac{p-1}{2}\rfloor}V_{2p-2-4i},\quad S^2(V_p)
=\bigoplus_{i=0}^{\lfloor\frac{p}{2}\rfloor}V_{2p-4i}.
\end{equation}

We now look for  the multiplicity of $\mathbb C \otimes {\rm adj}$ 
in \eqref{e1}--\eqref{e4}.
We look for the contribution of \eqref{e1} given by the first line: 
by \eqref{CG} $\Lambda^2(V_p)$ contains ${\rm adj}$ 
if and only if $p$ is even. 
So the contribuition of the first line  is $|\{q_i>1 
\text{ odd}\colon \mu_i=1\}|=|\{q_i>1\colon \mu_i=1\}|$. 
Indeed, if $q_i$ is even, then $\langle \cdot, \cdot\rangle_{q_i-1}$ is skewsupesymmetric; since $\widetilde L(q_i-1)\subset \C^{m|n}_{\bar 0}$, 
the form $\langle \cdot, \cdot\rangle_{q_i-1}$ is symplectic on 
$\widetilde L(q_i-1)$; in particular, $\mu_i\geq 2$. 
The second line does not contribute.\par
The two terms in the third  line of \eqref{e1}, by \eqref{CG}, 
involve the $V_p$ with $p$ odd and $\mu_p>1$ 
and  the even $p>0$  with $\mu_p>1$, respectively. 
Hence the contribution of  equals $|\{q_i>1\colon \mu_i>1\}|$.
Summing up, the contribution of  \eqref{e1} is exactly 
\begin{equation}
\label{c1}
	|\{q_i>1 \,  :\,  1\le i\le r\}|.
\end{equation}
A similar analysis shows that the  contribution of \eqref{e2}  is
 \begin{equation}
 \label{c2}
	 |\{q_i>1\colon r+1\leq i\leq s\}|.
\end{equation}
Now we deal with \eqref{e3}. 
Note that $\mathfrak{osp}(1|\nu_p)$ does not contain $\rm{adj}$, 
hence the upper line in \eqref{e3} does not contribute. 
The same holds for the first summand in the lower line, 
because $\rm{adj}$ occurs in $S^2(V_{q_p-1})$ only when $q_p-1$ is odd, so 
$\g^\natural(q_p-1)=\mathfrak{spo}_{1|\nu_p}=\{0\}$. 
As for the second summand, $\rm{adj}$ occurs in $\Lambda^2(V_{q_p-1})$ 
only when $q_p-1$ is even,
 so $\g^\natural(q_p-1)=\mathfrak{osp}_{1|\nu_p}$. 
 We have to check whether the trivial representation of 
 $\mathfrak{osp}_{1|\nu_p}$ occurs
 in $S^2(\C^{1|m})$. 
 Using the dimension formula for irreducible finite dimensional representations of 
 $\mathfrak{osp}_{1|m}$, which are always typical, one proves that 
$$
S^2(\C^{1|m})=L_{\mathfrak{osp}(1|m)}(\omega_2)\oplus\C.
$$
In particular the second summand in \eqref{e3} contributes for 
$|\{q_i>1\colon q_i \text{ odd}, s+1\leq i\leq u\}|$. 
Note now that  if $q_p$ is even, then   $\langle \cdot, \cdot \rangle_{q_p-1}$ is skewsupersymmetric, 
hence $\mu_p$ has to be even. 
We conclude that the contribution of \eqref{e3} is
 \begin{equation}
 \label{c3}
 	|\{q_i>1\colon  s+1\leq i\leq u\}|.
\end{equation}
 A  similar analysis for~\eqref{e4} gives the following contribution.
 \begin{equation}
 \label{c4}
 	|\{q_i>1\colon q_i \text{ even}, 
	u+1\leq i\leq t\}|=|\{q_i>1\colon u+1\leq i\leq t\}|.
\end{equation}
For \eqref{20}--\eqref{24} we use the Clebsch-Gordan decomposition
$$
V(q_j-1)\otimes V(q_p-1)=\sum_{k=1}^{\min\{q_p,q_j\}}V(q_p+q_j-2k).
$$
Note that $V(2)$ appears in the r.h.s  either if $q_j=q_p+2$ or if $q_p=q_j+2$, 
so that $q_j$ and $q_p$ are both even or both odd, 
and we cannot have at the same time $\mu_p=\nu_j=1$. 
Looking at \eqref{20}, in the first summand we have 
  $q_p=q_j+2$ and looking for the trivial representation in 
  $\C^{\mu_p|0}\otimes \C^{\mu_j|0}$ forces $\mu_p=\mu_j=1$. 
Since $1\le p< j\leq r$, $q_p-1$  and $q_j-1$ must be even. 
Hence the contribution of  the first summand  is given by  

\begin{equation}
\label{c5}
	|\{(q_p,q_j)\colon 1\le p< j\leq r ,\, q_p=q_j+2,\,
	\mu_p=\mu_j=1,\text{ $q_p,q_j$ odd} \}|.
\end{equation}
A similar analysis for the second summand gives 
 \begin{equation}
 \label{c6}
 	|\{(q_p,q_j)\colon r+1\le p < j\leq s,\, q_p=q_j+2,\,
	\nu_p=\nu_j=1,\text{ $q_p,q_j$ even} \}|.
\end{equation}
The first summand in \eqref{22} does not contribute, 
since $V(2)$ appears in the right tensor factor either if $q_j=q_p+2$ or if $q_p=q_j+2$, 
so that $q_j$ and $q_p$ are both even or both odd, 
and we cannot have at the same time $\mu_p=\nu_j=1$.
Finally, none of the  remaining summands can have the trivial representation on the left tensor factor. \par
Summing up, \eqref{c1}--\eqref{c4} contribute precisely to the r.h.s of the first line of \eqref{ff} 
whereas \eqref{c5}, \eqref{c6} contribute to the second and third line, respectively. \par
It is now clear  that in order  to have multiplicity one, 
there is only one $q_i>1$
contributing $1$ to the first summand in the r.h.s of \eqref{ff}. 
The only partition of this kind which might give contribution also to the second 
or third summand in the r.h.s of \eqref{ff} is $((3,1), (0))$, with $m=4, n=0$ 
but we are excluding $\g=\mathfrak{osp}_{4|0}$ from consideration, since it  is not simple. 
Since $\g^\natural$ has irreducible components which are either simple Lie algebras, 
abelian (even) Lie algebras or $\mathfrak{osp}_{1|n}$, we disregard partitions of the form 
$((q^l, 1^r), (1^n))$ with $n>0, r>0$ and similar ones.
\end{proof}

Now we look for equality of central charges for $\W$-algebras obtained 
from the nilpotent elements involved in \eqref{ff}. 
We choose $(a,b)=\tfrac{1}{2} {\rm Str}(ab)$ as invariant form on $\mathfrak{osp}_{m|n}$. With this choice $h^\vee=m-n-2$. 
The central charge of is computed via  formula (2.6) or Remark 2.2 
from \cite{KacRoaWak03}:
\begin{equation}
\label{CC} 
	c(k,f)={\rm sdim}\,\g_0-\frac{1}{2}{\rm sdim}\, 
	\g_{\frac{1}{2}}-\frac{12}{k+h^{\vee}}|\rho-(k+h^{\vee})x|^2.
\end{equation}
Notably, the direct computation of the conformal levels, 
performed after lengthy computations as in the previous Subsection,  has the following  uniform outcome.

\begin{Pro} 
\label{pro:osp1} 
The conformal levels for $\W^k(\mathfrak{osp},f)$, 
for nilpotent elements $f$ attached to  the pairs partitions $(\alpha |\beta)$ 
listed in (1)--(12), 
with respect to half the supertrace form on $\mathfrak{osp}(m|n)$, are given by 

\begin{equation}
\label{list1}
	k_1=-h^\vee+\frac{h^\vee}{q},\quad 
	k_2=-h^\vee+\frac{h^\vee+1}{q},\quad 
	k_3=-h^\vee+\frac{h^\vee+2}{q+1},\quad
	k_4=-h^\vee+\frac{h^\vee-g+2}{q-1}
\end{equation}
if $q$ is odd and by 
\begin{equation}
\label{list2}
	k_1'=-h^\vee+\frac{h^\vee+2}{q},\quad 
	k_2'=-h^\vee+\frac{h^\vee+1}{q},\quad 
	k_3'=-h^\vee+\frac{h^\vee}{q+1},\quad
	k_4'=-h^\vee+\frac{h^\vee-g}{q-1}
\end{equation}
if $q$ is even. 
$g$ is the difference between the multiplicity of $q$ in $\alpha$ 
and the multiplicity of $q$ in $\beta$ (cf. \eqref{part}).
\end{Pro}
Let us discuss cases (1) from \Cref{confBCD}
The decomposition of $\Lambda ^2(\C^{ql+r|0})$ is 
\begin{align}\label{deccase1}
& \Lambda ^2(\C^{ql+r|0})  = & \\\nonumber 
& \qquad \mathfrak{so}_r\oplus
\left( S^2(\C^{l})\otimes  \Lambda^2 V_{q-1}\right)\oplus \left(\Lambda ^2(\C^l)\otimes  S^2 V_{q-1}\right)\oplus\left((\C^l\otimes\C^r)\otimes V_{q-1}\right).
\end{align}
Let $q$ be even. Then $\g^\natural=\mathfrak{sp}_l\times \mathfrak{so}_r.$  (We set, here and in the following,  $\mathfrak{so}_r=\{0\}$ for $r=0,1$). The conformal weight of the generators are  $1$ (corresponding to $\mathfrak{sp}_l\times\mathfrak{so}_r$), all integers between $2$ and $q$ (given by the intemediate factors) and $\frac{q+1}{2}$. The corresponding highest weights of the $\g^\natural$-modules appearing in the decomposition \eqref{deccase1} are $\lambda_1=\omega_2^{\mathfrak{so}_r}, \lambda_2=2\omega_1^{\mathfrak{sp}_l},\lambda_3=
triv^{\mathfrak{sp}_l}+\omega_2^{\mathfrak{sp}_l},\lambda_4=\omega_1^{\mathfrak{sp}_l}+\omega_1^{\mathfrak{so}_r}$ and the $h_\lambda$ for the second, third and fourth weights are 
\begin{align*}h_{\lambda_2}&=\frac{l+2}{2 + l (2 - q + q^2) -r + q (-2 + k + r)},\\
h_{\lambda_3}&=\frac{l}{2 + l (2 - q + q^2) -r + q (-2 + k + r)},\\h_{\lambda_4}&=\frac{l+1}{2(2 + l (2 - q + q^2) -r + q (-2 + k + r))}+\frac{r-1}{2(k+l q-l+r-2)}.\end{align*}
\begin{itemize}
\item $k=k'_1$. We have
$$h_{\lambda_2}=1,\quad  h_{\lambda_3}=\frac{l}{l+2},\quad  h_{\lambda_4}=\frac{1}{2} \left(\frac{l+1}{l+2}+\frac{q (r-1)}{r}\right).$$
The first two values are less or equal than $1$; the third cannot equal $(q+1)/2$. Hence \Cref{criterion} holds and  $k'_1$ is collapsing.

\item $k=k'_2$. 
$$h_{\lambda_2}=\frac{l+2}{l+1},\quad  h_{\lambda_3}=\frac{l}{l+1},\quad  h_{\lambda_4}=\frac{q+1}{2}.$$
\Cref{criterion} does not apply.
\item  $k=k'_3$.  We have
$$h_{\lambda_2}=\frac{(l+2) (q+1)}{l-r+2},\quad  h_{\lambda_3}=\frac{l
   (q+1)}{l-r+2},\quad  h_{\lambda_4}=\frac{q+1}{2}.$$
\Cref{criterion} does not apply.
   \item $k=k'_4$. We have
  \begin{align*}
&h_{\lambda_2}=\frac{(l+2) (q-1)}{ l
   (q-1)+r-2},\quad h_{\lambda_3}=\frac{l (q-1)}{
   l (q-1)+r-2},\\
&h_{\lambda_4}=\frac{(q-1) (l (q (r-1)-1)+(r-2) r)}{2 (r-2) (l
   (q-1)+r-2)}.
  \end{align*}
\Cref{criterion} applies for partitions different from $(2^2,0|0), (2^2,1|0), (q^2,2|0)$.
   \end{itemize}
  
\par
Let $q$ be odd. Then $\g^\natural=\mathfrak{so}_l\times \mathfrak{so}_r$. 
The highest weights of the corresponding $\g^\natural$-modules in the  decomposition \eqref{deccase1} are $\lambda_1=\omega_2^{\mathfrak{so}_r}, \lambda_2=
triv^{\mathfrak{so}_l}+2\omega_1^{\mathfrak{so}_l},\lambda_3=\omega_2^{\mathfrak{so}_l},\lambda_4=\omega_1^{\mathfrak{so}_l}+\omega_1^{\mathfrak{so}_r}$ and the $h_\lambda$ for the second, third and fourth weights are 
\begin{align*}h_{\lambda_2}&=\frac{l}{(-2 + l + k q + (-1 + q) (-2 + l q + r))},\\h_{\lambda_3}&=\frac{l-2}{-2 + l + 
 k q + (-1 + q) (-2 + l q + r)},\\ h_{\lambda_4}&=\frac{r-1}{
 2 (k-2 - l + l q + r)} + \frac{l-1}{
 2 (-2 + l + k q + (-1 + q) (-2 + l q + r))}.\end{align*}
 When $l=1$ we disregard $\lambda_2, \lambda_3$.
We have to exclude the cases when denominators vanish: 
they are either $l=2$ or $r=2$ for $k_1$,  $l=1$ or $r=1$ for $k_2$, $l=r$ for $k_3$, $r=0$ for $k_4$.
\begin{itemize}
\item 
	$k=k_1$. 
	We have, for $l\ne2$ and $r\ne 2$
	$$h_{\lambda_2}=\frac{l}{l-2},\quad  h_{\lambda_3}=1,\quad  h_{\lambda_4}=\frac{l (q
  	 (r-1)+r-2)-2 q (r-1)-r+2}{2 (l-2)
   	(r-2)}.$$
Since  equals $(q+1)/2$ for $r=2 + 2 q - l q$, so $k_1$ is collapsing 
for $l=1, r\ne q+2$ and for $l\ge 3$.
\item 
	$k=k_2$. 
	For $l=1$ we have $h_{\lambda_4}=q/2$, hence  $k_2$ is collapsing. 
	For $l>1$ we have  $h_{\lambda_4}=(q+1)/2$ hence \Cref{criterion} does not apply.

\item 
	$k=k_3$. 
	We have  $h_{\lambda_4}=(q+1)/2$ hence \Cref{criterion} does not apply.
\item  
	$k=k_4$.  
	We have
\begin{align*}
	&h_{\lambda_2}=\frac{l (q-1)}{l
  	 (q-1)+r},\quad h_{\lambda_3}=\frac{(l-2) (q-1)}{l
  	 (q-1)+r},\\
	&h_{\lambda_4}=\frac{(q-1) (l q
   	(r-1)+l+(r-2) r)}{2 r (l
   	(q-1)+r)}.
\end{align*}
The last value equals  $(q+1)/2$ for non positive values of $r$. 
Hence $k_4$ is collapsing.
\end{itemize}
Performing the above analysis in the remaining cases one proves that \Cref{criterion} never applies to the conformal levels $k_2, k_3, k'_2, k'_3$. Moreover, we have

\begin{Pro}
\label{Pro:concl_Type-BCD}
The conformal levels $k_1, k_4, k'_1, k'_4$ are collapsing in the range specified in \Cref{BCD}.

\end{Pro}
\begin{center}
\begin{landscape}
\begin{table}
\scalebox{0.55}{
\begin{tabular}{|c  ||c | c  | c || c | c | c |}
\hline
 & $k^\natural$& $k_1$ & $k_4$ & $k^\natural$&$k'_1$&$k'_4$\\\hline
  $(q^l,1^r|0)$ & $k^\natural_1=qk+(q-1) (lq-2)+(q-1)r$ &$l=1, r\ne q+2$ or $l\ge 3$ &  \sf{CL}&  $k^\natural_1=\frac{1}{2}(qk+q((q-1)l-2)+(q-1)r)$ & \sf{CL} &$l\ne 2$ or $ l=2 \,{\&}\,  q\ne 2\,{\&}\, r \ne 2$ \\
 & $k^\natural_2=k+(q-1)l$ & & &  $k^\natural_2=k+(q-1)l$ & &  \\\hline
    $(0|q^l,1^r)$ &  $k_1^\natural=\frac{1}{2}(-qk+(ql+2)(q-1)+(q-1)r)$  &   \sf{CL} &   \sf{CL}&   $k^\natural_1=-kq+q((q-1)l+2)+(q-1)r$ & $q\ne r$&  \sf{CL}\\ 
     & $k^\natural_2=-\frac{1}{2}(k-(q-1)l)$ & & &$k^\natural_2=-\frac{1}{2}(k-(q-1)l)$ &  &  \\\hline
   $(q^l|1^n)$ &  $k_1^\natural=qk+(q-1) (lq-2)-(q-1)r$ &   $l\ne 4, n\ne q(l-2)-2$  & $n\notin\{q-1,\frac{(q-1)l}{2}\}$ & $k^\natural_1=\frac{1}{2}(qk+q((q-1)l-2)-(q-1)r)$ & $n\ne q(l+2)$ & $n\ne \frac{(l-1)q-4}{2}$\\
    & $k_2^\natural=-\frac{1}{2}(k+(q-1)l)$ & & $\frac{l-l q}{-3 l
   q+l+n},\frac{(l-2) (q-1)}{l
   (3 q-1)-n}\notin F$ & $k^\natural_2=-\frac{1}{2}(k+(q-1)l)$ &  & $\frac{(2 + l) (-1 + q)}{(-2 - n + l (-1 + q)}\notin F$\\\hline
    $(q^l,1|1^n)$&  $k_1^\natural=qk+(q-1) (lq-2)-(q-1)(r-1)$ &$n\ne q(l-2)-1$ & $n\ne \frac{2(n-1)}{q-1}$ &$k^\natural_1=\frac{1}{2}(qk+q((q-1)l-2)-(q-1)(r-1))$ &  $n\ne q(l+2)+1$& $n\ne q(l-2)-1$\\
     & $k_2^\natural=-\frac{1}{2}((k+(q-1)l)$ & & & $k^\natural_2=-\frac{1}{2}(k+(q-1)l)$ & &  $\frac{(l+2) (q-1)}{l
   (q-1)-n-1},\frac{l(q-1)}{
   l(q-1)-n-1}\notin F$ \\\hline
  $(1^m|q^l)$&  $k_1^\natural=\frac{1}{2}(-qk+(q-1)(ql+2)-(q-1)r)$ & $m\ne q(l+2)+2$ & $l\ne \frac{-1 + 2 m - q}{-1 + q}$ &$k^\natural_1=-kq +q((q-1)l+2)-(q-1)m$ &  $m\ne q(l-2)$&  $m\notin\{q+1, \frac{1}{2}(4+l(q-1)\}$
  \\  & $k_2^\natural=k-l(q-1)$ &
$\frac{(2 + l) (-1 + q)}{-m + l (-1 + q) + q}\notin F$ & & $k^\natural_2=k-(q-1)l$ & & $\frac{l (q-1)}{l (q-1)-m+2},\frac{(l-1)
   (q-1)}{l (q-1)-m+2}\notin F$ \\ &
    & $\frac{l (-1 + q)}{-m +l (-1 + q) + q}\notin F$ & & & & \\\hline
  $(1|q^l,1^r)$&  $k_1^\natural=\frac{1}{2}(-qk+(q-1) (ql+2)+(r-1)(q-1))$ &   \sf{CL}&  \sf{CL}& $k^\natural_1=-k+q((q-1)l+2)+(q-1)(r-1)$ &  $l\ne 3,4$&  \sf{CL}\\
  & $k_2^\natural=-\frac{1}{2}(k-(q-1)l)$ & & &$k^\natural_2=-\frac{1}{2}(k-(q-1)l)$ &  & \\\hline
  $(q,1^r|q^l)$&  $k_1^\natural=\frac{1}{2}(-qk+(q^2-1)(l+1)-(q-1)r$ & \sf{CL}& $l\ne 1+\frac{2r}{q-1}$ & \sf{DNA} & \sf{DNA} & \sf{DNA}\\
   & $k_2^\natural=-k-(q-1)(l-1)$ & &  $\frac{(l+1) (q-1)}{l(q-1)-q-r+1}\notin F$  &  & &\\
   & &  & $\frac{(l-1) (q-1)}{l (q-1)-q-r+1}\notin F$ & & & \\\hline
  $(q^l|q,1^r)$&  \sf{DNA} & \sf{DNA} & \sf{DNA} & $k^\natural_1=\frac{1}{2}(qk+q((q-1)l-q-1)-(q-1)r)$ & \sf{CL} &  $l\ne 1+\frac{2r}{q-1}$ \\ & $ $ && & $k^\natural_2=
  -\frac{1}{2}(k+(q-1)(l-1))$ &  & $-\frac{(l+1) (q-1)}{l
   (-q)+l+q+r+1}\notin F$\\ &  &  & & & & $-\frac{(l-1) (q-1)}{l
   (-q)+l+q+r+1}\notin F$\\\hline
  $(q^l,1^r|q)$& \sf{DNA} & \sf{DNA} & \sf{DNA}  & $k^\natural_1=\frac{1}{2}(qk+q((q-1)l-q-1)-(q-1)(r-1)$ & \sf{CL} &  $\frac{(l+1) (q-1)}{l (q-1)-q+r-1}\notin F$\\
  &  & & &  $k^\natural_2= -\frac{1}{2}(k+(q-1)(l-1))$ & &  $\frac{(l-1) (q-1)}{l
   (q-1)-q+r-1}\notin F$\\\hline
  $(q,1|q^l,1^r)$&  $k_1^\natural=\frac{1}{2}(-k+(q-1)(q(l-1)+2)+(q-1)(r-1))$ & \sf{CL} & $\frac{(l+1) (q-1)}{l
   (q-1)-q+r}\notin F$  & \sf{DNA}& \sf{DNA} & \sf{DNA} \\
   & $k_2^\natural=-\frac{1}{2}(k-(q-1)(l-1))$ & & $\frac{(l-1)
   (q-1)}{l
   (q-1)-q+r}\notin F$ &  & &\\
   & & & $\frac{(q-1) (l
   (q r-1)+r (r-q))}{2 (r-1)
   (l (q-1)-q+r)}\notin F$ & &  &\\\hline
  $(q^l,1|q,1^r)$ &  \sf{DNA} & \sf{DNA} & \sf{DNA} & $k^\natural_1=\frac{1}{2}((qk+q((q-1)l-q-1)-(q-1)(r-1))$ & \sf{CL} & $-\frac{(l-1) (q-1)}{l
   (-q)+l+q+r}\notin F$
   \\ & \sf{DNA} & & &$k^\natural_2=-\frac{1}{2}(k+(q-1)(l-1))$ &  & $-\frac{(l+1)
   (q-1)}{l
   (-q)+l+q+r}\notin F$  \\\hline
  $(q|q^l,1^r)$&  $k_1^\natural=\frac{1}{2}(-qk+((q-1)(q(l-1)+2)+(q-1)r)$ & \sf{CL} & $\frac{(l-1) (q-1)}{l
   (q-1)-q+r+1}\notin F$ & \sf{DNA} &  \sf{DNA} & \sf{DNA}\\
   & $k_2^\natural=-\frac{1}{2}(k-(q-1)(l-1))$ &  & $\frac{(l+1)
   (q-1)}{l(q-1)-q+r+1}\notin F$ &  &  & \\\hline
\end{tabular}}
\captionof{table}
{\footnotesize{Collapsing levels for orthosymplectic types. 
Here $F=\{2,3,\ldots,q\}$.  
{\sf{CL}} means collapsing. 
{\sf{DNA}} means that \Cref{criterion} does not apply. \label{BCD}}}
\end{table}
\end{landscape}
\end{center}
Here we can also conclude that  if $k$  is an admissible conformal level for  which $L_{k^{\natural}} (\g ^{\natural})$ is not quasi-lisse, then  $k$ cannot be collapsing.

Using \Cref{prop:ext} \ref{prop:ext-iii} we can detect many new collapsing levels:
   
\begin{Pro} 
   Assume that $k$ is a conformal level as in \Cref{pro:osp1} 
   such that $\g^{\natural}$ is simple, $k^{\natural}$ is  admissible 
   and $L_{k^{\natural}} (\g^{\natural})$ has only one irreducible ordinary module.
   Then
   $k$ is a collapsing level and $\mathcal W_k(\mathfrak{osp}, f) = L_{k^{\natural}} (\g^{\natural})$.
\end{Pro}
\begin{proof}
By \Cref{prop:ext} \ref{prop:ext-iii} we get that $L_{k^{\natural}} (\g^{\natural})$ embeds into $\W_k(\mathfrak{osp}, f)$. 
Using the complete reducibility at admissible levels we get that
$\W_k(\mathfrak{osp}, f)$ is a direct sum of irreducible $L_{k^{\natural}} (\g^{\natural})$-modules.  
Then, by the assumption of the proposition we get that
$\W_k(\mathfrak{osp}, f)$ is a direct sum of certain copies of 
$L_{k^{\natural}} (\g^{\natural})$. 
Since $\g^{\natural}$ is simple, it is clear that the multiplicity is one. 
The claim follows. 
\end{proof}

  
  \begin{Ex}
Let $\g = \mathfrak{so}_8$, and $f$ corresponds to the partition $(5,1^3)$. 
Then $k_3= - \frac{14}{3}$ is the conformal level appearing in  \eqref{list1}.
 
Now it is natural to ask if this level is collapsing or not. 
One shows that \Cref{criterion}  cannot be applied because there exist a generator $W$  of conformal weight three  such that ${L^{\g^{\natural}}}_0W = 3 W$. 
But the previous corollary still gives that  $k_3 = -\frac{14}{3}$ is a collapsing level 
and $\mathcal W_{k} (\g, f) = L_{-\frac{4}{3}} (\mathfrak{sl}_2) (= L_{-\frac{2}{3}} (\mathfrak{so}_3) )$. 
Here we used that fact that  $L_{-\frac{4}{3}} (\mathfrak{sl}_2)$ 
is boundary admissible with only one irreducible ordinary module.  
  
We should  also note that $L_{k} (\g)$ is quasi-lisse 
with one irreducible ordinary module \cite{AV25}. 
Moreover, one can easily show that one of the 
singular vectors in $V^k(D_4)$ is, after QHR (quantum hamiltonian 
reduction), 
mapped to a generator $W$ of conformal weight three.
\end{Ex}

 This example generalizes to the following important case.
  
  \begin{Co}
     Let $\g = \mathfrak{so}_{q+3}$, where  $q = 4 s +1$, $ s\in {\Z}_{\ge 1}$,    and $f$ corresponds to the either the partition $(q,1^3)$ or to the partition $(q^3,1)$.  Then $k_3$ is  a collapsing level and
  $\mathcal W_{k_3} (\g, f)= L_{- 2 + \frac{2}{2 s+1} } (\mathfrak{sl}_2)$.
  \end{Co}
  \begin{proof}
  The proof follows from  the  previous corollary and the fact that  $k^{\natural} =  -2 + \frac{2}{2 s+1}$ is a boundary admissible   level for $\g^{\natural} = \mathfrak{sl}_2$ (and  therefore $L_{k^{\natural}} (\g^{\natural})$ has only one irreducible  ordinary module). 
  \end{proof}


\section{Examples and conjectures in the exceptional types}
\label{Sec:examples_exceptional}
We explain in this section 
the strategy to get examples of 
(new) collapsing levels or conformal levels 
for the $\W$-algebras associated with 
simple Lie algebras of exceptional types. 

Assume that $\g$ is a simple Lie algebra of 
exceptional type and keep the notation of 
\Cref{Sec:W-algebras}. 
We intend to apply \Cref{criterion} 
relatively to the embedding 
$$V^{k^\natural} \longhookrightarrow \W^{k}(\g,f).$$
By \Cref{Lem:CondA}, 
 condition \ref{crit-i} is always satisfied.  
On the other hand, it is possible 
to compute the central charge for 
each $\W^k(\g,f)$ for $f$ running 
through the finite set of nilpotent orbits. 
So one can easily detect 
the level $k$ for which the condition 
\ref{crit-ii} about the central charges 
is satisfied. 
For such a $k$, 
we check whether 
Condition \ref{crit-iii} holds 
using \Cref{Lem:condition_C}.  
If it is the case, then we 
check 
Condition \ref{crit-iv}: 
the conformal weight $h_i$ of each generator  
is known, and the value $h_{\lam_i}$ 
can be computed using \eqref{eq:h_lami}. 
More precisely, 
we consider the weights given in 
Column~\ref{col-VII} (cf.~Appendix~\ref{Appendix:exceptional}) 
and we compare for each of these weights $\lam_i$, 
the quantity 
$$\dfrac{(\lam_i , \lam_i + 2\rho_\natural)}{
2(k^\natural + h^\vee_\natural)},$$
where $\rho_\natural$ and $h^\vee_\natural$ 
are the half-sum of positive roots and the dual Coxeter number 
for $\g^\natural$,  
with the conformal weights strictly 
greater than $1$, given by Column~\ref{col-VIII} 
in this appendix.

By Part \ref{criterion-i} of \Cref{criterion}, 
if all the conditions \ref{crit-1}--\ref{crit-iv} 
hold, we conclude that $k$ is collapsing, 
and if only the conditions 
\ref{crit-1}--\ref{crit-iii} hold 
we conclude that the above embedding 
is conformal. 
One can use alternatively 
\Cref{Pro:paolo}. 
If furthermore 
$k^\natural$ 
is admissible, we have even a 
finite extension by \Cref{Pro:FE_admissible}.

It may happen that $k$ is collapsing 
even if one of the 
conditions \ref{crit-iii} or \ref{crit-iv} 
is not satisfied. 
In this case, we sometimes use other arguments 
to conclude. For instance, if the level 
is admissible, then we can exploit the results or \cite{AEM}. 

In the examples below, we look for solutions of the equation 
\begin{equation}\label{eqcc}c(k,f)=c(k^\natural)+r,\end{equation}
where $c(k,f)$ is the central charge of $\W^k(\g,f)$,
$c(k^\natural)$
is the central charge of the affine vertex algebra 
$\bigotimes_{i=1}^s V^{k_i^\natural}(\g_i^\natural)$ 
corresponding to the semisimple part of $\g^\natural$, 
and $r$ is the dimension of the center of $\g^\natural$. 
In solving  \eqref{eqcc}
when $r>0$ we are assuming that the Heisenberg algebra which is the affinization of the center of $\g^\natural$ acts nondegenerately.
To complete our analysis, when $r>0$ and $k$ is such that the Heisenberg algebra acts degenerately, we  also check  whether $c(k,f)=c(k^\natural)$.

We also formulate a number of conjectures 
about the associated variety 
of some non admissible affine vertex algebras. 
The general strategy is the following (and some specific 
examples are detailed). 
Assume that a level $k$ 
is collapsing, that is, 
$$\W_k(\g,f) \cong L_{k^\natural}(\g^\natural),$$
and that we know that the associated variety  
of the right-hand side is a nilpotent orbit closure. 
This happens for instance if 
$k^\natural$ is admissible. 
Denote by $H^0_f(L_k(\g))$ 
the 
Drinfeld--Sokolov reduction of $L_k(\g)$ 
corresponding to $f$, 
and $\mc{S}_f$ 
the nilpotent Slodowy slice through $f$. 
Assume that Kac--Wakimoto 
conjecture saying that 
$$\W_k(\g,f) = H^0_f(L_k(\g))$$
if $H^0_f(L_k(\g))$ is nonzero, 
holds. 
Then using the fact that the 
associated variety of $H^0_f(L_k(\g))$ 
is the intersection $X_{L_k(\g)} \cap \mc{S}_f$ 
we can sometimes predict the associated variety 
of $L_k(\g)$.  
In fact, in such a situation, we also 
assume that $X_{L_k(\g)}$ 
is irreducible (\cite[Conjecture 1]{AM18a}), that is, we assume that 
$X_{L_k(\g)}$ is the closure of 
a nilpotent orbit. 

In general, a dimension argument 
is enough to guess the correct variety. 
Remember here that 
$$\dim (\overline{\mb{O}} \cap \mc{S}_f ) =  
\dim \overline{\mb{O}}  - \dim G.f,$$
where $\mb{O}$ is a nilpotent orbit 
and $G.f$ is the nilpotent orbit of $f$. 
Sometimes we need 
to also use the Hasse diagram 
of the exceptional simple Lie algebra $\g$ 
to predict the variety. 

We refer to 
\Cref{Conj:G2-E6a1},  
\Cref{Conj:E6:D5a1},
\Cref{Conj:E6:2A1},
\Cref{Conj:F4-HR},
\Cref{Conj:E7:D5},
\Cref{Conj:E7:A23A1},
\Cref{Conj:E7:A5'}  
for such conjectures. 
We have not listed here all the examples that we can obtain in this way.

\subsection{Case of $G_2$}
We detail here our strategy and results 
for $\g$ of type $G_2$. 
We have $h^\vee=4$. 
There are four nonzero nilpotent orbits which 
are, in the Bala--Carter classification, 
$G_2$ (dimension 12), 
 $G_2(a_1)$ (dimension 10), 
 $\tilde A_1$ (dimension 8) 
 and 
 $A_1$ (dimension 6). 
 The central charges and other useful data 
 are given in \Cref{Tab:Data-G2}. Note that in all cases equation \eqref{eqcc} reads 
$c(k,f) = c(k^\natural)$.

\subsubsection*{$\ast$ $G_2$}
The solutions of the equation 
\eqref{eqcc}
are 
$$k + 4 = 4/7
\quad \text{ and } \quad 
7/12.$$ 
Here $\g^\natural=\{0\}$ 
and \Cref{criterion} applies 
using \Cref{Rem:reductive}. 
We conclude that both levels 
are collapsing. 
In fact, both levels 
$k = -4 + 4/7$ 
and $k = -4 + 7/12$ are admissible 
and we can alternatively 
use \cite{AEM}. 

\subsubsection*{$\ast$ $G_2(a_1)$} 
The solutions of the equation 
\eqref{eqcc}
are 
$$k + 4 = 7/6
\quad \text{ and } \quad 2.$$
Here \Cref{criterion} does not apply 
because the condition \ref{crit-iii} fails. 
But $k= - 4 +7/6$ is admissible 
and we know by \cite{AEM} 
that this level is admissible. 
On the other hand, by \cite{Fasquel-OPE} 
we know that $\W_{-2}(G_2,f) \cong \C$ 
for $f$ in the subregular nilpotent orbit $G_2(a_1)$. 
Hence $k = -4+2=-2$ is collapsing.  

\subsubsection*{$\ast$ $\tilde A_1$}
The solutions of the equation 
\eqref{eqcc}
are 
$$k + 4 = 2/3, \quad 7/6 
\quad \text{ and } \quad 2.$$ 
Here \Cref{criterion} applies, both Conditions \ref{crit-iii} 
and \ref{crit-iv} hold, and 
we conclude that all these levels are collapsing. 
Note that $k = -4+7/6$ is admissible 
and one can also use \cite{AEM}. 

To check Condition \ref{crit-iv}, we proceed as follows. 
In both cases, $\g^\natural$ has type $A_1$. 
In the case where $k=-4+2/3$, 
we have 
$k^\natural= k+3/2= -2+1/6$ from Column \ref{col-II} 
and the weights of $A_1$ to consider are 
$2\varpi_1$ 
and $3\varpi_1$ where 
$\varpi_1 = \rho_\natural$ is the fundamental weight of $A_1$, 
but the quantities 
$$
\dfrac{(2\varpi _1, 4 \varpi_1)}{
2 \times 1/6} = 12, 
\quad 
\dfrac{(2\varpi _1, 5 \varpi_1)}{
2\times 1/6} = 45/2,
$$  
do not appear as conformal weights in the Column \ref{col-VIII}. 
We argue similarly with 
$k=-2$.

\subsubsection*{$\ast$ $A_1$} 
The solutions of the equation 
\eqref{eqcc}
are 
$$k + 4 = 7/3, \quad 8/3  
\quad \text{ and } \quad 5/2.$$ 
Here \Cref{criterion} apply, both Conditions \ref{crit-iii} 
and \ref{crit-iv} hold for $k = -4+ 7/3$ 
and $k = -4+ 8/3$, but only 
Condition \ref{crit-iii} holds for $k = -h^\vee+ 5/2$. 
For the first two one we conclude that 
the level is collapsing. 
For the last one, it a priori gives only a conformal 
embedding. 
But all these levels are admissible and 
we know that for $k = -4+ 5/2$, 
$\W_{k}(\g,f)$ is a finite extension of $L_{k^\natural}(\g^\natural)$ 
(see \cite[Theorem~10.16]{AEM}). 

The necessary data and our conclusions are summarized 
in \Cref{Tab:Data-G2}. 

\begin{Rem}
From \Cref{Tab:Data-G2} 
we have a complete classification of collapsing levels 
for $\g=G_2$. 
\end{Rem}

\subsection{Case of $E_6$}
There are 20 nonzero nilpotent orbits. 
 The central charges and other useful data 
 are given in \Cref{Tab:Data-E6}.
We have $h^\vee=12$. 
 Below, we  detail some examples, 
and formulate a few related conjectures. 
We denote by $X_V$ the associated 
variety of a vertex algebra $V$, 
and we write $\mb{O}_T$ 
for the nilpotent orbit 
of Bala--Carter type $T$. 

\subsubsection*{$\ast$ $E_6$}
The solutions of the equation 
$c(k,f) = c(k^\natural)$ 
are 
$$k+12= 12/13 \quad \text{ and } \quad 13/12.$$
Both give collapsing admissible levels. 

\subsubsection*{$\ast$ $E_6(a_1)$}
The solutions of the equation 
$c(k,f) = c(k^\natural)$ 
are 
$$k+12= 13/9 \quad \text{ and } \quad 9/7.$$
Both  give collapsing levels, 
and the second one is not admissible. 
So we have 
$$\W_{-12+13/9}(E_6,E_6(a_1))\cong \C  
\quad\text{and}\quad 
\W_{-12+9/7}(E_6,E_6(a_1))\cong \C.$$

\begin{Conj} 
\label{Conj:G2-E6a1}
We have 
$X_{L_{-12+9/7}(E_6)} = \overline{\mb{O}_{E_6(a_1)}}$. 
\end{Conj}

\subsubsection*{$\ast$ $D_5$} 
The solutions of the equation 
$c(k,f) = c(k^\natural)+1$ 
are 
$$k+12= 13/8  \quad \text{ and } \quad 8/5.$$
(If $k^\natural=0$ then $c(k,f)\ne 0$).
\Cref{criterion} applies for $k=-12+8/5$. 
The level $-12+13/8$ is admissible and we know 
that $\W_{-12+13/8}(E_6, D_5)$ is lisse. 
On the other hand, $\g^\natural =\C$. 
So $-12+13/8$ cannot be collapsing, otherwise 
$\W_{-12+13/8}(E_6, D_5)$ would be isomorphic 
to $M(1)$ which is not lisse. 
So 
the embedding 
$V^{k^\natural}(\g^\natural) \cong M(1) \hookrightarrow 
\W_{-12+13/8}(E_6, D_5)$ 
is conformal and $k=-12+13/8$ is not collapsing. 
Moreover, by \Cref{prop:ext} \ref{prop:ext-i} 
we conclude that 
$\W_{-12+13/8}(E_6, D_5)$ is a lattice VOA.

\subsubsection*{$\ast$ $E_6(a_3)$} 
The solutions of the equation 
$c(k,f) = c(k^\natural)$ 
are 
$$k+12= 13/6 \quad \text{ and } \quad 2.$$
The level $-12+13/6$ is admissible and collapsing. 
The level $-10$ is not, and Condition \ref{crit-iii} of \Cref{criterion} 
does not hold so we cannot conclude. 

\subsubsection*{$\ast$ $D_5(a_1)$} 
 The unique solution of the equation 
$c(k,f) = c(k^\natural)$ 
is $9/4$, and there are no rational solutions 
of $c(k,f) = c(k^\natural)+1$. 
Condition \ref{crit-iii} of \Cref{criterion} 
does not hold and we cannot conclude. 
Nevertheless, we expect $\W_{-12+9/4}(E_6,D_5(a_1))$ 
to be lisse 
and we formulate a conjecture.  

\begin{Conj} 
\label{Conj:E6:D5a1}
$X_{L_{-12+9/4}(E_6)} = \overline{\mb{O}_{D_5(a_1)}}$.
\end{Conj}

\subsubsection*{$\ast$ $A_5$}
The solutions of the equation 
$c(k,f) = c(k^\natural)$ 
are 
$$k+12= 12/7, \quad 13/6 \quad \text{ and } \quad 9/5.$$
\Cref{criterion} apply for the second levels. 
\Cref{criterion}, Condition \ref{crit-iii}, applies 
for the last one, 
but not Condition~\ref{crit-iv}. 
Indeed, $\g^\natural$ has type $A_1$, 
we have $k^\natural=k+17/2=-2+3/10$ 
and for the weight $\varpi_1 = \rho^\natural$ 
of $A_1$ we get 
$$\dfrac{(\varpi_1, 3 \varpi_1)}{2 \times 3/10} = 5/2 
$$ 
which appears as a conformal weights in Column \ref{col-VII}.  
So Condition~\ref{crit-iv} fails. 
However, since $k^\natural$ is admissible 
we conclude that 
$\W_{-12+9/5} (E_6,A_5)$ 
is a finite extension of 
$L_{-2+3/10}(A_1)$. 

For the first case, we know 
that 
$\W_{-12+12/7} (E_6,A_5) \cong L_{-2+3/14}(A_1) 
\oplus L_{-2+3/14}(A_1;\varpi_1)$ (\cite{AEM}). 
Therefore 
$k$ is collapsing if and only if $k+12=13/6$. 
Moreover, 
$$\W_{-12+13/6} (E_6,A_5) \cong L_{-2+2/3}(A_1).$$

\subsubsection*{$\ast$ $A_4 + A_1$}
The solutions of the equation 
$c(k,f) = c(k^\natural)+1$ 
are 
$$k+12= 39/14  \quad \text{ and } \quad 8/3,$$
and there is no solution for $c(k,f) = c(k^\natural)$. 
Also Condition \ref{crit-iii} of \Cref{criterion} does not hold.  

 \subsubsection*{$\ast$ $D_4$}
The solutions of the equation 
$c(k,f) = c(k^\natural)$ 
are 
$$k+12= 12/7, \quad 13/6 \quad \text{ and } \quad 9/4,$$
and \Cref{criterion} holds for all cases, 
so $k$ is collapsing if and only if 
$k+12=12/7,13/6,9/4$. 
Moreover, 
$$\W_{-12+12/7}(E_6,D_4)\cong L_{-3+3/7}(A_2), 
\quad \W_{-12+13/6}(E_6,D_4)\cong L_{-3+4/3}(A_2),$$
$$\W_{-12+9/4}(E_6,D_4)\cong L_{-3+3/2}(A_2).$$ 
The last isomorphism gives a new quasi-lisse 
$\W$-algebra 
and suggests  
that the intersection of $\overline{\mb{O}_{D_5(a_1)}} $ 
with the Slodowy slice 
associated with $D_4$ is 
the minimal nilpotent orbit closure of $A_2$. 
Indeed, remember that by the case $D_5(a_1)$ 
we expect that $X_{L_{-12+9/4}(E_6)} = \overline{\mb{O}_{D_5(a_1)}}$ 
(see \Cref{Conj:E6:D5a1}). 
This fits the Hasse diagram of $E_6$.

\subsubsection*{$\ast$ $A_4$} 
The solutions of 
the equation 
$c(k,f) = c(k^\natural)+1$ 
are 
$$k+12= 13/5,\quad 9/4 \quad \text{ and } \quad 8/3.$$
\Cref{criterion} applies only for $8/3$, 
and we get 
$$\W_{-12+8/3}(E_6,A_4)\cong L_{-2+2/3}(A_1)\otimes M(1).$$ 
Since $\W_{-12+13/5}(E_6,A_4)$ is quasi-lisse 
because the level is admissible (\cite{Ara09b}), 
it must be a (possibly trivial) conformal extension of 
$L_{-2+3/5}(A_1) \otimes V_{\sqrt{2p}\mathbb Z}$ for some $p \in \mathbb Z_{>0}$ by 
\Cref{prop:ext}~\ref{prop:ext-ii}.

The unique solution of 
the equation is 
$c(k,f) = c(k^\natural)$ 
 $$k+12=12/5,$$
 and \Cref{criterion} applies. 
 So we conclude that 
 $$\W_{-12+12/5}(E_6,A_4)\cong L_{-2+2/5}(A_1).$$

\subsubsection*{$\ast$ $D_4(a_1)$}
The solutions of 
the equation 
$c(k,f) = c(k^\natural)+2$ 
are 
$$k+12= 13/4 \quad \text{ and } \quad 4,$$ 
but \Cref{criterion} does not apply. 
The level $k=-12+13/4$ is admissible 
and we know it cannot be admissible 
since $\W_{-12+13/4}(E_6,D_4(a_1))$ 
is admissible while $M(2)$ is not. 
On the other hand, if $k^\natural=0$ then $c(k,f)\ne 0$. 

\subsubsection*{$\ast$ $A_3+A_1$} 
The equations 
$c(k,f) = c(k^\natural)+1$ and  $c(k,f) = c(k^\natural)$ 
have no rational solutions, 
so we have no collapsing levels.

\subsubsection*{$\ast$ $2 A_2 + A_1$}
The unique rational solution of the equation 
$c(k,f) = c(k^\natural)$ 
is 
$$k+12= 13/3.$$
\Cref{criterion} does not apply but 
this corresponds to an admissible which 
we know it is collapsing.

\subsubsection*{$\ast$ $A_3$}
The solutions of the equation 
$c(k,f) = c(k^\natural)+1$ 
are 
$$k+12= 12/5, \quad 13/4 \quad \text{ and } \quad 4.$$
\Cref{criterion} applies only for $k=-8$ 
and we have 
$$\W_{-8}(E_6,A_3) \cong M(1) \otimes L_{-1}(A_1).$$
On the other hand, we know that 
$\W_{k}(E_6,A_3)$ is quasi-lisse for $k+12=12/5, \, 13/4$. 
Hence we conclude that 
$k$ is collapsing if and only if $k=-8$. 
If $k^\natural=0$, then $c(k,f)\ne c(k^\natural)$.

\subsubsection*{$\ast$ $A_2+2 A_1$}
The solutions of the equation 
$c(k,f) = c(k^\natural)+1$ 
are 
$$k+12= 13/3 \quad \text{ and } \quad 5.$$
The first one corresponds to an admissible level, 
and we know that it cannot give a collapsing level. 
Indeed, if it were the case, 
$\W_{-12+13/3}(E_6, A_2+2 A_1)$ which is quasi-lisse, 
would be isomorphic to $M(1) \otimes L_{-1}(A_1)$, 
a contradiction. 
\Cref{criterion}, Condition \ref{crit-iii}, applies 
but not Condition~\ref{crit-iv} in both cases. 

The equation 
$c(k,f) = c(k^\natural)$ has one 
solution 
$$k+12= 9/2.$$
Here \Cref{criterion} applies. 
We conclude 
that $-12+9/2$ is collapsing 
and we have 
$$\W_{-12+9/2}(E_6,A_2+2 A_1)\cong\C.$$

\subsubsection*{$\ast$ $2A_2$}
The solutions of the equation  
$c(k,f) = c(k^\natural)$ 
are 
$$k+12= 13/3 \quad \text{ and } \quad 3.$$
The level $ k=-12+13/3$ is admissible and we know that 
$k=-12+13/3$ is collapsing. 
For $k=-9$, \Cref{criterion}, Condition~\ref{crit-iii} 
holds but not Condition~\ref{crit-iv}. 
So we conclude that 
the embedding 
$\tilde V_{-3}(G_2) \hookrightarrow 
\W_{-9}(E_6, A_2)$ is conformal 
and we do not know whether the level 
is collapsing.

\subsubsection*{$\ast$ $A_2+A_1$}
The equation 
$c(k,f) = c(k^\natural)+1$ 
has no rational solution, 
and the equation 
$c(k,f) = c(k^\natural)$ has one 
solution 
$$k+12= 9/2.$$
\Cref{criterion} does no apply and we cannot 
conclude.

\subsubsection*{$\ast$ $A_2$}
The solutions of the equation 
$c(k,f) = c(k^\natural)$ 
are 
$$k+12= 13/3,\, 9/2 \quad \text{ and } \quad 6.$$
The first one corresponds to an admissible 
level, and we know it is not collapsing. 
In fact we know 
that 
$\W_{-12+13/3}(E_6,A_2)$ 
is a finite extension 
of $L_{-3+4/3}(A_2) \otimes 
L_{-3+4/3}(A_2)$. 
\Cref{criterion} applies for the last two ones, 
and we conclude that 
$k$ is collapsing if and only if $k=-12+9/2$ 
or $k=-6$.

\subsubsection*{$\ast$ $3A_1$} 
The solutions of the equation 
$c(k,f) = c(k^\natural)$ 
are 
$$k+12= 13/2 \quad \text{ and } \quad 6.$$
\Cref{criterion} applies in both cases 
and 
we have 
$$\W_{-12+13/2}(E_6,3 A_1) \cong L_{1}(A_2) 
\quad \text{ and } \quad
\W_{-6}(E_6,3 A_1) \cong L_{-2+3/2}(A_1).$$ 
Moreover,  
$k$ is collapsing if and only if $k=-12+13/2$ 
or $k=-6$. 

\subsubsection*{$\ast$ $2A_1$}
The solutions of the equation 
$c(k,f) = c(k^\natural)+1$ 
are 
$$k+12= 13/2,\quad 9/2 \quad \text{ and } \quad 8.$$
The first one corresponds to an admissible level, 
and we know that it cannot give a collapsing level. 
Indeed, if it were the case, 
$\W_{-12+13/2}(E_6, 2A_1)$ which is quasi-lisse, 
would be isomorphic to $M(1) \otimes L_{-3/2}(B_3)$, 
a contradiction. 
Only Condition~\ref{crit-iii} \Cref{criterion} applies for 
$k-=12+9/2$ and $k=-4$. 


The unique solution of the equation
$c(k,f) = c(k^\natural)$ 
is  
$$k+12= 6.$$
\Cref{criterion} applies in this case, 
so $k=-6$ is collapsing 
and we conclude that 
$$\W_{-6}(E_6,2A_1) \cong L_{-2}(B_3).$$
In particular, this positively answers part of \cite[Conjecture 10.5]{AEM}.

We know that $X_{L_{-2}(B_3)} = \overline{\mb{O}_{\text{short}}}$, 
where $\mb{O}_{\text{short}}$ is the short nilpotent 
orbit of $B_3$ of dimension 10, see \cite{AM18a}. 
If we believe that the associated variety 
of $\W_{-6}(E_6,2A_1)$ is the intersection 
of that of $L_{-6}(E_6)$ with the Slodowy 
slice associated with $2A_1$, then 
we can formulate a conjecture, 
observing that the orbit $A_2$ is the only 
one in $E_6$ of dimension $42=\dim \mb{O}_{2A_1}+ 10$ 
(we have $\dim \mb{O}_{2A_1}=32$).

\begin{Conj}
\label{Conj:E6:2A1}
We have $X_{L_{-6}(E_6)}= \overline{\mb{O}_{A_2}}$.
\end{Conj}

\subsubsection*{$\ast$ $A_1$}
The solutions of the equation 
$c(k,f) = c(k^\natural)$ 
are 
$$k+12= 13/2,\quad 8 \quad \text{ and } \quad 9.$$
The first one corresponds to an admissible level 
which is not collapsing. 
In fact, we know that 
$\W_{-12+13/2}(E_6,A_1)$ 
is a finite extension 
of $L_{-6+7/2}(A_5)$. 
For the last two ones, \Cref{criterion} apply, 
and we have 
$$\W_{-4}(E_6,A_1) \cong L_{-1}(A_5) 
\quad \text{ and } \quad 
\W_{-3}(E_6,A_1)\cong \C.$$


\subsection{Case of $F_4$} 
Our conclusions are summarised in \Cref{Tab:Data-F4}. 
We proceed as the cases of $G_2$ or $E_6$. 
We omit here the details. 
From the table, we see that our 
list of collapsing levels for $F_4$ is almost 
exhaustive. The only undetermined 
cases is for the orbits $F_4(a_1)$ 
and $C_3(a_1)$ and the common level $-6$. 

We conjecture the following. 

\begin{Conj} 
\label{Conj:F4-HR}
We have 
$\W_{-6}(F_4, F_4(a_3)) \cong \C$ 
and $\W_{-6}(F_4, C_3(a_1)) \cong L_{-1/2}(A_1)$. 
Moreover, 
$$X_{L_{-6}(F_4)}= \overline{\mathbb{O}_{F_4(a_3)}}.$$
\end{Conj}

Note that, from \Cref{Tab:Data-F4} (orbit n$^\circ$7), we have 
$$\W_{-6}(F_4,B_2) \cong L_{-1/2}(A_1) \otimes L_{-1/2}(A_1).$$
In particular, $\W_{-6}(F_4,B_2)$ 
is a new non-admissible quasi-lisse simple $\W$-algebra. 
Then the above two conjectures can be  interpreted by 
Hamiltonian reduction by stages 
between the following nilpotent orbits in $F_4$:

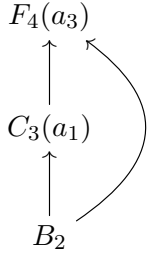
\begin{center}
\hspace{0.25cm}\xymatrix{
F_4(a_3) \\ 
C_3(a_1) 
\ar[u] 
\\ 
\ar[u]  
\ar@/_3pc/[uu]
B_2}
\end{center}

Observe first 
that 
$$H^0_{f}(L_{-1/2}(A_1) \otimes L_{-1/2}(A_1)) 
\cong L_{-1/2}(A_1)
\quad 
\text{ and }
\quad 
H^0_{f}(L_{-1/2}(A_1)) 
\cong \C,$$ 
where $f$ is a nonzero element in $A_1 \times \{0\} \cong \mf{sl}_2$. 
Then if we assume that Kac--Wakimoto conjecture 
holds for the orbits associated to $B_2$ and $F_4(a_3)$ 
for the level $-6$, that is, 
$$H_{B_2}^0(L_{-6}(F_4))  \cong \W_{-6}(F_4, B_2)
\quad 
\text{ and }
\quad  
H_{F_4(a_3)}^0(L_{-6}(F_4))  \cong \W_{-6}(F_4, F_4(a_3)) 
\cong \C,$$
\Cref{Conj:F4-HR} 
can be interpreted as 
$$H_{C_3(a_1)}^0 
( 
H_{B_2}^0(L_{-6}(F_4))
) 
{\cong} 
H^0_{F_4(a_3)} (L_{-6}(F_4)) 
{\cong}  \C.$$


\subsection{Case of $E_7$} 
Our conclusions are summarised in \Cref{Tab:Data-E7}--\Cref{Tab:Data-E7-3}.  
We comment here only a few interesting cases. 

\subsubsection*{$\ast$ $E_6(a_1)$}  
The solutions of the equation 
$c(k,f) = 1$ 
are 
$$k+18= 19/9 \quad \text{ and } \quad 9/4.$$ 
\Cref{criterion} applies to the first case and we get 
$$\W_{-18+19/4}(E_7,E_6(a_1)) \cong M(1).$$
On the other hand, 
$\W_{-18+19/9}(E_7,E_6(a_1))$ is lisse. 
So it is a finite extension of rank one lattice VOA by 
\Cref{prop:ext} \ref{prop:ext-i}.
We conclude that $k$ is collapsing 
if and only if $k=-18+9/4$.

\subsubsection*{$\ast$ $D_5$} 
The solutions of the equation 
$c(k,f) = c(k^\natural)$ 
are 
$$k+18= 19/8, \quad 7/3 \quad \text{ and } \quad 12/5.$$
\Cref{criterion} applies to the second case and we get 
 $$\W_{-18+7/3}(E_7,D_5) \cong L_{-2+1/3}(A_1)\otimes L_{-2+2/3}(A_1).$$
 On the other hand we know by \cite{AEM} that 
 $\W_{-18+19/8}(E_7,D_5)$ is a finite extension 
 of $L_{-2+3/8}(A_1)\otimes L_{-2+3/4}(A_1)$. 

We claim that the non-admissible 
$\W$-algebra 
$\W_{-18+12/5}(E_7,D_5)$ 
is a finite extension of 
the admissible vertex algebra $L_{-2+2/5}(A_1)\otimes L_{-2+4/5}(A_1)$ 
by \Cref{Pro:FE_admissible}. 
Combining with \Cref{Conj:finite_extension} 
(Conjecture 1.3 of \cite{AEM})  
about finite extensions, we formulate a conjecture. 
 
\begin{Conj}
\label{Conj:E7:D5}
We have $X_{L_{-18+12/5}(E_7)}=\overline{\mathbb{O}_{E_7(a_4)}}$.
\end{Conj}

Our expectation comes from the following reasons. 
The associated variety 
of $L_{-2+2/5}(A_1)\otimes L_{-2+4/5}(A_1)$ is $\mc{N}\times \mc{N}$ 
where $\mc{N}$ is the nilpotent cone of $\mf{sl}_2$. 
Next if we assume that the associated variety of $\W_{-18+19/8}(E_7,D_5)$ 
has the same dimension (see \Cref{Conj:finite_extension}), that is four, and 
is the intersection of that of $L_{-18+12/5}(E_7)$ 
with the Slodowy slice of $D_5$, we get 
the expected conjecture looking at the Hasse diagram of $E_7$.

\subsubsection*{$\ast$ $A_2+ 3 A_1$} 
The solutions of the equation 
$c(k,f) = c(k^\natural)$ 
are 
$$k+18= 19/3, \quad 15/2 \quad \text{ and } \quad 6.$$
\Cref{criterion} applies to all cases. 
In particular, with the last one, we get 
$$\W_{-12}(E_7, A_2+ 3 A_1) \cong L_{-2}(G_2).$$
This positively answers one part of \cite[Conjecture 10.10]{AEM} 
and this shows that the non-admissible $\W$-algebra 
$\W_{-12}(E_7, A_2+ 3 A_1)$ is quasi-lisse. 

We know (\cite{ADFLM}) that the associated variety of $L_{-2}(G_2)$ 
is the subregular nilpotent orbit closure of $G_2$. 
Hence we conjecture the following, 
assuming as in the previous conjectures 
that the associated variety of $\W_{-12}(E_7, A_2+ 3 A_1)$ 
is the intersection of that of $L_{-12}(E_7)$ 
with the Slodowy slice of $A_2+ 3 A_1$ 
and using the Hasse diagram of $E_7$. 

\begin{Conj}
\label{Conj:E7:A23A1}
We have
$X_{L_{-12}(E_7)}=\overline{\mathbb{O}_{D_{4}(a_1)}}$.  
\end{Conj}

\subsubsection*{$\ast$ $(A_5)^\prime$} 
The solutions of the equation 
$c(k,f) = c(k^\natural)$ 
are 
$$k+18= 19/6,\quad 18/7, \quad 14/5 \quad \text{ and } \quad 10/3.$$
\Cref{criterion} applies to the first and the last cases, 
and we get collapsing levels. 

For the second one, $k=-18+18/7$ is admissible 
while $k^\natural$ is not for this value. 
There is no contradiction 
with \cite[Conjecture 6.4 and Lemma 6.5]{AEM} 
because the nilpotent Slodowy slice 
$\overline{\mb{O}_{A_6}} \cap \mc{S}_{(A_5)^\prime}$ 
is not collapsing, 
where $\mc{S}_{(A_5)^\prime}$ is the Slodowy slice 
associated with $(A_5)^\prime$, 
that is, it is not isomorphic to a product 
of nilpotent orbits in $A_1\times A_1$. 
For this reason we conclude that $-18+18/7$ is not collapsing 
for $(A_5)^\prime$ because 
the associated variety of $\W_{-18+18/7}(E_7,(A_5)^\prime)$ 
has finitely many nilpotent orbits. 
So if $\W_{18/7-18}(E_7,(A_5)^\prime)$ 
were isomorphic to 
some $L_{k^\natural_1}(A_1) \otimes L_{k^\natural_2}(A_1)$ 
this variety would be contained in the nilpotent 
cone of $A_1 \times A_1$ and this is not possible 
according the Hasse diagram of $E_7$.

Similarly to the case $D_5$ with level $k=-18+12/5$, 
we claim that the non-admissible 
$\W$-algebra $\W_{-18+14/5}(E_7,(A_5)^\prime)$ 
is a finite extension of 
the admissible affine vertex algebra $L_{-2+3/10}(A_1)\otimes L_{-2+2/5}(A_1)$, 
and we formulate a conjecture. 
 
 \begin{Conj}
 \label{Conj:E7:A5'}
 We have $X_{L_{-18+14/5}(E_7)}=\overline{\mathbb{O}_{E_7(a_5)}}$.  
 \end{Conj}


\subsection{Case of $E_8$} 
Our conclusions are summarised in Tables~\ref{Tab:Data-E8}--\ref{Tab:Data-E8-6}.  
It works similarly to the case of $E_7$ 
and we omit the details.

\section{  Griess algebra inside  $\W^k(\g,f)$}
\label{Griess}
In this section we investigate the structure of vertex algebra $ V=\Com (V^k(\g^{\natural}), \W^k(\g,f))$ and show that the weight-two space
$V_2$ generates the Griess algebra. We conjecture that $V_2$ is always semi-simple and associative, which would then imply that $V_2$ generates a vertex subalgebra of $V$ isomorphic to a tensor product of  copies of Virasoro vertex algebras. Then we get that $V^k(\g^{\natural})$ conformally embeds into $\W_k(\g,f)$ if and only if all central charges of all Virasoro components  vanish. We present some examples.


\subsection{The structure of $\Com (V^k(\g^{\natural}), \W^k(\g,f))$}
Let $V$ be a VOA with conformal vector $\omega$ such that 
\begin{enumerate}
\item $V = \oplus_{n \in {\Z} }V_n$, ${\omega_0}_{\vert V_n} \equiv n \mbox{\text{Id}}$;
\item $V_0 = {\C} {\bf 1}$;
\item $V_1 =\{0\}$.
\end{enumerate}
Then $V_2$ is a commutative algebra with product, 
$$ a \cdot b = a _{(1)} b, \quad a,b \in V_2$$ and an inner product given by $(a,b) = a_{(3)} b$. $V_2$ is called the Griess algebra.

The following result was proved 
in \cite[Lemma 5.1]{Miyamoto}. 
 
\begin{lemma} 
\label{Lem:Miyamoto}
Let $V$ be any VOA satisfying assumptions (1)-(3) above.
   Assume that $e/2$ is a idempotent in Griess algebra $V_2$. Then $e$ is a conformal vector of central charge $c= 2 (e, e)$.
\end{lemma}

The next result is due to 
\cite{Lam}:
 
 \begin{Pro} 
 \label{Pro:Lam}
 The algebra $V_2$ is semi-simple and associative 
 if and only only if $V_2$ generates a subalgebra which is a tensor product of Virasoro VOAs. 
 \end{Pro}

Let $V = \Com (V^k(\g^{\natural}), \W^k(\g,f))$ with conformal vector $ \overline L = L-L^{\g^{\natural}}$. Assume that $k^{\natural}$ is not critical and $\g^{\natural}$ is a Lie algebra. 
Note that $V$ satisfies assumptions (1)-(3). Denote by $V^{Vir}_c$ the universal Virasoro vertex algebra of central charge $c$, and $L^{Vir}_c$ its simple quotient. 
Using \Cref{Lem:Miyamoto} and \Cref{Pro:Lam} we obtain:


\begin{Pro}
Assume that
$$
V_2=\bigoplus_{i=1}^s \mathbb C e_i,\quad
e_i\cdot e_j=\delta_{ij}e_i,\quad (e_i)_{(0)} e_j =0, \quad
\bar L=2\sum_{i=1}^s e_i.
$$
Set \(c_i=8(e_i,e_i)\).  Then $V_2$ generates the Virasoro subalgebra of $V$ isomorphic to
$$
V^{\mathrm{Vir}}_{c_1}\otimes\cdots\otimes
V^{\mathrm{Vir}}_{c_s}.
$$
\end{Pro}

 Assume that $\dim V_2 \geq 2$. We expect that $V_2$ is semisimple and associative as Griess algebra. This could imply that $\W^k(\g,f)$ has the affine-Virasoro vertex subalgebra:
$$ U = V^{k^{\natural}} (\g^{\natural}) \otimes V^{Vir}_{c_1} \otimes \cdots \otimes V^{Vir} _{c_s}. $$
If any of central charges $c_i$ is non-zero, then the Virasoro vector $ \omega^{(i)}$ of $ V^{Vir}_{c_i}$ is non-zero in the quotient $\W_k(\g,f) $, so we cannot have conformal embedding of $L_{k^{\natural}} (\g^{\natural})$ in $\W_k(\g,f)$.

Therefore the necessary condition for conformal embedding is that each central charges $c_i$ vanishes. This is probably very exceptional.
On the other one can investigate conformal embeddings of $L_{k^{\natural}}(\g^{\natural}) \otimes L^{Vir}_{c_1} \otimes \cdots \otimes L^{Vir} _{c_s}$ in $\W^k(\g,f)$.

\begin{Conj} \label{slutnja-1} Assume that $\g $ is a simple Lie algebra, and let $f$ be a nilpotent element. 
 Set $V = \Com (V^k(\g^{\natural}), \W^k(\g,f))$.  Then the Griess algebra  $V_2$ is semisimple and associative.
\end{Conj}

\subsection{Example: affine $\W$-algebra $\W^k(\mf{sl}_6, f_{4,2})$ }
We shall see that \Cref{slutnja-1}   holds in the case  $\g = \mf{sl}_6$, and $f$ is encoded  by the partition $(4,2)$.
 We shall denote it by  $f= f_{4,2}$.   
Then  the central charge is 
$c(k,f) = \frac{35 k}{k+6} - 66 k -194$ 
and $\W^k(\g,f_{4,2})$ is strongly generated by
\begin{itemize}
\item the Heisenberg field $J$ of level $4 (k+ 9/2)$. 
Let $M(k)$ be the corresponding Heisenberg vertex algebra 
which has central charge $1$ for $k \ne -9/2$, 
\item 
the Virasoro field $L$, three fields $G^{+}, G^{0}, G^-$ of conformal weight $2$ such that
$$ J(0) G^0 =0,  \ J(0) G^{\pm } = a^{\pm} G^{\pm}, \quad a^{\pm} \ne 0,$$
(we write the state-field correspondence as $Y(X,z)=\sum_{n\in \mathbb Z} X(n)z^{-n-m}$ if $X\in V_m$).
\item 
the fields  $W^{+}, W^{0}, W^-$ of conformal weight three  such that
$$ J(0) W^0 =0,  \ J(0) W^{\pm } = b^{\pm} W^{\pm}, \quad b^{\pm} \ne 0. $$
\item a field \(W_4\) of conformal weight four such that
$J(0)W_4=0$.
\end{itemize}

We will see that in these cases our criterion fails, but there exist 
 two alternative  methods for studying conformal embeddings. First method is of course using explicit OPE calculation. 
 The second is based on new results on partial Quantum Hamiltonian Reduction. 

 The following result has been communicated to one of us by S. Nakatsuka.
\begin{lemma} 
The generators $J, L, G^0$ satisfy the OPEs:
\begin{align*}
& J \mbox{ is the Heisenberg vector} , &\nonumber \\
& L   \mbox{ conformal-Virasoro generator of central charge} \ c_k, &\nonumber    \\
& G^0  \mbox{ Virasoro vector of central charge} \ d_k = -  \frac{(3k + 13)(4k + 17)(5k + 24)}{
(k + 5)(k + 6)} , &\nonumber \\
& J(n) G^0 = 0  \quad \mbox{for} \ n \geq 0, & \nonumber \\
& L(z) G^0 (w) \sim  \frac{ d_k /2 }{(z-w) ^4}  
+ \frac{ 2 G^0} {(z-w) ^2} + \frac{ \partial G^0}{z-w}, & \nonumber \\
& L(2) G^0 =   \frac{d_k}{2}  {\bf 1}. & \nonumber  
  \end{align*}
  \end{lemma}
  
  \begin{Pro} 
  The following statements hold:
  \begin{enumerate} [{\rm (1)}]
  \item the embedding $M(k) \hookrightarrow \W_k(\mf{sl}_6,f_{4,2})$ 
  is conformal if and only if $k =-13/3$;
  \item the embedding $M(k) \hookrightarrow \W_k(\mf{sl}_6,f_{4,2})$ is  not conformal  for $k= - 45/11$.  
  \end{enumerate}
  \end{Pro}
  
  \begin{proof}
  Note that $L(2) G^0= 0$ for $k=-13/3$. Hence  the condition in \cref{cprimo} holds and the embedding is conformal by  \cite[Theorem 1.1]{AMP23}.
   
  If $k= -45/11$, $L(2) G^0 \ne 0$, and therefore the criterion does not work. 
  The embedding is not conformal.
  \end{proof}
  
  Now we consider the case $k=-45/11$. 
  Let $\bar L = L- L^{\mathcal H}$, 
  where $L^{\mathcal H}$ is the Virasoro vector in the Heisenberg vertex algebra $M(k)$.
   
  \begin{Pro} Assume that $k=-45/11$.  Then
  \begin{itemize}
  \item The Griess algebra generated by $\bar L, G^0$ is semi-simple and it generates a 
  subalgebra of  $\W^k(\mathfrak{sl}_6, f_{4,2})$ isomorphic to
  $ V^{Vir}_{c_1} \otimes V^{Vir} _{c_2}$ where
  $$ c_1=- \frac{52}{55}, c_2 = \frac{52}{55}. $$ 
  \item The simple vertex algebra $\W_k(\mathfrak{sl}_6, f_{4,2})$ contains a non-zero  vertex  subalgebra isomorphic to 
  $M(k) \otimes \widetilde  L^{Vir}_{c_1} \otimes \widetilde  L^{Vir}_ {c_2}$;
  where $\widetilde  L^{Vir}_{c_i}$ are  certain quotients (possibly simple) of $  V^{Vir} _{c_i}$, $i=1,2$.
  \end{itemize}
  \end{Pro}
  

 We can also apply result from \cite{JG} which could avoid explicit OPE formulas. We  just discuss the case of  $\W^k(\mf{sl}_{n+2}, f_{n,2})$, where $f_{n,2}$ is encoded by the partition $(n,2)$. The main result of \cite{JG} shows that there is a QHR of Virasoro type, such that
$$ H_{f_0} ^0 (\W^k(\mf{sl}_{n+2}, f_{n,1,1})) 
= \W^k (\mf{sl}_{n+2}, f_{n,2}). $$
Since $H_{f_0} ^{0}$ acts on the affine  $\mathfrak{sl}_2$ subalgebra of 
$\W^k(\mf{sl}_{n+2}, f_{n,1,1})$, 
this  implies that there exists   a generator of $\W^k (\mathfrak{sl}_{n+2}, f_{n,2})$  which is a Virasoro vector of central charge $c_{k^{1}} = 1 - \frac{ 6 (k ^{1} +1)^2}{k^{1} +2} $.
(Here $k^{1} = k + n-1$ is the level of the $\mf{sl}_2$ subalgebra in  
$\W^k(\mf{sl}_{n+2}, f_{n,1,1})$.)
In particular, \Cref{slutnja-1} holds in this case (and probably in general).

We can also apply this for detection of new conformal/collapsing levels. So assume that: 
\begin{enumerate} [{\rm (1)}]
\item 
\label{Greiss1}
the central charge of 
$\W^k (\mf{sl}_{n+2}, f_{n,2})$ is $1$ 
(= the central charge of the affine subalgebra of 
$\W^k (\mf{sl}_{n+2}, f_{n,2})$), 
\item 
\label{Greiss2} 
$c_{k^{1}}  =0$, i.e., $k^1 \in \{-1/2, -4/3\}$.
\end{enumerate}
Then we can apply our criterion for conformal 
embeddings.

Let us just mention that in the case $\g = \mf{sl}_{6}$, $f= f_{4,2}$,  the central charge of 
$\W^k (\mf{sl}_{6}, f_{4,2})$ is $1$ if and only if $k \in \{-\frac{13}{3}, - \frac{45}{11}\}$. 
But \ref{Greiss2} holds only for $k=-13/3$. 
So $k=-13/3$ is conformal level, but $k=-\frac{45}{11}$ is not. 
One can show also that $k=-13/3$ is collapsing by further applications of our criterions.

\section{The examples for exceptional Lie superalgebras}

We now discuss the simple exceptional basic classical Lie superalgebras. There are three cases to consider and we will always start by giving the decomposition of the Lie superalgebra as a module for the even subalgebra and we then discuss the $\mathcal W$-algebras corresponding to even nilpotent elements.

\subsection{The example of $D(2, 1; \alpha)$}

The even subalgebra of $D(2, 1; \alpha)$ is $\mathfrak{sl}_2 \oplus \mathfrak{sl}_2 \oplus \mathfrak{sl}_2$. 
For details on the commutation relations see \cite{CG}. 
We normalise the invariant bilinear form on $D(2, 1; \alpha)$ such that the first $\mathfrak{sl}_2$ has embedding index one, the second one $-\alpha^{-1}-1$ and the third one $-\alpha - 1$. 

\subsubsection*{$\ast$ $(f, 0, 0)$}

We start with the minimal nilpotent element. In this case $c = -6k-3$ 
and the two affine $\mathfrak{sl}_2$-vertex subalgebras have levels 
$\ell_1 = - k(\alpha^{-1}+1)-1$ and $\ell_2 = - k(\alpha+1)-1$. 
We get conformal levels for 
\begin{enumerate}
\item $k = \frac{1}{2}$, 
\item $k = -(1+\alpha)^{-1}$, $\ell_1 = \alpha^{-1}-1, \ell_2=0$, 
\item $k = -(1+\alpha^{-1})^{-1}$, $\ell_1 =0, \ell_2 = \alpha-1$. 
\end{enumerate}
The first case is non-collapsing and studied in detail in \cite{CGL}. Note that cases (2) and (3) are related by an outer automorphism of $\W^k(D(2,1; \alpha), f)$ which acts on the parameters by 
$k \mapsto k$ and $\alpha \mapsto \frac{1}{\alpha}$, and on generators (in the notation of \cite{CGL}) as follows:
$$ L\mapsto L, \qquad u \mapsto u' \ \text{and}\ u' \mapsto u,\ \text{for}\ u = e,f,h,\qquad  G^{\pm\pm} \mapsto G^{\pm\pm},\qquad G^{\pm \mp} \mapsto -G^{\mp \pm}.$$

We claim that case (2) is collapsing for all $\alpha \neq -1$. 
\Cref{criterion} applies for generic values of $\alpha$ to conclude that these levels are collapsing. We cannot apply it in the case where $\alpha$ is a rational number in the interval $(-1,0)$, but we can still see that these are collapsing levels as follows. 
By Proposition 2.2 of \cite{KL19}, it suffices to show that
\begin{enumerate}
\item The Virasoro element for the affine coset lies in the radical of the bilinear from $\langle u,v \rangle_2 = u_{(3)} v$ on the weight $2$ subspace,
\item The bilinear from $\langle u,v\rangle_{3/2} = u_{(2)} v$ on the weight $\frac{3}{2}$ space is degenerate.
\end{enumerate}
But this is clear since it holds generically and is a closed condition.


\subsubsection*{$\ast$ $(f, f, 0)$}

In this case 
\[
c = 1 + 6k\alpha^{-1}
\]
and $\mathfrak{g}^\natural \cong \mathfrak{osp}_{1|2}$. The level of the affine vertex subalgebra is $\ell = -k(\alpha+1) -1$ and we get conformal levels for
\begin{enumerate} 
\item $k = \frac{2\alpha+3}{4\alpha+4}$, $\ell = -\frac{2\alpha+ 7}{4}$,
\item $k = \frac{2\alpha+1}{4\alpha+4}$, $\ell = \frac{2\alpha- 5}{4}$. 
\end{enumerate}
\Cref{criterion} shows that these are collapsing for generic $\alpha$, and by the above argument they are collapsing for all $\alpha \neq -1$.

\subsection{The example of $G(3)$}

The even subalgebra of $G(3)$ is $G_2 \oplus \mathfrak{sl}_2$. 
We normalise the Killing form, such that the long roots of $G_2$ have norm two. Then the embedding index of $G_2$ in $G(3)$ is one and the one of $\mathfrak{sl}_2$ is $-\frac{3}{4}$.
The Lie superalgebra $G(3)$ decomposes as 
\[ 
G(3) = G_2 \otimes \mathbb C \oplus  \mathbb C \otimes \mathfrak{sl}_2 \oplus  \rho_{\omega_1} \otimes \rho_\omega.
\]
The affine vertex superalgebra $V^k(G(3))$ then has affine subalgebra $V^k(G_2) \otimes V^\ell(\mathfrak{sl}_2)$ with $\ell = -\frac{3}{4}k$. The central charge of $V^k(G(3))$ is $\frac{3k}{k+2}$.

\subsubsection*{$\ast$ $(0, f)$} 

Firstly we consider the nilpotent element that is $0$ in the  $G_2$-subalgebra and principal in the $\mathfrak{sl}_2$-subalgebra. 
The central charge of the $\mathcal W$-algebra is
\[
c(k) = \frac{3k}{k+2} + \frac{9}{2}k -2 -\frac{7}{2}.
\]
$\mathfrak{g}^\natural \cong \mathfrak g_2$ and the level of the affine subalgebra of  $\g_2$ is $k -1$ so that we get a conformal embedding if
\[
c(k) = \frac{14(k-1)}{k+3}. 
\]
The solutions are 
\begin{enumerate}
\item $k = -\frac{5}{3}$,
\item $k= -\frac{2}{3}$,
\item $k=1$.
\end{enumerate}
\Cref{criterion} applies in all cases, so all three are collapsing levels. 

\subsubsection*{$\ast$ $(G_2, 0)$} 

Next we consider the principal embedding of $\mathfrak{sl}_2$ in $\g_2$ 
and trivial one in $\mathfrak{sl}_2$. 
The central charge is
\[
c(k) = 
\frac{3k}{k+2} -168k -288.
\]
$\mathfrak g^\natural \cong \mathfrak{sl}_2$ 
and the level of the affine $\mathfrak{sl}_2$ subalgebra is $\ell - 3$. We thus need to solve
\[
c(k) =  \frac{3 (\frac{3}{4}k +3)}{\frac{3}{4}k +1}.
\]
The solutions are 
\begin{enumerate}
\item $k = -\frac{11}{6}$ and $\ell -3 = -\frac{13}{8}$,
\item $k =-\frac{12}{7}$ and $\ell -3 = -\frac{12}{7}$,
\item $k= -\frac{3}{2}$ and $\ell -3 = -\frac{15}{8}$.
\end{enumerate}
\Cref{criterion} applies in all cases, so all three are collapsing levels. 

\subsubsection*{$\ast$ $(G_2(a_1), 0)$}

The next one we consider is the nilpotent element that is of type $G_2(a_1)$ in the $\mathfrak g_2$-subalgebra and zero in the $\mathfrak{sl}_2$-subalgebra. 
In this case  $\mathfrak g^\natural \cong \mathfrak{osp}_{1|2}$, 
 the central charge is 
\[
c(k) = 
\frac{3k}{k+2} -24 k   -26
\]
and the level of the affine $\mathfrak{osp}_{1|2}$  subalgebra is $\ell -2$.
We need to solve 
\[
c(k) = \frac{(\ell-2)}{\ell - \frac{1}{2}}. 
\]
The solutions are 
\begin{enumerate}
\item $k = -\frac{5}{3}$ and $\ell -2 = -\frac{3}{4}$,
\item $k=-1$ and $\ell -2 = -\frac{5}{4}$.
\end{enumerate}
\Cref{criterion} applies to the first case, so it is a collapsing levels. However for the second criterion condition \ref{crit-iv} of \Cref{criterion} fails.

\subsubsection*{$\ast$ $(\tilde A_1, 0)$}

The next one we consider is the nilpotent element that is of type $\tilde A_1$ in the $\mathfrak g_2$-subalgebra and zero in the $\mathfrak{sl}_2$-subalgebra. 
The central charge is 
\[
c(k) = 
 \frac{3k}{k+2} - 18 k -21 
\]
and 
$\mathfrak g^\sharp \cong \mathfrak{sl}_2 \oplus  \mathfrak{sl}_2$.
The first $\mathfrak{sl}_2$ has level $k$ which gets shifted to $k + \frac{1}{2}$ and the second one has level $\ell$ which gets shifted to $\ell  -2$ so that we need to solve for
\[
c(k) = \frac{3(k+\frac{1}{2})}{k+\frac{5}{2}} + \frac{3(\ell -2)}{\ell}.
\]
The solutions are
\begin{enumerate}
\item $k = -\frac{8}{3}$ and $k+\frac{1}{2} = -\frac{13}{6}$ and $\ell -2 = 0 $,
\item $k =-\frac{5}{3}$ and $k+\frac{1}{2} = -\frac{7}{6}$ and $\ell -2 = -\frac{3}{4} $,
\item  $k=-1$ and $k+\frac{1}{2} =  -\frac{1}{2}$ and $\ell -2 =  -\frac{5}{4}$,
\item $k=-\frac{1}{2}$ and $k+\frac{1}{2} = 0$ and $\ell -2 =  -\frac{13}{8}$.
\end{enumerate}
In the first case, Condition \ref{crit-1} of \Cref{criterion} fails.
In the second, third and fourth cases, Condition \ref{crit-iv} of \Cref{criterion} fails, 
so the embeddings are conformal but maybe not collapsing.

\subsection{The example of the superalgebra $F(4)$}

\[
F(4) = \mathfrak{so}_7 \oplus \mathfrak{sl}_2 
\oplus \rho_{\omega_3} \otimes \rho_{\omega}.
 \]
Let us also decompose as a $G_2$-module, then 
\[
F(4) = G_2 \oplus \mathfrak{osp}_{1|2} 
\oplus \rho_{\omega_1} \otimes \rho_{\omega}.
 \]
 
 \subsubsection*{$\ast$ $((7), 0)$}
 
In this case the central charge is 
\[
c(k) 
= \frac{8k}{k+3} - 168k  -390. 
\]
The level of the affine $\mathfrak g^\sharp \cong \mathfrak{osp}_{1|2}$  subalgebra is $\ell = -\frac{2}{3}k$ and the level shift is $-3$ so that we need all solutions of 
\[
c(k) = \frac{\ell -3}{\ell - \frac{3}{2}}
\]
The solutions are 
\begin{enumerate}
\item $k = -\frac{21}{8}$ and $\ell -3 = -\frac{5}{4} $,
\item $k=- \frac{18}{7}$ and $\ell -3 = -\frac{9}{7}$,
\item $k= - \frac{7}{3}$ and $\ell -3 = -\frac{13}{9} $.
\end{enumerate}
\Cref{criterion} applies to all cases, so all three are collapsing. 

\subsubsection*{$\ast$ $((5, 1^2), 0)$}

In this case the affine subalgebra is a Heisenberg VOA (at level $k+1$ 
if normalised in such a way that the charged strong generators have weight $\pm 1$)
times an affine $\mathfrak{sl}_2$ at level $-\frac{2}{3}k - 3 $.
The central charge is
\[
c = \frac{8k}{k+3} -60k -104 -4 \cdot 14 + 42
\]
and we have to solve 
\[
 \frac{8k}{k+3} -60k -118 = 4 + \frac{18}{2k+3}
\]
with solutions 
\begin{enumerate}
\item $k = -\frac{12}{5}$ 
\item $k =-2$
\end{enumerate}
\Cref{criterion} applies to all cases, so all three are collapsing. 

\subsubsection*{$\ast$ $((3, 1^4), 0)$}

In this case 
\[
c = \frac{8k}{k+3} - 12k -14
\]
and the affine $\mathfrak{so}_4$ has level $k+1$ while the affine $\mathfrak{sl}_2$
has level $\ell = -\frac{2}{3}k -2$
and we have to solve
\[
\frac{8k}{k+3} - 12k -14 = \frac{6(k+1)}{k+3} + \frac{3\ell }{\ell +2}
\]
with solutions
\begin{enumerate}
\item $k = - 1$ $\ell =  - \frac{4}{3}$ 
\item $k =- \frac{9}{4}$, $\ell = -\frac{1}{2}$
\end{enumerate}
\Cref{criterion} applies to the first  case, but \ref{crit-iv} of \Cref{criterion} fails in the second case, i.e., in the second case we only know that there is a conformal embedding.

\subsubsection*{$\ast$ $((1^7), 2)$}

In this case 
\[
c = \frac{8k}{k+3} + 4k - 6, 
\]
the affine $\mathfrak{so}_7$ has level $k - 1$ 
and we have to solve
\[
\frac{8k}{k+3} + 4k - 6 = \frac{21(k-1)}{k+4}
\]
with solutions
\begin{enumerate}
\item $k = -\frac{9}{4}$, 
\item $k =1$,
\item $k=-1$.
\end{enumerate}
\Cref{criterion} applies to all cases, so all three are collapsing. 

%
%
%
%
\appendix
\section{Tables in the exceptional types}
\label{Appendix:exceptional}
Assume in this appendix that $\g$ is simple of exceptional type. 
Let $\mf{s}$ be the simple Lie algebra generated by an $\mf{sl}_2$-triple 
$(e,h,f)$ 
and $\g^\natural$ the centraliser in $\g$ of $\mf{s}$. 
Write 
$\mf{k}=\bigoplus_{i=0}^r \mf{k}_i$ 
the decomposition into simple factors of 
the reductive Lie algebra 
$\mf{k}:=\g^\natural \oplus \mf{s}$. 
In this notation,  
the center is $\mf{k}_0=\g^\natural_0$,  
$ \mf{s}=\mf{k}_r$, the other factors are the simple factors of $\g^\natural$. 
Tables~\ref{Tab:Data-E6}--\ref{Tab:Data-E8-6} provide 
the necessary data to verify whether \Cref{criterion} applies 
and our conclusions about $\W_k(\g,f)$. 
Let us explain the tables by column:\\ 

\begin{enumerate}[{\rm I}]

\item 
\label{col-I}
attributes  
a number to each nilpotent orbit (following the numbering of \texttt{GAP4}). \\

\item 
\label{col-II} describe $G.f$ in 
the Bala--Carter classification. \\

\item 
\label{col-III} gives the dimension of $G.f$.  \\

\item 
\label{col-IV}
indicates the semisimple type of $\g^\natural$, and the center if there is one.  \\

\item 
\label{col-V} gives 
the values of the $k_i^\natural$'s.  
When $\g^\natural$ has a center of dimension one, 
$k_0^\natural$ is defined (up to nonzero scalar) by: 
$$k_0^\natural = (x|x)_\g  k + \dfrac{1}{2} (\kappa_\g(x,x)
- \kappa_{\g^0}(x,x)-\kappa_{\g^{1/2}}(x,x))$$
for a fixed nonzero $x \in \g_0^\natural$. 
The order of the values of the $k_i^\natural$'s respects 
the order of the factors in Column \ref{col-IV}. 
If $\g^\natural$ has a center of dimension two (which occurs only once), 
we give the value of the above right-hand side for a nonzero 
element, and we treat separately this case. 
In particular, when there is a center, the coefficient 
in $k$ in the first line always gives the value of 
$ (x|x)_\g$ for our choice of $x \in \g_0^\natural \setminus\{0\}$. 
\\ 

\item 
\label{col-VI}
gives the central charge $c(k,f)$ of the $\W$-algebra 
$\W^k(\g,f)$ in a factorized form.  \\

\item 
\label{col-VII}
indicates the decomposition of $\g$ into simple $\mf{k}$-modules as follows: 
the highest weights of an irreducible $\mf{k}$-module in $\g$ 
with multiplicity $\mu$ are given 
in brackets by 
\begin{align*}
[\lam^{(1)}_1,\cdots,\lam^{(1)}_{\ell_1}; \ldots ; 
\lam^{(r-1)}_1,\cdots,\lam^{(r-1)}_{\ell_{r-1}}
; \lam_r ]^{\mu},
\end{align*}
which  means that $ \g$ is a sum of the $\mf{k}$-modules 
\begin{align*}
&  \left( L_{\mf{k_1}}(\lam^{(1)}_1 \varpi_1 + \cdots + \lam^{(1)}_{\ell_1} 
 \varpi_{\ell_1}) 
 \right) 
\otimes  & \\
& \qquad \cdots \otimes  
\left( L_{\mf{k_{r-1}}}(\lam^{(r-1)}_1 \varpi_1 + \cdots + 
\lam^{(r-1)}_{\ell_{r-1}} \varpi_{\ell_{r-1}})  \right) 
\otimes L_{\mf{s}}( \lam_r \varpi_1),& 
\end{align*}
with multiplicity 
$\mu=\mu(\lam^{(1)}_1,\cdots,\lam^{(1)}_{\ell_1}; \ldots ; 
\lam^{(r-1)}_1,\cdots,\lam^{(r-1)}_{\ell_{r-1}}
; \lam_r)$. 
When the multiplicity $\mu$ is one, we do not indicate it. 
As a rule, first terms give the highest weights corresponding to 
$\g^\natural$ in the same order as in the 
Column~\ref{col-IV}; 
the last number corresponds 
to the height weight of $\mf{s}\cong \mf{sl}_2$. 
We use the Bourbaki numbering for the simple roots of 
the simple factors of 
$\g^\natural$; 
in particular, for $G_2$, the first one is the short root. 

When $\g^\natural$ has a 
nontrivial center, we indicate in parenthesis 
the eigenvalue of $x \in \g_0^\natural \setminus\{0\}$ as fixed 
in Column \ref{col-V}
acting on the corresponding simple module.  \\

\item 
\label{col-VIII} gives the conformal weights 
of the strong generators of $\W^k(\g,f)$. 
The weights can be read off from the weights 
of $\mf{s}$-modules (it is $1+i/2$ if $i$ is the last value of 
each module). 
The multiplicity of each weight can be deduced 
from the decomposition of $\g$ 
into $\mf{s}$-modules (that we do not indicate 
here).
\\

\item 
\label{col-IX} gives the multiplicity 
of the representation $\C \times \mf{s}$ 
in the $\mf{k}$-module $\g$. 
This information is redundant 
with Column \ref{col-VII}; 
it always corresponds to the module 
$[0,\ldots,0;2]$ of the Column \ref{col-VII}. 
This just eases checking  Condition~\ref{crit-iii} 
in \Cref{criterion}.\\

\item 
\label{col-X} gives the possible values of $k+h^\vee$, 
where  $k$ is a solution of the equation 
$$c(k,f)=c(k^\natural)+r,$$ 
where $c(k,f)$ is the central charge of $\W^k(\g,f)$ 
as in Column \ref{col-VI}, 
$c(k^\natural)$
is the central charge of the affine vertex algebra 
$\bigotimes_{i=1}^s V^{k_i^\natural}(\g_i^\natural)$ 
corresponding to the semisimple part of $\g^\natural$, 
and $r$ is the dimension of the center of $\g^\natural$.

When $r>0$, we also specify the value of $k+h^\vee$ 
for $k$ such that $k_0^\natural=0$ 
provided  that the equality $c(k,f)=c(k^\natural)$ 
holds for such a value of $k$. 
In this case, it is indicated at the end,
after a semicolon.
Note that we write only the rational solutions  
(for instance, for $\g=E_6$ 
and $G.f=D_5(a_1)$, 
the equation $c(k,f)=1$ has purely complex 
solutions that we do not indicate.) 

The admissible values of $k$ appear in bold in the tables. 
\\

\item 
\label{col-XI} summarizes our conclusions as follows:
\begin{itemize}
\item we write $\sf{CL}$ 
when the level $k$ is collapsing, 
in particular if 
all conditions of \Cref{criterion} are 
satisfied, 
that is, when the multiplicity given in Column \ref{col-IX} is 
one, for each $i$ such that $h_i >1$, 
$h_{\lambda_i}$ is different from $h_i$ 
(the computations are made 
using the data of Columns \ref{col-VII} and \ref{col-VIII} 
for the values of $k+h^\vee$ as in Column~\ref{col-X}). 

\item if the multiplicity given in Column \ref{col-IX} is one, 
but the condition \ref{crit-iii} of 
\Cref{criterion} is not verified, 
then we have a conformal embedding $\tilde V_k(\g) \hookrightarrow \W_k(\g,f)$ 
and we write $\sf{CE}$. 
It may happen that $k$ is collapsing, $\sf{CE}$ 
indicates that it we do not know whether $L_k(\g)\cong \W_k(\g,f)$. 

\item if the multiplicity given in Column \ref{col-IX} is $>1$, 
then our criterion does not apply, 
and we write $\sf{DNA}$ (do not apply) if we are not able conclude 
(by other methods) that the 
embedding $\tilde V_k(\g)\hookrightarrow \W_k(\g,f)$ is conformal. 

\item for some cases, even if the conditions 
of \Cref{criterion} are not all verified, we can conclude 
using other arguments (for instance, if $k$ is admissible, 
or using explicit computations of OPEs, using \Cref{prop:ext}, etc.) 
that $k$ is collapsing or that $\W_k(\g,f)$ is a finite extension 
of $L_k(\g)$. 
In the first case, we write $\sf{CL}$ to indicate that $k$ is 
collapsing; in the second case, we write $\sf{FE}$ 
(in this case, we also mean that it is not collapsing). 
Both cases of course give conformal embeddings.  
For instance, for $\g=E_6$, and $G.f=A_2$, 
we know from \cite{AEM} that $\W_{-12+13/3}(E_6,A_2)$ 
is a finite extension of $L_{4/3-3}(A_2) \otimes L_{4/3-3}(A_2)$. 
For $\g=G_2$, and $G.f=G_2(a_1)$, 
we know that $\W_{-2}(G_2,f)\cong \C$ using OPEs, 
hence we write $\sf{CL}$ 
although the condition \ref{crit-iii} of \Cref{criterion} 
fails.

\item when we know by other arguments (for instance if 
$k$ is admissible) that 
$k$ cannot be collapsing, we write 
$\sf{CE|NCL}$ when the embedding is still conformal, 
and merely $\sf{NCL}$ when we do not know whether it 
is conformal.  
In general, this happens in the following situation: 
$\g^\natural$ has a nontrivial 
center, $k$ is admissible so that $\W_k(\g,f)$ is quasi-lisse  
and so it cannot be an extension   
of an Heisenberg vertex algebra. 
\item Line 1 of Tables 2-14 corresponds to the minimal orbit, in which case  there are always  three values of $k$ in column X.  Two of then  are collapsing and the remaining one gives a finite extension. These conclusions  are also proven in 
\cite[Theorem 3.3]{AKMPP18}, \cite[Theorem 6.8]{AKMPP17}.

\end{itemize}

\end{enumerate}

%
%

\newpage

\pagestyle{empty}
\begin{landscape}
{\tiny
\begin{tabular}{|c|c|c|c|p{1.5cm}|c|p{5.5cm}|p{3.3cm}|c|p{2.1cm}|p{2.4cm}|}
\hline 
I & 
II & 
III & 
IV & 
V & 
VI & 
VII & 
VIII & 
IX & 
X & 
XI \\[0.5em]
n$^\circ$ & 
$G.f$ & 
$\dim G.f$ & 
$\g^\natural$ & 
$k^\natural$  & 
$c(k,f)$ & 
Decompositions into $(\g^\natural\times \mf{s})-$modules & 
Conformal weights & 
mult. 
& 
$k+h^\vee$, 
\newline
$c(k,f)=c(k^\natural)$ & 
Conclusions
\\[-0.5em] 
&& && && && &&\\
\hline
&& && && && &&\\
20&
$E_6$ & 
72 & 
$0$ &
& 
$-\frac{6 (12 k+131) (13 k+144)}{k+12}$ & 
$[ 22 ], [ 16 ], [ 14 ], [ 10 ], [ 8 ], [ 2 ]$ & 
$2,5,6,8,9,12$ & 
1& 
$\boldsymbol{12/13},\boldsymbol{13/12}$ & 
$\sf{CL}$, $\sf{CL}$\\
\hline
19&
$E_6(a_1)$  &  
70 & 
$0$ &
& 
$-\frac{8 (7 k+75) (9 k+95)}{k+12}$ &
$[ 16 ], [ 14 ], [ 10 ]^{2}, [ 8 ], [ 6 ], [ 4 ], [ 2 ]$ & 
$2,3,4,5,6^2,8,9$ & 
1& 
$\bs{13/9}$, $9/7$ & 
$\sf{CL}$, $\sf{CL}$\\
\hline
18& 
$D_5$  & 
68 & 
$\C$ &  
$12 k+126$ & 
$-\frac{2 (180 k^2+3739 k+19416)}{k+12}$  & 
$[ 14] (0), [ 10] ( -3 ), [ 10] (0), [ 10] ( 3 ), [ 8] ( 0 ), [6] (0),$
\newline $ [4] ( -3 ), [4] (3), [2] (0), [ 0] ( 0 )$ &  
$1,2,3^2,4,5,6^3,8$ & 
1& 
$\bs{13/8}$, $8/5$ & 
$\sf{CE}$, $\sf{CL}$\\
\hline
17 & 
${E_6(a_3)}$ & 
66 &
$0$ & 
& 
$-\frac{36 (k+10) (6 k+59)}{k+12}$ & 
$[ 10 ]^{2}, [ 8 ]^{2}, [ 6 ]^{2}, [4 ]^{3}, [2]^{3}$ & 
$2^3,3^3,4^2,5^2,6^2$  & 
3 & 
$\bs{13/6}$, $2$ &
$\sf{CL}$, $\sf{DNA}$\\
\hline
\end{tabular} \\[0.25em] 
\captionof{table}{\footnotesize{Data for the $\mf{sl}_2$-triples 
and some collapsing levels in $E_6$}} 
\label{Tab:Data-E6}}

{\tiny
\begin{tabular}{|c|c|c|c|p{1.5cm}|c|p{5.5cm}|p{3.3cm}|c|p{2.1cm}|p{2.1cm}|}
\hline 
I & 
II & 
III & 
IV & 
V & 
VI & 
VII & 
VIII & 
IX & 
X & 
XI \\[0.5em]
n$^\circ$ & 
$G.f$ & 
$\dim G.f$ & 
$\g^\natural$ & 
$k^\natural$  & 
$c(k,f)$ & 
Decompositions into $(\g^\natural\times \mf{s})-$modules & 
Conformal weights & 
mult. 
& 
$k+h^\vee$, 
\newline
$c(k,f)=c(k^\natural)$ & 
Conclusions
\\[-0.5em] 
&& && && && &&\\
\hline
5& 
$A_2 A_1$ & 
46&
$\C A_2$ & 
$6k+45$ \newline  
$k+6$ & 
$-\frac{30 k^2+385 k+1236}{k+12}$ &   
$[ 0, 0; 4 ] (0), [ 1, 0; 3 ] (-1), [ 0, 1; 3 ] (1),[ 1, 0; 2 ] (2),$
\newline 
$ [ 0, 0; 2 ](0){2},[ 0, 0; 1 ] (3),  [ 0, 1; 1 ] (1), [ 1, 1; 0 ] (0), $  
\newline 
$[ 0, 1; 2 ] (-2), [ 1, 0;1 ] (-1),  [ 0, 0; 0 ] (0)$ &       
$1^9, (3/2)^ 8 ,  2^8 , (5/2)^ 6 ,  3$ & 
2&
$\varnothing$; $9/2$ & 
$\sf{DNA}$ \\
\hline
4& 
$A_2$  & 
$42$ & 
$A_2^2$  &  
$k+6$ 
\newline 
$k+6$  & 
$-\frac{12 (k+6) (2 k+11)}{k+12}$ & 
$[ 0, 0; 0, 0; 4 ], [ 1, 0; 1, 0; 2 ], [ 0, 0; 0, 0; 2 ],$          
$[ 1, 1; 0, 0; 0 ], [ 0, 1; 0, 1; 2 ], [ 0, 0; 1, 1; 0 ]$ & 
$ 1^{16}, 2^{19}, 3$ & 
1& 
$\bs{13/3}$, $9/2$, $6$ & 
$\sf{FE}$, $\sf{CL}$, 
$\sf{CL}$\\
\hline
3& 
$3 A_1$  & 
40& 
$A_1  A_2$ & 
$k+11/2$ 
\newline 
$2k+12$ &
$-\frac{18 k^2+169 k+372}{k+12}$ & 
$[ 1; 0, 0; 3 ], [ 0; 1, 1; 2 ], [ 0; 0, 0; 2 ],$
$[ 0; 1, 1; 0 ], [ 1; 1, 1; 1 ], [ 2; 0, 0; 0 ]$ &
$ 1^{11}, ( 3/2)^{16}, 2^9 , (5/2) ^2$ &
1& 
$\bs{13/2}$, $6$  & 
$\sf{CL}$, $\sf{CL}$\\
\hline
2& 
$2 A_1$  &  
32& 
$\C  B_3$ & 
$3 k+18$ 
\newline 
$k+4$ & 
$-\frac{2 ( 6 k^2+37 k+48 )}{k+12}$ & 
$[ 1,0, 0; 2 ] (0), [ 0, 0, 0; 2 ](0), [ 0, 0, 1; 1 ] (3/2),$  
\newline 
$[ 0, 0, 1; 1 ] (-3/2), [ 0, 1, 0 ; 0 ] (0),  [ 0, 0, 0; 0 ] (0)$ & 
$1^{22}, (3/2)^{16} , 2^8$ & 
1& 
$\bs{13/2}$, $9/2$, $8$; 
\newline
$6$ & 
$\sf{CE|NCL}$, $\sf{CE}$, 
$\sf{CE}$; 
\newline
$\sf{CL}$\\
\hline
\end{tabular} \\[0.25em] 
\captionof{table}{\footnotesize{Data for the $\mf{sl}_2$-triples 
and some collapsing levels in $E_6$ (continued)}} 
\label{Tab2:Data-E6}}

\bigskip

{\tiny
\begin{tabular}{|c|c|c|c|p{1.5cm}|c|p{5cm}|p{3.5cm}|c|p{3cm}|l|}
\hline 
I & 
II & 
III & 
IV & 
V & 
VI & 
VII & 
VIII & 
IX & 
X & 
XI \\[0.5em]
n$^\circ$ & 
$G.f$ & 
$\dim G.f$ & 
$\g^\natural$ & 
$k^\natural$  & 
$c(k,f)$ & 
Decompositions into $(\g^\natural\times \mf{s})-$modules & 
Conformal weights & 
mult. 
& 
$k+h^\vee$, 
$c(k,f)=c(k^\natural)$ & 
Conclusions
\\[-0.5em] 
&& && && && &&\\
\hline
4&
$G_2$  &  
12&
$0$ &  
& 
$-\frac{2 (7k+24 ) (12 k+41)}{4 + k}$& 
$[ 10 ], [ 2 ]$ &
$2,6$& 
1&
$\bs{4/7},$ $\bs{7/12}$&
$\sf{CL}$, $\sf{CL}$\\
\hline
3&
$G_2(a_1)$ & 
10&
 $0$ &   
&
$-\frac{4 ( k+2) (6 k+17)}{k+4}$ &
$[ 4 ], [ 2 ]^{3} $ & 
$2^3,3$ &
3&
$\bs{7/6}$, $2$ &
$\sf{CL}$, $\sf{CL}$ \\
\hline
2&
$\tilde{A}_1$ &  
8&
$A_1$ & 
$k+3/2$ & 
$-\frac{18 k^2+81 k +92}{k+4}$& 
$[ 3; 1], [ 2; 0 ], [ 0; 2 ] $ & 
$1^3,(3/2)^4,2$& 
1&
$\bs{7/6}$, $2/3$, $2$& 
$\sf{CL}$, $\sf{CL}$, $\sf{CL}$\\
\hline
1&
$A_1$  &  
6&
$A_1$ & 
$3 k+5$  & 
$-\frac{2 k (3k+5)}{k+4}$ &
$[ 0; 2 ], [ 3; 1 ], [ 2; 0 ]$ & 
$1^3, 2, (5/2)^2$ & 
1&
$\bs{7/3}$, $\bs{8/3}$, $\bs{5/2}$&
$\sf{CL}$, $\sf{CL}$, $\sf{FE}$\\
\hline
\end{tabular} \\[0.25em] 
\captionof{table}{\footnotesize{Data for the $\mf{sl}_2$-triples 
and some collapsing levels in $G_2$}} 
\label{Tab:Data-G2}}

{\tiny
\begin{tabular}{|c|c|c|c|p{1.3cm}|c|p{5cm}|p{3.7cm}|c|p{2.9cm}|p{1.8cm}|}
\hline 
I & 
II & 
III & 
IV & 
V & 
VI & 
VII & 
VIII & 
IX & 
X & 
XI \\[0.5em]
n$^\circ$ & 
$G.f$ & 
$\dim G.f$ & 
$\g^\natural$ & 
$k^\natural$  & 
$c(k,f)$ & 
Decompositions into $(\g^\natural\times \mf{s})-$modules & 
Conformal weights & 
mult. 
& 
$k+h^\vee$, 
$c(k,f)=c(k^\natural)$ & 
Conclusions
\\[-0.5em] 
&& && && && &&\\
\hline
15&
$F_4$ & 
48& 
 $0$ &  
 &
$-\frac{4 (13 k+108) (18 k+149 )}{k + 9}$ &
$[ 22 ], [ 14 ], [ 10 ], [ 2 ]$ & 
$2,6,8,12$ &
 1  &
$\bs{9/13}, \bs{13/18}$ &
$\sf{CL}$, $\sf{CL}$\\
\hline
14&
$F_4(a_1)$ &  
46& 
 $0$ & 
 & 
 $-\frac{6 (5 k + 39) (12 k+95)}{k + 9}$ &
 $[ 14 ], [ 10 ]^{2}, [ 6 ], [ 4 ], [ 2 ] $ & 
$2,  3, 4,  6^2, 8$ & 
 1&
$\bs{13/12}$, $6/5$ &
$\sf{CL}$, $\sf{CL}$ \\
\hline
13&
${F_4(a_2)}$ & 
44& 
 $0$  &  
  &
$-\frac{4 (54 k^2+3069 + 814 k)}{k + 9}$  & 
$[ 10 ]^{2}, [ 8 ], [ 6 ], [ 4 ], [ 2 ]^{3}$ & 
$ 2^ 3 , 3, 4, , 5, , 6^2$ &
3 &
$\varnothing$ &
\\
\hline
12&
${C_3}$ & 
42& 
$A_1$  & 
$k+6$ &
$-\frac{2 (105 k^2+5967 + 1582 k )}{k + 9}$ &
$[ 0; 10 ], [ 1; 9 ], [ 1; 3 ], [ 0; 6 ], [ 0; 2 ], [ 2; 0 ]$ &
$1^3 ,  2, ,( 5/2)^ 2, 4,  (11/2)^2 , 6$ &
1&  
$\bs{9/7}$, $\bs{13/10}$, $4/3$ &
$\sf{CL}$, $\sf{FE}$, $\sf{CL}$\\
\hline
11&
${B_3}$ & 
42& 
$A_1$ & 
$8k +60$  &
$-\frac{2 (84 k^2 + 1231 k + 4509)}{k + 9}$&
$[ 0; 10 ], [ 4; 6 ], [ 0; 2 ], [ 2; 0 ]$ & 
$1^ 3 ,  2,  4^5 ,  6$ & 
 1 &
$\bs{9/7}$, $\bs{13/8}$, $5/3$ &
$\sf{CL}$, $\sf{CL}$, $\sf{CL}$\\
\hline
\end{tabular} \\[0.25em] 
\captionof{table}{\footnotesize{Data for the $\mf{sl}_2$-triples 
and some collapsing levels in $F_4$}} 
\label{Tab:Data-F4}}

{\tiny
\begin{tabular}{|c|c|c|c|p{1.2cm}|c|p{5.2cm}|p{3.6cm}|c|p{2.5cm}|p{1.6cm}|}
\hline 
I & 
II & 
III & 
IV & 
V & 
VI & 
VII & 
VIII & 
IX & 
X & 
XI \\[0.5em]
n$^\circ$ & 
$G.f$ & 
$\dim G.f$ & 
$\g^\natural$ & 
$k^\natural$  & 
$c(k,f)$ & 
Decompositions into $(\g^\natural\times \mf{s})-$modules & 
Conformal weights & 
mult. 
& 
$k+h^\vee$, 
\newline 
$c(k,f)=c(k^\natural)$ & 
Conclusions
\\[-0.5em] 
&& && && && &&\\
\hline
44&
$E_7$ & 
126 & 
$0$ & 
&
$-\frac{7 (18 k+305) (19 k+324)}{k+18}$ & 
$[ 34 ], [ 26 ], [ 22 ], [ 18 ], [ 14 ], [ 10 ], [ 2 ]$ & 
$2,6,8, 10, 12, 14, 18$ &
1& 
$\bs{18/19}$, $\bs{19/18}$& 
$\sf{CL}$, $\sf{CL}$\\
\hline 
43& 
$E_7(a_1)$ &  
124& 
$0$ &  
& 
$-\frac{9 (11 k+184) (14 k+233)}{k+18}$ &
$[ 26 ], [ 22 ], [ 18 ], [ 16 ], [ 14 ], [ 10 ]^{2}, [ 6 ], [ 2 ]$ & 
$2, 4, 6^2 ,8, 9, 10, 12, 14$ &
1 &
$\bs{19/14}$, $14/11$&
$\sf{CL}$, $\sf{CL}$\\
\hline 
42& 
$E_7(a_2)$ &   
122& 
$0$ & 
& 
$-\frac{(9 k+148) (106 k+1737)}{k+18}$& 
$[ 22 ], [ 18 ], [ 16 ], [ 14 ]^{2}, [ 10 ]^{2}, [ 8 ], [ 6 ], [ 2 ] ^{2}$ & 
$ 2^2 ,4, ,5,  6^2 ,8^2 ,9, 10, 12$ &
2 & 
$\bs{171/106}$, $14/9$ & 
$\sf{NCL}$, $\sf{DNA}$\\
\hline 
41&
$E_7(a_3)$ &
120&    
$0$ & 
& 
$-\frac{666 k^2+21455 k+172800}{k+18}$ &
$[ 18 ], [ 16 ], [ 14 ]^{2}, [ 10 ]^{3}, [ 8 ], [ 6 ]^{2}, [ 4 ], 
[ 2 ]^{2}$ &
$2^2 ,3, 4^2 , 5, 6^3 ,8^2 , 9, 10$ &
2& 
$\varnothing$ & 
\\ 
\hline 
40&
$E_6$ &  
120&
$A_1$ &  
$3k+48$ & 
$-\frac{936 k^2+30707 k+251856}{k+18}$ &
$[ 0; 22 ], [ 2; 16 ], [ 0;14 ], [ 0; 10 ], [ 2; 8 ], [ 0; 2 ], [ 2; 0 ] $ & 
$1, 2, 5^ 3 ,  6, 8,  9^3 ,  12$ & 
1& 
$\bs{18/13}$, $\bs{19/12}$, $14/9$ &  
$\sf{CL}$, $\sf{CL}$, $\sf{CL}$\\
\hline 
\end{tabular} \\[0.25em] 
\captionof{table}{\footnotesize{Data for the $\mf{sl}_2$-triples 
and some collapsing levels in $E_7$}} 
\label{Tab:Data-E7}}

{\tiny
\begin{tabular}{|c|c|c|c|p{1.5cm}|c|p{5.5cm}|p{3.3cm}|c|p{2.1cm}|p{1.8cm}|}
\hline 
I & 
II & 
III & 
IV & 
V & 
VI & 
VII & 
VIII & 
IX & 
X & 
XI \\[0.5em]
n$^\circ$ & 
$G.f$ & 
$\dim G.f$ & 
$\g^\natural$ & 
$k^\natural$  & 
$c(k,f)$ & 
Decompositions into $(\g^\natural\times \mf{s})-$modules & 
Conformal weights & 
mult. 
& 
$k+h^\vee$, \newline 
$c(k,f)=c(k^\natural)$ & 
Conclusions
\\[-0.5em] 
&& && && && &&\\
\hline 
25& 
$D_5(a_1)$ &  
106& 
$\C A_1$ & 
$4 k+58$ \newline 
$k+13$ & 
$-\frac{9 \left(20 k^2+575 k+4136\right)}{k+18}$ &
$[ 0, 10 ] (0), 
[ 0, 8 ] (0), 
[ 1, 7 ] (-1), 
[ 0, 6 ] (0)^{2}, 
$
\newline 
$[ 0, 6 ] (-2),
[ 0, 6 ] (2),
[ 0, 4 ] (0),
[ 1, 7 ] (1),  
$
\newline 
$[ 1, 5 ] (1), 
[ 0, 2 ] (2), 
[ 0, 2 ] (0)^{2},
[ 1, 5 ] (-1),
$ 
\newline 
$[ 1, 1 ] (1),[ 2, 0 ] (0), 
[ 0, 2 ] (-2), 
[ 1, 1 ] (-1), 
[ 0, 0 ]$ & 
$1^ 4 , ( 3/2)^ 4 ,  2^ 4,  3, ( 7/2)^ 4 ,  4^ 4,$ 
\newline
$(9/2)^ 4 , 5,  6$ &
2&
$\varnothing$; $7/2$&
$\sf{DNA}$\\
\hline
24&
$A_4 A_1$ & 
104& 
$\C^2$ &  
$6 k+81$
\newline
$4 k+57$& 
$-\frac{2 \left(63 k^2+1718 k+11709\right)}{k+18}$ &
$[ 8] (0, 0), 
[ 7] (-1, 1 ), 
[ 7] ( 1, -1 ), 
[ 6] ( 2, 0), $
\newline
$[ 6] ( 0, 0), 
  [ 6] ( -2, 0 ),
  [ 5] ( -1, 1 ), 
[ 5] ( 1, -1),$ 
\newline
$[ 4] ( 0, 0),
[ 5] ( -1, -1), 
[ 5] ( 1, 1 ), 
[ 3] ( 1, 1 ),$
\newline
$ [ 3] ( -1, 1), 
 [ 4] ( 0, -2), 
 [ 4] ( 2, -2), 
[ 4] ( 0, 2 ),$
\newline 
$[ 4] ( -2, 2 ), 
 [ 2] ( 0, 0)^{2}, 
[ 3] ( -1, -1), 
[ 3] ( 1, -1),$ 
\newline
 $ [ 1] ( -1, 1 ), 
[ 1] ( 1, -1 ), 
[ 2] ( -2, 0), 
 [ 2] ( 2, 0),$ 
 \newline 
 $    [ 1] ( -3, 1), 
 [ 1] ( 3, -1 ), [ 0] ( 0, 0 )^{2}$ &
$1^ 2 , ( 3/2)^ 4 , 2^ 4,  (5/2)^ 4, 3^ 5,$
\newline 
$(7/2)^ 4 , 
 4^ 3 , 
  ( 9/2)^ 2 ,  5$& 
2& 
$\varnothing$&
$\sf{DNA}$\\
\hline
\end{tabular} \\[0.25em] 
\captionof{table}{\footnotesize{Data for the $\mf{sl}_2$-triples 
and some collapsing levels in $E_7$ (continued)}} 
\label{Tab:Data-E7-2}}

{\tiny
\begin{tabular}{|c|c|c|c|p{1.5cm}|c|p{6cm}|p{3.3cm}|c|p{2.1cm}|l|}
\hline 
I & 
II & 
III & 
IV & 
V & 
VI & 
VII & 
VIII & 
IX & 
X & 
XI \\[0.5em]
n$^\circ$ & 
$G.f$ & 
$\dim G.f$ & 
$\g^\natural$ & 
$k^\natural$  & 
$c(k,f)$ & 
Decompositions into $(\g^\natural\times \mf{s})-$modules & 
Conformal weights & 
mult. 
& 
$k+h^\vee$, 
\newline
$c(k,f)=c(k^\natural)$ & 
Conclusions
\\[-0.5em] 
&& && && && &&\\
\hline
15& 
$D_4(a_1)$& 
94&
 $A_1^3$ &  
 $k+12$ 
 \newline 
 $k+12$ 
 \newline 
 $k+12$ & 
 $-\frac{3 (k+12) (24 k+299)}{k+18}$&
 $[ 0;0;0; 6 ]^{2}, 
 [ 0; 1; 1; 4 ], 
 [ 0; 0; 0; 4 ], 
 [ 1; 0; 1; 4 ],$ 
 \newline 
 $ [ 1; 0; 1; 2 ], 
 [ 0; 1; 1; 2 ], [
  0; 0; 0; 2 ]^{3}, 
  [ 0; 0; 2; 0 ], $
  \newline
  $  [ 2; 0; 0; 0 ], 
  [ 1; 1; 0; 4 ], 
  [ 1; 1; 0; 2 ], 
  [ 0; 2; 0; 0 ]$& 
$1^ 9,  2^{ 15 },  3^{13 },  4^2$ &
$3$ &
$6$ &
$\sf{DNA}$\\
\hline
14& 
$(A_3 A_1)^\prime$ &
92&   
$A_1^3$ &  
$k+\frac{23}{2}$  
\newline 
$k+12$ 
\newline 
$2k+24$ &
$-\frac{2 \left(33 k^2+791 k+4743\right)}{k+18}$& 
$[ 0; 0; 0; 6 ], 
[ 1; 0; 0; 5 ], 
[ 0; 1; 1; 4 ], [ 1; 0; 0; 3 ],
[ 0; 0; 2; 4 ], $      
\newline
$[ 0; 0; 0; 2 ]^{2}, 
[ 1; 0; 2; 1 ], 
[ 0; 1; 1; 2 ], [ 0;0; 2; 0 ], 
 $    
\newline
$ [ 0; 2; 0; 0 ], 
[ 2; 0; 0; 0 ], 
[ 1; 1; 1; 3 ]$& 
$ 1^ 9, ( 3/2)^ 6 , 2^ 6, 
(5/2)^{ 10},  3^ 7,$ 
\newline 
$(7/2)^ 2 , 4$ &
2& 
$14/3$, $6$&
$\sf{DNA}$, $\sf{DNA}$\\
\hline 
13& 
$(2 A_2)A_1$ & 
90& 
$A_1^2$ & 
$3k+36$ 
\newline 
$3k+35$ & 
$-\frac{54 k^2+1223 k+6912}{k+18}$&
$[ 0;1; 5 ], 
[ 2; 0; 4 ], 
[ 0; 2;4 ], 
[ 2; 1; 3 ], 
[ 0;1;3 ],[ 2;2;2 ],$ 
\newline
$
[ 0; 0; 2 ]^{2}, 
[ 2; 1; 1 ], 
[ 0; 3; 1 ], 
[ 2; 0; 0 ],[ 0; 2; 0 ]$&
$1^ 6, (3/2)^{ 10 }, 2^{ 11}, 
(5/2)^ 8, 3^6 ,$
\newline  
$(7/2)^ 2$ &
2&
$\bf{19/3}$&
$\sf{CL}$ \\
\hline
\end{tabular} \\[0.25em] 
\captionof{table}{\footnotesize{Data for the $\mf{sl}_2$-triples 
and some collapsing levels in $E_7$ (continued)}} 
\label{Tab:Data-E7-3}}

{\tiny
\begin{tabular}{|c|c|c|c|p{1.2cm}|c|p{5.5cm}|p{4cm}|c|p{2.1cm}|p{1.4cm}|}
\hline 
I & 
II & 
III & 
IV & 
V & 
VI & 
VII & 
VIII & 
IX & 
X & 
XI \\[0.5em]
n$^\circ$ & 
$G.f$ & 
$\dim G.f$ & 
$\g^\natural$ & 
$k^\natural$  & 
$c(k,f)$ & 
Decompositions into $(\g^\natural\times \mf{s})-$modules & 
Conformal weights & 
mult. 
& 
$k+h^\vee$, 
\newline 
$c(k,f)=c(k^\natural)$ & 
Conclusions
\\[-0.5em] 
&& && && && &&\\
\hline
69&
$E_8$  & 
240& 
$0$ & 
&
$-\frac{8 \left((30 k+869\right) \left((31 k+900\right)}{k+30}$&  
$[ 58 ], [ 46 ], [ 38 ], [ 34 ], [ 26 ], [ 22 ], [ 14 ], [ 2 ] $ & 
$ 2,  8, 12, 14,  18,  20,  24,  30$ & 
1&
$\bf{30/31}$, 
$\bf{31/30}$
& 
$\sf{CL}$, 
$\sf{CL}$
\\
\hline
68& 
$E_8(a_1)$ &
238 & 
$0$ &  
& 
$-\frac{10 \left((19 k+546\right) \left((24 k+689\right)}{k+30}$&
$[ 46 ], [ 38 ], [ 34 ], [ 28 ], [ 26 ], [ 22 ], [ 18 ], [ 14 ], [ 10 ], [ 2 ]$&
$2,  6,  8, 10, 12,  14, 15, 18,  20,  24$ & 
1&
$\bf{31/24}$, 
${24/19}$& 
$\sf{CL}$, $\sf{CL}$ \\
\hline
67&
$E_8(a_2)$ &  
236 & 
$0$ &  
 & 
$-\frac{12 \left((13 k+370\right) \left((20 k+569\right)}{k+30}$&
$[ 38 ], [ 34 ], [ 28 ], [ 26 ], [ 22 ]^{2}, [ 18 ], [ 16 ], [ 14 ], [ 10 ],[ 6 ],$
\newline 
$ [ 2 ]$ & 
$2,  4, 6, 8,  9,10, 12^ 2 ,
14,  15, 18, 20$ & 
1&
$\bf{31/20}$, 
20/13& 
$\sf{CL}$, $\sf{CL}$ \\
\hline
66& 
$E_8(a_3)$  & 
234 &
 $0$ & 
 &
$-\frac{2 \left(1200 k^2+67781 k+957150\right)}{k+30}$&
$[ 34 ], [ 28 ], [ 26 ]^{2}, [ 22 ], [ 18 ]^{2}, [ 16 ], [ 14 ], 
[ 10 ]^{2},[ 8 ],$
\newline 
$ [ 2 ]^{2}$ & 
$2^ 2, 5,  6^ 2, 8,  9, 10^2 ,12, ,14^ 2 , 
 15,18$ &          
2&
$\varnothing$&
\\
\hline
65&
$E_8(a_4)$ & 
232 &  
$0$ &   
& 
$-\frac{16 \left((7 k+195\right) \left((15 k+419\right)}{k+30}$&
$[ 28 ], [ 26 ], [ 22 ]^{2}, 
[ 18 ]^{2}, [ 16 ], [ 14 ]^{3}, 
[ 10 ]^{2}, [ 8 ],$
\newline 
$[ 6 ], 
          [ 4 ], [ 2 ]$ & 
$2, 3, 4,  5, 6^ 2 ,  8^ 3 , 9,  
 10^ 2 ,  12^ 2 , 14,$ 
 \newline  
 $15$ &          
1&
$\bf{31/15}$, 15/7&
$\sf{CL}$, $\sf{CL}$\\
\hline
\end{tabular} \\[0.25em] 
\captionof{table}{\footnotesize{Data for the $\mf{sl}_2$-triples 
and some collapsing levels in $E_8$}} 
\label{Tab:Data-E8}}

{\tiny
\begin{tabular}{|c|c|c|c|p{1.2cm}|c|p{5.75cm}|p{4cm}|c|p{2.1cm}|p{1.4cm}|}
\hline 
I & 
II & 
III & 
IV & 
V & 
VI & 
VII & 
VIII & 
IX & 
X & 
XI \\[0.5em]
n$^\circ$ & 
$G.f$ & 
$\dim G.f$ & 
$\g^\natural$ & 
$k^\natural$  & 
$c(k,f)$ & 
Decompositions into $(\g^\natural\times \mf{s})-$modules & 
Conformal weights & 
mult. 
& 
$k+h^\vee$, 
\newline 
$c(k,f)=c(k^\natural)$ & 
Conclusions
\\[-0.5em] 
&& && && && &&\\
\hline
53&
$E_7(a_3)$ & 
220& 
$A_1$ & 
$k+25$ & 
$-\frac{2 \left(333 k^2+17761 k+236850\right)}{k+30}$& 
$[ 0; 18 ], [ 0; 16 ], [ 1; 15 ], [ 0; 14 ]^{2}, [ 0; 10 ]^{3}, 
[ 1; 11 ],$ 
\newline 
$ [ 1; 9 ]^{2}, [ 0; 8 ], [ 0; 6 ]^{2}, 
          [ 1; 5 ], [ 0; 4 ], [ 0; 2 ]^{2}, [ 2; 0 ],$
          \newline $[ 1; 1 ]$ &
$1^ 3 , ( 3/2)^ 2 ,  2^ 2 ,  3, ( 7/2)^2 ,  4^ 2 , 
 5,$
 \newline 
 $(11/2)^4 , 6^ 3 , ( 13/2)^ 2 ,  8^ 2 ,( 17/2)^ 2 ,$
 \newline
$   9, 10$ & 
2&  
$\varnothing$&
\\
\hline
52&
$E_6(a_1)A_1$ &
218&  
$\C$ & 
$6 k+153$ & 
$-\frac{15 \left(34 k^2+1781 k+23326\right)}{k+30}$& 
$[ 16] (0), [ 14] ( 0), [ 13] ( 1 ), [ 13] ( -1), 
[ 12] ( 2 ),$ 
\newline 
$ [ 11] ( 1),[ 11] ( -1), [ 10] ( 0)^{ 2}, 
[ 8] ( 0 ),[ 9] ( 1),$ 
\newline 
$ [ 9] ( -1 ), [ 8] ( -2), [ 7] ( 1 ), [ 7] ( -1),[ 8] ( 2 ), [ 6] ( 0 ),$ 
\newline 
$ [ 5] ( 1 ), [ 5] ( -1), [ 4] ( 2),[ 4] ( 0 ), 
[ 4] ( -2 ), [ 3] ( 1 ),$ 
\newline 
$ [ 3] ( -1 ), [ 2] ( 0)^{2},[ 0] ( 0), 
      [ 1] ( -3 ), [ 1] ( 3) $ & 
$ 1, 1(3/2)^ 2 ,2^ 2 , ( 5/2)^2 ,  3^ 3 ,$
\newline 
$( 7/2)^ 2 ,  4, (9/2)^ 2 ,  5^ 3 , 
( 11/2)^ 2 ,  6^ 2 ,$
\newline 
$( 13/2)^ 2 ,  7^ 2 , 
  ( 15/2)^ 2 ,  8, 9$ &    
2&
$\varnothing$ &
\\
\hline
\end{tabular} \\[0.25em] 
\captionof{table}{\footnotesize{Data for the $\mf{sl}_2$-triples 
and some collapsing levels in $E_8$ (continued)}} 
\label{Tab:Data-E8-2}}

{\tiny
\begin{tabular}{|c|c|c|c|p{1.2cm}|c|p{5.85cm}|p{4cm}|c|p{2.1cm}|p{1.4cm}|}
\hline 
I & 
II & 
III & 
IV & 
V & 
VI & 
VII & 
VIII & 
IX & 
X & 
XI \\[0.5em]
n$^\circ$ & 
$G.f$ & 
$\dim G.f$ & 
$\g^\natural$ & 
$k^\natural$  & 
$c(k,f)$ & 
Decompositions into $(\g^\natural\times \mf{s})-$modules & 
Conformal weights & 
mult. 
& 
$k+h^\vee$, 
\newline 
$c(k,f)=c(k^\natural)$ & 
Conclusions
\\[-0.5em] 
&& && && && &&\\
\hline
44& 
$A_6 A_1$ & 
212&  
$A_1$ & 
$7k+180$ & 
$-\frac{6 \left(57 k^2+2888 k+36580\right)}{k+30}$& 
$[ 2; 12 ], [ 1; 11 ], [ 0; 10 ], [ 1; 9 ], [ 2; 8 ], [ 3; 7 ], 
          [ 4; 6 ], $
\newline 
$[ 3; 5 ], [ 2; 4 ], [ 1; 3 ], [ 0; 2 ]^{\oplus 2}, [ 1; 1 ], 
          [ 2; 0 ]$ &       
$1^ 3 , ( 3/2)^ 2 ,  2^ 2, ( 5/2)^ 2 ,  3^ 3, 
( 7/2)^ 4 ,$
\newline 
$ 4^ 5 , 
  ( 9/2)^ 4 ,  5^ 3 , ( 11/2)^ 2 ,  6, ( 13/2)^ 2 , $
  \newline 
  $7^ 3$ & 
2& 
$14/3$&
$\sf{DNA}$\\
\hline
43&
$D_6(a_1)$ & 
210& 
$A_1^2$ &
 $k+48$ 
 \newline 
 $k+48$ &
$-\frac{2 \left(186 k^2+9509 k+121590\right)}{k+30}$ & 
$[ 0; 0; 14 ], [ 0; 1; 11 ], [ 1; 0; 11 ], [ 0; 0; 10 ]^{2},[ 0; 1; 9 ],$
\newline 
$ [ 1; 0; 9 ], [ 0; 0; 8 ],[ 0; 0; 6 ]^{2},[ 1; 1; 8 ], [ 0; 1; 5 ], [ 1; 0; 5 ], $ 
\newline 
$ [ 0; 1; 3 ], [ 1; 0; 3 ], [ 0; 0; 2 ]^{2}, [ 1; 1; 2 ],
[ 2; 0; 0 ], [ 0; 2; 0 ] $&
$1^6 ,  2^ 6 , ( 5/2)^ 4 , ( 7/2)^ 4 , 
 4^ 2 ,  5^ 5 ,$
 \newline  
$  ( 11/2)^ 4 ,  6^ 2 , ( 13/2)^ 4 ,  8$ &           
2&
$\varnothing$&
\\
\hline
42& 
$A_6$ &  
210& 
$A_1^2 $ & 
$k+24$ 
\newline 
$7k+180$&
$-\frac{2 \left(168 k^2+8501 k+107550\right)}{k+30}$ &
$[ 2; 0;12 ], [ 1; 1; 10 ], [ 0; 0; 10 ], [ 2; 0; 8 ], [ 4; 0; 6 ],$ 
\newline 
$ [ 2; 0; 4 ], [ 3; 1; 6 ], [ 1; 1; 2 ], [ 0; 0; 2 ], [ 2; 0; 0 ],[ 0; 2; 0 ]$& 
$1^6 ,  2^ 5 ,  3^ 3 , 4^{13 },  5^ 3 , 
 6^ 5 ,  7^ 3$ & 
1&
$\bf{30/7}$, 
$\bf{31/7}$, 
$14/3$&
$\sf{CL}$, 
$\sf{FE}$, 
$\sf{FE}$ 
\\
\hline
\end{tabular} \\[0.25em] 
\captionof{table}{\footnotesize{Data for the $\mf{sl}_2$-triples 
and some collapsing levels in $E_8$ (continued)}} 
\label{Tab:Data-E8-3}}

{\tiny
\begin{tabular}{|c|c|c|c|p{1.2cm}|c|p{5.75cm}|p{4cm}|c|p{2.1cm}|p{1.4cm}|}
\hline 
I & 
II & 
III & 
IV & 
V & 
VI & 
VII & 
VIII & 
IX & 
X & 
XI \\[0.5em]
n$^\circ$ & 
$G.f$ & 
$\dim G.f$ & 
$\g^\natural$ & 
$k^\natural$  & 
$c(k,f)$ & 
Decompositions into $(\g^\natural\times \mf{s})-$modules & 
Conformal weights & 
mult. 
& 
$k+h^\vee$, 
\newline 
$c(k,f)=c(k^\natural)$ & 
Conclusions
\\[-0.5em] 
&& && && && &&\\
\hline
31& 
$D_4 A_2$ &
198&  
$A_2$ & 
$2k+46$ &
$-\frac{8 \left(39 k^2+1960 k+24630\right)}{k+30}$& 
$[ 0, 0; 10 ], [ 1, 0; 8 ], [ 0, 1; 8 ], [ 1, 0; 6 ], [ 1, 1; 6 ], 
[ 0, 1; 6 ],$  
\newline 
$  [ 1, 0; 4 ],[ 0, 0; 4 ], [ 0, 1; 4 ], [ 0, 0; 2 ]^{2},[ 0, 2; 2 ], [ 2, 0; 2 ],$
\newline
 $  [ 1, 1; 0 ]$&
$1^8 ,  2^{ 14 },  3^ 7 ,  4^{ 14 },  5^ 6 ,  6$ &
2&
$\varnothing$&
\\
 \hline
30& 
$A_4 A_2 A_1$ &  
196& 
$A_1$ & 
$15k+350$ & 
$-\frac{2 (3 k+70) (25 k+564)}{k+30}$& 
$[ 2; 8 ], [ 3; 7 ], [ 4; 6 ], [ 3; 5 ], [ 1; 5 ], [ 6; 4 ], [ 2; 4 ],$ 
\newline 
$[ 5;3 ], [ 1;3 ], [ 4; 2 ], [ 0; 2 ]^{2}, [ 5; 1 ], [ 2; 0 ]$ &
$1^ 3 , ( 3/2)^ 6 ,  2^ 7 , 
( 5/2)^ 8 ,  3^{ 10 }, ( 7/2)^ 6 ,$ 
\newline
$   4^ 5 , ( 9/2)^ 4 ,  5^ 3$ &
2& 
$\bf{168/25}$, 
$\bf{217/30}$, 
$20/3$&
$\sf{NCL}$, 
$\sf{NCL}$, 
$\sf{DNA}$
\\
\hline
29& 
$D_5(a_1) A_1$ &
196& 
 $A_1^2$ & 
 $k+22$ 
 \newline 
$8k+184$  & 
$-\frac{2 \left(93 k^2+4414 k+52440\right)}{k+30}$& 
$[ 0; 0; 10 ], [ 0; 2; 8 ], [ 1; 2; 7 ], [ 0; 4; 6 ], [ 0; 2; 6 ], 
[ 0; 2; 4 ],$ 
\newline 
$  [ 1; 2; 5 ],[ 0; 4; 2 ], [ 0; 0; 2 ]^{2}, [ 0; 2; 0 ],
[ 2; 0; 0 ],[ 1; 0; 3 ],$
\newline 
$  [ 1; 4; 1 ]$&
$1^ 6 , ( 3/2)^{ 10 },  2^ 7 , ( 5/2)^ 2 , 
 3^ 3 , ( 7/2)^ 6 , $
 \newline 
$   4^ 8 , ( 9/2)^ 6 ,  5^ 3 ,  6$ &
2&
$\varnothing$&
\\
\hline
\end{tabular} \\[0.25em] 
\captionof{table}{\footnotesize{Data for the $\mf{sl}_2$-triples 
and some collapsing levels in $E_8$ (continued)}} 
\label{Tab:Data-E8-4}}

{\tiny
\begin{tabular}{|c|c|c|c|p{1.2cm}|c|p{5.9cm}|p{3.6cm}|c|p{2.1cm}|p{1.7cm}|}
\hline 
I & 
II & 
III & 
IV & 
V & 
VI & 
VII & 
VIII & 
IX & 
X & 
XI \\[0.5em]
n$^\circ$ & 
$G.f$ & 
$\dim G.f$ & 
$\g^\natural$ & 
$k^\natural$  & 
$c(k,f)$ & 
Decompositions into $(\g^\natural\times \mf{s})-$modules & 
Conformal weights & 
mult. 
& 
$k+h^\vee$, 
\newline 
$c(k,f)=c(k^\natural)$ & 
Conclusions
\\[-0.5em] 
&& && && && &&\\
\hline
22&
$D_4(a_1) A_2$ & 
184& 
$A_2$ & 
$6k+ 132$ & 
$-\frac{16 \left(  6 k^2+254 k +2685\right)}{k + 30}$& 
$[ 1, 1; 6 ], [ 3, 0; 4 ], [ 0, 3; 4 ], [ 2, 2; 2 ],[ 0, 0;2 ], [ 1, 1; 0 ] $ &
$1^8 ,  2^{ 28 },  3^{ 20 },  4^ 8 $ & 
 1&
$\bf{31/4}$, 
$25/3$, 
$9$& 
$\sf{CL}$, 
$\sf{CL}$, 
$\sf{CL}$
\\
\hline
21 
&
$D_4A_1$ &
184& 
$C_3$ & 
$k+\frac{41}{2}$ & 
$-\frac{174 k^2+ 8219 k +97410}{k + 30}$& 
$[ 0, 0, 0; 10 ], [ 1, 0, 0; 7 ], [ 0, 1, 0; 6 ], [ 1, 0, 0; 5 ],$ 
\newline 
$ [ 0, 0, 0; 2 ]^{2}, [ 0, 0, 1; 1 ], [ 2, 0, 0; 0 ]$& 
$1^{ 21 }, ( 3/2)^{ 14 },  2^ 2 , 
( 7/2)^ 6 ,  4^{ 14}, $
\newline 
$( 9/2)^ 6 , 
   6$ & 
2&
$\varnothing$&
\\ 
\hline
20&
$A_3 A_2 A_1$ & 
182 & 
$A_1^2$ & 
$k+\frac{41}{2}$ 
\newline 
$24k+528$&
$-\frac{3 (30 k^2+1253 k+13070)}{k + 30}$& 
$[ 0; 4; 6 ], [ 1; 2; 5 ], [ 0; 6; 4 ], [ 0; 2;4 ], [ 1; 6; 3 ], 
 [ 0; 8; 2 ],$ 
 \newline 
 $ [ 0; 4; 2 ], [ 0; 0; 2 ],[ 1; 4; 1 ], [ 0; 2; 0 ],  [ 2;0; 0 ]$ & 
$1^ 6 , ( 3/2)^ {10},  2^{ 15 }, ( 5/2)^{ 14 }, 
 3^{ 10 }, $
 \newline 
 $( 7/2)^ 6 ,  4^ 5$ & 
1&
$\bf{31/4}$,
$25/3$,\newline   
$19/2$, 
$8$ & 
$\sf{CE}$, 
$\sf{FE}$, 
\newline 
$\sf{CL}$,
$\sf{CE}$
\\
\hline
\end{tabular} \\[0.25em] 
\captionof{table}{\footnotesize{Data for the $\mf{sl}_2$-triples 
and some collapsing levels in $E_8$ (continued)}} 
\label{Tab:Data-E8-5}}

{\tiny
\begin{tabular}{|c|c|c|c|p{1.2cm}|c|p{5.7cm}|p{3.6cm}|c|p{2.1cm}|p{1.7cm}|}
\hline 
I & 
II & 
III & 
IV & 
V & 
VI & 
VII & 
VIII & 
IX & 
X & 
XI \\[0.5em]
n$^\circ$ & 
$G.f$ & 
$\dim G.f$ & 
$\g^\natural$ & 
$k^\natural$  & 
$c(k,f)$ & 
Decompositions into $(\g^\natural\times \mf{s})-$modules & 
Conformal weights & 
mult. 
& 
$k+h^\vee$, 
\newline 
$c(k,f)=c(k^\natural)$ & 
Conclusions
\\[-0.5em] 
&& && && && &&\\
\hline
12& 
$A_3 A_1$ & 
164&  
$A_1  B_3$ & 
$k+\frac{39}{2}$ 
\newline 
$k+18$ &
$-\frac{3 (22 k^2+857 k+8390 )}{k + 30}$ & 
$[0; 0, 0, 0;  6 ], [ 1;0, 0, 0;  5 ], [ 0; 1, 0, 0; 4 ], $ 
         \newline 
$   [ 0;0, 0, 1;  4 ], [ 1;0, 0, 0 ; 3 ], [ 0;0, 0, 0;  2 ]^{2}, $
    \newline 
$    [ 1;1, 0, 0; 1; 1 ], [ 2;0, 0, 0;  0 ], [ 0;0, 0, 1;  2 ],$
\newline    
$ [ 1;0, 0, 1;  3 ], [ 0;  0, 1, 0; 0 ]$& 
$1^{ 24}, ( 3/2)^{ 14 },  2^{ 10}, 
( 5/2)^{ 18},  3^{ 15},$
\newline 
$( 7/2)^ 2 , 4,$ & 
 2&
$8$, $10$&
$\sf{DNA}$, $\sf{DNA}$\\ 
\hline
11& 
$(2 A_2) A_1$ &  
162 & 
$A_1  G_2$  & 
$3k +59$
\newline 
$k+8$ & 
$-\frac{2 (27 k^2+ 997 k + 9330 )}{k + 30}$ & 
$[ 1; 0, 0; 5 ], [ 2; 0, 0; 4 ], [ 0; 1, 0; 4 ], [ 1; 1, 0; 3 ],$ 
\newline 
$ [ 1; 0, 0; 3 ], [ 2; 1, 0; 2 ], [ 0; 0, 0; 2 ]^{2}, [ 3; 0, 0; 1 ],$  \newline 
$ [ 1; 1, 0; 1 ], [ 2; 0, 0; 0 ], [ 0; 0, 1; 0 ]$ &
$1^{17}, ( 3/2)^{ 18},  2^{ 23}, 
( 5/2)^{ 16},  3^{ 10} ,$
\newline 
$ ( 7/2)^ 2$ &
 2&
$\varnothing$&
\\ 
\hline
10&
$2 A_2$ &  
156& 
$G_2^2$ &
$k+18$ 
\newline 
$k+18$ &
$-\frac{4 (k+18) (12k+205 )}{k + 30}$ &  
$[ 1, 0; 0, 0; 4 ], [ 0, 0; 1, 0; 4 ], [ 1, 0; 1, 0; 2 ],$ 
\newline
$[ 0, 0; 0, 1; 0 ], [ 0, 1; 0, 0; 0 ], [ 0, 0; 0, 0; 2 ]$&
$1^{ 28},  2^{ 50},  3^{ 14} $& 
1&
$\bf{31/3}$, 
$12$, $10$&
$\sf{CL}$, 
$\sf{CL}$, 
$\sf{CE}$ \\
\hline
\end{tabular} \\[0.25em] 
\captionof{table}{\footnotesize{Data for the $\mf{sl}_2$-triples 
and some collapsing levels in $E_8$ (continued)}} 
\label{Tab:Data-E8-6}}
\end{landscape}

\end{document}